\documentclass{scrartcl}
\pdfoutput=1

\usepackage[utf8]{inputenc}
\usepackage{bbm}
\usepackage{todonotes}
\usepackage{lipsum}
\usepackage{url}
\usepackage{amsmath}
\usepackage{amssymb}
\usepackage{tikz-cd}
\usepackage{amsthm}
\usepackage{mathtools}
\usepackage{MnSymbol}
\usepackage{stmaryrd}
\usepackage{enumitem}
\usepackage{hyperref}

\newtheorem{Def}{Definition}
\newtheorem{Thm}{Theorem}
\newtheorem{Lem}{Lemma}
\newtheorem{Cor}{Corollary}
\newtheorem{Prop}{Proposition}
\newtheorem{Rem}{Remark}
\newtheorem{Ex}{Example}

\newtheorem*{Def*}{Definition}
\newtheorem*{Thm*}{Theorem}
\newtheorem*{Lem*}{Lemma}
\newtheorem*{Cor*}{Corollary}
\newtheorem*{Prop*}{Proposition}
\newtheorem*{Rem*}{Remark}
\newtheorem*{Ex*}{Example}
\newtheorem*{Con*}{Conjecture}
\newtheorem*{Aux*}{Auxiliary lemma}
\newcommand{\mH}{\mathcal{H}}

\newcommand{\pa}{\partial}
\newcommand{\actdiski}{\mathcal{A}^{D^{z_0}_R}_{\mH; 1}}
\newcommand{\actdiskii}{\mathcal{A}^{D^{z_0}_R}_{\mH; 2}}
\newcommand{\acti}{\mathcal{A}^{D}_{\mH; 1}}
\newcommand{\actii}{\mathcal{A}^{D}_{\mH; 2}}
\newcommand{\R}{\mathbb{R}}
\newcommand{\Cx}{\mathbb{C}}
\newcommand{\can}{\text{can}}
\newcommand{\nab}[1]{\prescript{#1}{}{\nabla}}
\newcommand{\tor}[1]{\prescript{#1}{}{T}}
\newcommand{\hvec}[1]{\Gamma (T^{(1,0)} #1)}

\font\afont=cmssbx10 scaled \magstep5     % for the title
\font\dfont=cmssbx10 scaled \magstep2     % for section headings and author name
\font\efont=cmssbx10 scaled \magstephalf
\hypersetup{
	unicode = true,
	pdfborder = {0 0 0},
	linktoc = page,
	colorlinks,
	linkcolor = black,
	citecolor = black,
	filecolor = black,
	urlcolor = blue
}

\begin{document}

%\bibliographystyle{alpha}

%_______________________________________________________________________________
\begin{titlepage}
  \vspace*{6mm}
  \begin{center}
     {\afont Pseudo-Holomorphic Hamiltonian\\}\vspace{0.5cm}
     {\afont Systems}
     \\[1cm]
     {\large by}
     \\[1cm]
     {\dfont Luiz Frederic Wagner}
     \\[1cm]
     {\dfont March 16, 2023}
     \\[3cm]
     {\efont Abstract}
     \\[0.5cm]
  \end{center}
  In this paper, we first explore holomorphic Hamiltonian systems. In particular, we define action functionals for those systems and show that holomorphic trajectories obey an action principle, i.e., that they can be understood -- in some sense -- as critical points of these action functionals. As an application, we use holomorphic Hamiltonian systems to establish a relation between Lefschetz fibrations and almost toric fibrations. During the investigation of action functionals for holomorphic Hamiltonian systems, we observe that the complex structure $J$ corresponding to a holomorphic Hamiltonian system poses strong restrictions on the existence of certain trajectories. For instance, no non-trivial holomorphic trajectories with complex tori as domains can exist in $\mathbb{C}^{2n}$ due to the maximum principle. To lift this restriction, we generalize the notion of holomorphic Hamiltonian systems to systems with non-integrable almost complex structures $J$ leading us to the definition of pseudo-holomorphic Hamiltonian systems. We show that these systems exhibit properties very similar to their holomorphic counterparts, notably, that they are also subject to an action principle. Furthermore, we prove that the integrability of $J$ is equivalent to the closedness of the ``pseudo-holomorphic symplectic'' form. Lastly, we show that, aside from dimension four, the set of proper pseudo-holomorphic Hamiltonian systems is open and dense in the set of pseudo-holomorphic Hamiltonian systems by considering deformations of holomorphic Hamiltonian systems. This implies that proper pseudo-holomorphic Hamiltonian systems are generic.

\end{titlepage}

%_______________________________________________________________________________

\newpage
\pagenumbering{roman}
\tableofcontents
\newpage
\pagenumbering{arabic}

\section*{Acknowledgments}
\addcontentsline{toc}{section}{Acknowledgments}
\markboth{}{Acknowledgments}

First and foremost I thank my supervisor Kai Cieliebak for his guidance, helpful feedback, and insightful discussions. Moreover, I thank Lei Zhao who convinced me to write this paper and the remaining members of the symplectic group in Augsburg, in particular Urs Frauenfelder, for their input and company. Lastly, I thank my family and friends who supported me throughout my academic career, in particular Christian P.M. Schneider, Alexander Segner, and Florian Stuhlmann.

\newpage
\section{Introduction}
\label{sec:intro}

Hamiltonian systems (HSs) as the mathematical model for classical mechanics have been central to the advance of modern physics and mathematics alike. In physics, HSs provide a theoretical foundation for several approaches to quantization. In mathematics, the interest in HSs has led to the study of symplectic geometry and topology. The methods developed in this study, e.g. Floer theory, have proven to be of great success for various branches of mathematics and physics alike, for instance celestial mechanics and string theory.\\
Simply put, a HS consists of three data: a smooth manifold $M$, a symplectic $2$-form $\omega$ on $M$, and smooth function $H\in C^\infty (M,\mathbb{R})$. In physical terms, the symplectic manifold $(M,\omega)$ can be understood as the phase space of the system, while the function $H$, often called Hamilton function or, simply, Hamiltonian, assigns to every point in phase space its energy. These data allow us to define the Hamiltonian vector field $X_H$ on $M$ via the equation $\iota_{X_H}\omega = -dH$. The dynamics of the HS $(M,\omega, H)$ is governed by the vector field $X_H$. Precisely speaking, the physical trajectories of point-like particles described by the HS $(M,\omega, H)$ are exactly the integral curves of $X_H$. The connection between the integral curve equation of $X_H$ and the Hamilton equations known from classical mechanics is given by Darboux's theorem: every symplectic form $\omega$ can locally be written as
\begin{gather*}
 \omega = \sum^n_{i = 1} dp_i\wedge dq_i.
\end{gather*}
In such Darboux charts, the integrable curve equation of $X_H$ reduces to the Hamilton equations:
\begin{gather*}
 \dot q_i(t) = \frac{\partial H}{\partial p_i};\quad \dot p_i(t) = -\frac{\partial H}{\partial q_i}\quad\forall t\in I\ \forall i\in\{1,\ldots, m\}.
\end{gather*}
Since the integral curve equation is just a first-order differential equation, there always exists an open interval $I$ and a trajectory $\gamma:I\to M$ with $\gamma (t_0) = x$ for any initial value $x\in M$ and $t_0\in\mathbb{R}$. Furthermore, two trajectories $\gamma_1:I_1\to M$ and $\gamma_2:I_2\to M$ are identical iff they have the same domain ($I_1 = I_2\equiv I$) and attain the same value at some point $t_0\in I$. In particular, maximal trajectories for a given initial value are unique.\\
Physical trajectories also obey the action principle, i.e., they can be obtained as ``critical points'' of the action functional $\mathcal{A}_H:C^\infty(I,M)\to\mathbb{R}$ assigned to an exact\linebreak HS $(M,\omega = d\lambda, H)$:
\begin{gather*}
 \mathcal{A}_H[\gamma]\equiv \mathcal{A}^\lambda_H[\gamma]\coloneqq \int\limits_I \gamma^\ast\lambda - \int\limits_I H\circ\gamma (t)\, dt.
\end{gather*}
Here, ``critical point'' means that the first variation of $\mathcal{A}_H$ has to vanish at the trajectory $\gamma\in C^\infty(I,M)$ where we only allow for variations of $\gamma$ which keep the endpoints of $\gamma$ fixed. Sometimes, for instance in Floer theory, one wishes to view certain trajectories as actual critical points of some action functional. There are several ways to achieve this, e.g. by putting the endpoints of a trajectory on a Lagrangian or by only considering periodic trajectories.\\
Since HSs are given in terms of \underline{real-valued} manifolds $M$, forms $\omega$, and functions $H$, it is only natural to ask whether a similar construction with similar properties exists for \underline{complex-valued} manifolds $X$, forms $\Omega$, and functions $\mathcal{H}$. The answer to this question directly leads to the notion of \textbf{holomorphic Hamiltonian systems} (HHSs). Similarly to real Hamiltonian systems\footnote{To distinguish real- and complex-valued Hamiltonian systems, we call real-valued Hamiltonian systems real Hamiltonian systems from now on.} (RHSs), HHSs are also described by three data (cf. \autoref{subsec:def_HHS}): a complex manifold $X$ (implicitly defining an integrable complex structure $J$), a holomorphic symplectic $2$-form $\Omega$ on $X$, and a holomorphic function $\mH:X\to\mathbb{C}$.
%As for RHSs, one can also associate a (holomorphic) Hamiltonian vector field $X_\mH$, (holomorphic) trajectories, and action functionals with a HHS.
HHSs have been studied since the early 2000s, e.g. by Gerdjikov and Kyuldjiev et al. \cite{gerd2001}, \cite{gerd2002}, \cite{gerd2004} or by Arathoon and Fontaine \cite{arathoon2020}. In the given references, HHSs are usually viewed as complexifications of RHSs and mostly used as a tool to study RHSs which arise as real forms\footnote{The terms ``complexifications and real forms of Hamiltonian systems'' can be defined properly, but we do not give an explicit definition here. For us, it suffices to know that complexifications and real forms of Hamiltonian systems are defined similarly to complexifications and real forms of manifolds: a real manifold $M$ is the real form of a complex manifold $X$ and $X$ is a complexification of $M$ iff $M$ is the fixed point set of some anti-holomorphic involution on $X$.} of HHSs. In \cite{arathoon2020}, for instance, an integrable and compact RHS is constructed out of the HHS obtained from the complexification of the spherical pendulum.\\
In the present paper, we take a different approach. We study HHSs on their own and try to recreate the results known from RHSs for HHSs. To start with, we discuss the existence and uniqueness of holomorphic trajectories. Similarly to RHSs, holomorphic trajectories are defined as the holomorphic integral curves of the holomorphic Hamiltonian vector field $X_\mH$. We show in \autoref{subsec:holo_traj} that, locally, holomorphic trajectories always exist and are unique, given an initial value. Maximal holomorphic trajectories, however, are not unique anymore, even given an initial value, due to the effects of monodromy\footnote{Recently, the monodromy of the complexified Kepler problem has been studied by Sun and You (cf. \cite{shanzhong2020}).}. This behavior is in sharp contrast to RHSs. Nevertheless, the holomorphic trajectories still give rise to a foliation by energy hypersurfaces $\mH^{-1}(E)$ for regular values $E$ of $\mH$, as shown in \autoref{subsec:holo_traj}.\\
After this discussion, we prove in \autoref{subsec:holo_action_fun_and_prin} that the holomorphic trajectories satisfy an action principle\footnote{To the extent of the author's knowledge, an action principle for HHSs has not been formulated before in the literature.}, i.e., that they can be understood -- in some sense -- as critical points of certain action functionals. These action functionals are obtained by first decomposing a HHS $(X,\Omega,\mH)$ into four RHSs, one for each combination of real and imaginary part of $\Omega$ and $\mH$. To each RHS, we can assign the usual action functional of a RHS. Afterwards, we average each of these action functionals over the imaginary (or real) time axis and take an appropriate linear combination to obtain the action functional for the HHS $(X,\Omega,\mH)$. In fact, this method gives rise to a plethora of action functionals for the HHS $(X,\Omega,\mH)$ which simply differ by how one averages and takes the linear combination. We conclude \autoref{sec:HHS} with an application of HHSs. Precisely speaking, we establish a relation between Lefschetz fibrations and almost toric fibrations in \autoref{subsec:Lefschetz} using HHSs.\\
During the investigation of action functionals for HHSs, we observe that $J$, the complex structure of $X$, poses rather strong restrictions on the existence of certain holomorphic trajectories. In \autoref{subsec:holo_action_fun_and_prin}, we consider holomorphic trajectories whose domains are complex tori and interpret them as the complexification of periodic orbits. However, by the maximum principle, such holomorphic trajectories are always constant if the complex manifold in question is $X=\mathbb{C}^{2n}$ equipped with the standard complex structure $J = i$. The same argument does not hold anymore if we allow $J$ to be any almost complex structure. In his beautiful paper \cite{moser1995} from 1995, Moser shows that it is possible to pseudo-holomorphically embed complex tori in $\mathbb{R}^4$, where $\mathbb{R}^4$ is equipped with a suitable, not necessarily integrable almost complex structure $J$.\\
To avoid constraints imposed by the integrability of $J$, we introduce special Hamiltonian systems in \autoref{sec:PHHS} which are described by the same data as HHSs, but whose almost complex structure $J$ does not need to be integrable anymore. These Hamiltonian systems are called \textbf{pseudo-holomorphic Hamiltonian systems} (PHHSs) and exhibit, by design, the same properties as HHSs. In particular, pseudo-holomorphic trajectories of PHHSs induce foliations by regular energy hypersurfaces $\mH^{-1}(E)$ and obey an action principle (cf. \autoref{subsec:def_PHHS}).\\
At first glance, PHHSs may appear to be contrived and artificial, especially since, by definition, the imaginary part of $\Omega$ does not need to be closed anymore. However, the non-closedness of the imaginary part of $\Omega$ is an unavoidable consequence of the non-integrability of $J$, as we show in \autoref{subsec:rel_HHS_PHHS}. In fact, we prove that we recover a HHS from a PHHS if and only if $J$ is integrable or, equivalently, the imaginary part of $\Omega$ is closed. To further strengthen our claim that PHHSs are indeed a natural generalization of HHSs, we show that the space of proper\footnote{A proper PHHS is a PHHS which is not simultaneously a HHS.} PHHSs is open and dense in the space of PHHSs on a fixed manifold $X$ with $\text{dim}_\mathbb{R}(X)>4$ implying that proper PHHSs are generic. To prove that proper PHHSs are generic, we first give a method to construct proper PHHSs out of HHSs (cf. \autoref{subsec:constructing_PHHS}). The method itself is very interesting, since it is related to hyperkähler structures and allows us to equip the cotangent bundle of a complex manifold with the structure of a PHHS. Lastly, we use this construction to deform HHSs by proper PHHSs (cf. \autoref{subsec:deforming_HHS}).

\newpage
\section{Holomorphic Hamiltonian Systems}
\label{sec:HHS}

In this section, we recreate the results known from RHSs (laid out in \autoref{sec:intro}% or in \autoref{sec:RHS}
) for holomorphic Hamiltonian systems. We first cover the basic notion of a holomorphic Hamiltonian system (\autoref{subsec:def_HHS}), including the definition of a holomorphic symplectic manifold and Darboux's theorem for holomorphic symplectic manifolds, afterwards discuss the properties of holomorphic trajectories (\autoref{subsec:holo_traj}), then formulate an action principle for holomorphic Hamiltonian systems (\autoref{subsec:holo_action_fun_and_prin}), and, lastly, apply HHSs to investigate the relation between Lefschetz and almost toric fibrations (\autoref{subsec:Lefschetz}).

\subsection{HHS: Basic Definitions and Notions}
\label{subsec:def_HHS}

To begin with, we define a holomorphic symplectic manifold:
\begin{Def}[Holomorphic symplectic manifold]\label{def:holo_sym_man}
 A pair $(X,\Omega)$ is called \textbf{holomorphic symplectic manifold}\footnote{Warning: In some branches of algebraic geometry, the term ``holomorphic symplectic manifold'' is also used, but defined with additional constraints on $X$ and $\Omega$!} (HSM) iff $X$ is a complex manifold and $\Omega$ is a holomorphic $2$-form on $X$ which is closed and non-degenerate on the $(1,0)$-tangent bundle $T^{(1,0)}X$ of $X$. In this setup, $\Omega$ is called the holomorphic symplectic $2$-form of $(X,\Omega)$.
\end{Def}
Let us spend some time understanding the definition of a HSM. Recall that a complex manifold $X$ is defined via an atlas $\{(\phi_\alpha, U_\alpha)\}_{\alpha\in I}$ of charts with values in $\mathbb{C}^m$ such that their transition functions are holomorphic. If $\phi = (z_1,\ldots, z_m):U\to V\subset\mathbb{C}^m$ is such a holomorphic chart, then a holomorphic $2$-form $\Omega$ on $X$ can locally be written as:
\begin{gather*}
 \Omega\vert_U = \sum^{m}_{i,j = 1} \Omega_{ij} dz_i\wedge dz_j,
\end{gather*}
where $\Omega_{ij}:U\to\mathbb{C}$ are holomorphic functions on $U$. Similarly to the real case, one can define an exterior derivative $d$ on complex-valued forms such that closedness of $\Omega$ simply equates to $d\Omega = 0$. To understand the non-degeneracy in the definition of a HSM, recall that every complex manifold $X$ implicitly defines an integrable (almost) complex structure $J$ on $X$ and that, further, the complexified tangent and cotangent bundle of $X$, viewed as a real manifold, each decompose into a direct sum of two subbundles:
\begin{gather*}
 T_\mathbb{C}X = T^{(1,0)}X\oplus T^{(0,1)}X;\quad T^\ast_\mathbb{C}X = T^{\ast, (1,0)}X\oplus T^{\ast, (0,1)}X,
\end{gather*}
where the $(1,0)$- and $(0,1)$-bundles are fiberwise eigenspaces of $J$ (or its dual $J^\ast$) with eigenvalue $i$ and $-i$, respectively. By construction, the local forms $dz_i$ are local sections of $T^{\ast, (1,0)}X$ and, hence, map elements of $T^{(0,1)}X$ to zero. This implies that holomorphic $2$-forms can never be non-degenerate on the entire complexified tangent bundle, as they always vanish on the $(0,1)$-bundle. For a holomorphic $2$-form $\Omega$, we can at most achieve non-degeneracy on $T^{(1,0)}X$, i.e.:
\begin{gather*}
 \forall x\in X\ \forall V\in T^{(1,0)}_xX\backslash\{0\}\ \exists W\in T^{(1,0)}_xX:\quad \Omega_x(V, W) \neq 0.
\end{gather*}
In particular, as in the real case, non-degeneracy of $\Omega$ implies that the complex dimension $m$ of $X$ is even. Also note that, by construction, $\Omega$ and $J$ satisfy the following relations:
\begin{gather}\label{eq:J-anticompatible}
 \Omega (J\cdot, \cdot) = \Omega (\cdot, J\cdot) = i\cdot\Omega;\quad \Omega (J\cdot, J\cdot) = -\Omega.
\end{gather}
As in ``real'' symplectic geometry, there are two standard examples of HSMs: the first one is $X = \mathbb{C}^{2n}$ together with the standard form $\Omega = \sum^n_{j = 1} dP_j\wedge dQ_j$, where $(Q_1,\ldots, Q_n, P_1,\ldots, P_n)\in\mathbb{C}^{2n}$. The other one is the holomorphic cotangent bundle $X = T^{\ast, (1,0)}Y$ of a complex manifold $Y$ with canonical $2$-form $\Omega_\text{can} = d\Lambda_\text{can}$, where $\Lambda_\text{can}$ is the holomorphic Liouville $1$-form.\\
We know by Darboux's theorem that symplectic manifolds exhibit no local invariants, since they all are locally isomorphic to the standard symplectic manifold\linebreak $(\mathbb{R}^{2n}, \sum^n_{i = 1} dp_i\wedge dq_i)$. The same statement is true for HSMs:
\begin{Thm}[Darboux's theorem for HSMs]\label{thm:holo_Darboux}
 Let $(X,\Omega)$ be a HSM of complex dimension $\text{\normalfont dim}_\mathbb{C}(X) = 2n$ $(n\in\mathbb{N})$. Then, for every point $x\in X$, there is a holomorphic chart $\psi = (Q_1,\ldots, Q_n, P_1,\ldots, P_n):U\to V\subset\mathbb{C}^{2n}$ of $X$ near $x$ such that
 \begin{gather*}
  \Omega\vert_{U} = \sum\limits^n_{j = 1} dP_j\wedge dQ_j.
 \end{gather*}
\end{Thm}
The correctness of Darboux's theorem for HSMs is widely accepted in the mathematical community, however, no proof has yet been formally written down, at least to the author's extent of knowledge. For completeness' sake, a proof of Darboux's theorem for HSMs is provided in \autoref{app:darboux}. We will make use of Darboux's theorem for HSMs in \autoref{subsec:deforming_HHS}.\\
Now, let us turn our attention to holomorphic Hamiltonian systems:
\begin{Def}[Holomorphic Hamiltonian system]\label{def:holo_ham_sys}
 We call the triple $(X,\Omega, \mathcal{H})$ a \textbf{holomorphic Hamiltonian system} (HHS) iff $(X,\Omega)$ is a HSM and $\mH:X\to\mathbb{C}$ is a holomorphic function on $X$. In this setup, we call $\mH$ the \textbf{Hamilton function} or, simply, the Hamiltonian of the HHS $(X,\Omega, \mH)$.
\end{Def}
Examples of HHSs include the standard HSM $(\mathbb{C}^{2n}, \sum^n_{j = 1} dP_j\wedge dQ_j)$ together with any holomorphic function $\mH:\mathbb{C}^{2n}\to\mathbb{C}$ on it and holomorphic cotangent bundles $(T^{\ast, (1,0)}Y, \Omega_\text{can})$ together with natural Hamiltonians. We call a Hamiltonian on a holomorphic cotangent bundle natural iff it can be written as a sum of kinetic and potential energy, $\mH = \mathcal{T} + \mathcal{V}$. Hereby, we say a holomorphic function $\mathcal{V}:T^{\ast, (1,0)}Y\to\mathbb{C}$ denotes potential energy iff it factors through the holomorphic projection $\pi:T^{\ast, (1,0)}Y\to Y$, i.e., can be written as $\mathcal{V} = \mathcal{V}_0\circ\pi$ for some holomorphic function $\mathcal{V}_0:Y\to \mathbb{C}$. Furthermore, a holomorphic function $\mathcal{T}:T^{\ast, (1,0)}Y\to\mathbb{C}$ is called kinetic energy iff $2\mathcal{T}(x)\equiv g^\ast (x,x)$ $\forall x\in T^{\ast, (1,0)}Y$, where $g^\ast$ is the dualization of some holomorphic metric $g$ on $Y$. A holomorphic metric $g$ on $Y$ is a holomorphic symmetric $\mathbb{C}$-bilinear form which is non-degenerate on $T^{(1,0)}Y$, i.e., for any holomorphic chart $\phi = (z_1,\ldots, z_n):U\to V\subset\mathbb{C}^n$ of $Y$, $g$ can be written as
\begin{gather*}
 g\vert_U = \sum^n_{i,j = 1} g_{ij} dz_i\otimes dz_j,
\end{gather*}
where $g_{ij}:U\to\mathbb{C}$ are holomorphic functions satisfying $g_{ij} = g_{ji}$ and $\text{det}(g_{ij})\neq 0$.
We will study these examples in more detail in the upcoming subsections.

\subsection{Holomorphic Trajectories}
\label{subsec:holo_traj}

Our next goal is to investigate the dynamics of a HHS. As for RHSs, they are determined by a Hamiltonian vector field:
\begin{Def}[Holomorphic Hamiltonian vector field]\label{def:holo_ham_field}
 Let $(X,\Omega, \mH)$ be a HHS. We call the holomorphic vector field $X_\mH$ on $X$ defined by $\iota_{X_\mH}\Omega = -d\mH$ the (holomorphic) \textbf{Hamiltonian vector field} of the HHS $(X,\Omega, \mH)$.
\end{Def}
\begin{Rem}[$X_\mH$ is well-defined]\label{rem:ham_vec_field_well-def}
 Note that a holomorphic vector field $V$ on a complex manifold $X$ can -- in a holomorphic chart $\phi \equiv (z_1,\ldots, z_{m}):U\to V\subset\mathbb{C}^m$ -- be written as
 \begin{gather*}
  V\vert_U = \sum^m_{j = 1} V_j\partial_{z_j},
 \end{gather*}
 where $V_j:U\to\mathbb{C}$ are holomorphic functions on U. Thus, a holomorphic vector field on $X$ only attains values in the bundle $T^{(1,0)}X$. Together with the non-degeneracy of $\Omega$ on $T^{(1,0)}X$, this implies that the Hamiltonian vector field $X_\mH$ is well-defined.
\end{Rem}
Before we define what a trajectory of a HHS is, it is wise to study $X_\mH$ or, better yet, holomorphic vector fields in general. The following proposition is a standard result from complex geometry:
\begin{Prop}[Holomorphic vector fields $\Leftrightarrow$ $J$-preserving vector fields]\label{prop:holo_vec_field_equiv_J_pre_vec_field}
 Let $X$ be a complex manifold with complex structure $J\in\Gamma (\text{\normalfont End}(TX))$. Then, the tangent bundles $TX$\footnote{For $TX$, $X$ is viewed as a real manifold.} and $T^{(1,0)}X$ are isomorphic as smooth complex vector bundles via
 \begin{gather*}
  f:TX\to T^{(1,0)}X,\quad v\mapsto \frac{1}{2}(v - i\cdot J(v)),
 \end{gather*}
 where the fiberwise complex vector space structure of $TX$ is given by the complex multiplication $(a+ib)\odot v\coloneqq av + bJ(v)$ for every $a,b\in\mathbb{R}$ and $v\in TX$.\\
 Now consider the space $\Gamma_J (TX)$ of smooth real $J$-preserving\footnote{$J$-preserving vector fields are called infinitesimal automorphisms in \cite{kobayashi1969}.} vector fields on $X$:
 \begin{gather*}
  \Gamma_J (TX)\coloneqq \{V_R\in\Gamma (TX)\mid L_{V_R}J = 0\},
 \end{gather*}
 where $L_{V_R}J$ is the Lie derivative of $J$ with respect to $V_R$ and $\Gamma (TX)$ is the space of smooth real vector fields on $X$. Then, $\Gamma_J (TX)$ together with the standard commutator $[\cdot, \cdot]$ of vector fields and the complex structure $J$ forms a complex Lie algebra. In fact, $(\Gamma_J (TX), [\cdot, \cdot])$ is isomorphic as complex Lie algebras to the space $(\Gamma (T^{(1,0)}X), [\cdot,\cdot])$ of holomorphic vector fields on $X$ via
 \begin{gather*}
  F:\Gamma_J (TX)\to \Gamma (T^{(1,0)}X),\quad V_R\mapsto \frac{1}{2}(V_R - i\cdot J(V_R)).
 \end{gather*}
\end{Prop}
The proof of Proposition \autoref{prop:holo_vec_field_equiv_J_pre_vec_field} can be found in Chapter IX of \cite{kobayashi1969} (cf. Proposition 2.10 and 2.11). One important consequence of Proposition \autoref{prop:holo_vec_field_equiv_J_pre_vec_field} which we heavily use later on is the fact that the real and imaginary part of a holomorphic vector field commute:
\begin{Cor}[$V_R$ and $J(V_R)$ commute]\label{cor:commute}
 Let $X$ be a complex manifold with complex structure $J$ and $V_R\in\Gamma_J (TX)$ be a $J$-preserving vector field. Then:
 \begin{gather*}
  \left[V_R, J(V_R)\right] = 0.
 \end{gather*}
 In particular, the real and imaginary part of holomorphic vector fields on $X$ commute.
\end{Cor}
\begin{proof}
 By Proposition \autoref{prop:holo_vec_field_equiv_J_pre_vec_field}, we know that $(\Gamma_J (TX), [\cdot, \cdot])$ is a complex Lie algebra, hence:
 \begin{gather*}
  \left[V_R, J(V_R)\right] = J\left([V_R, V_R]\right) = 0.
 \end{gather*}
\end{proof}
By Proposition \autoref{prop:holo_vec_field_equiv_J_pre_vec_field}, we can associate with the Hamiltonian vector field $X_\mH$ of a HHS $(X,\Omega,\mH)$ a $J$-preserving vector field $X^R_\mH$ which is uniquely determined by
\begin{gather*}
 X_\mH = \frac{1}{2}(X^R_\mH - i\cdot J(X^R_\mH)).
\end{gather*}
Equipped with this knowledge, there are now two ways to define holomorphic trajectories of a HHS: the first one is to simply say that holomorphic trajectories are holomorphic integral curves of the holomorphic Hamiltonian vector field $X_\mH$. The second one is to define the holomorphic trajectories as analytic continuations of the integral curves of $X^R_\mH$. Both definitions are indeed equivalent and we make use of both of them. For our purposes, we use the first one as the actual definition and the second one to construct and investigate holomorphic trajectories afterwards:
\begin{Def}[Holomorphic trajectories]\label{def:holo_traj}
 Let $(X,\Omega,\mH)$ be a HHS and\linebreak $X_\mH = 1/2(X^R_\mH - i\cdot J(X^R_\mH))$ be its Hamiltonian vector field. We call a holomorphic map $\gamma:U\to X$ a \textbf{holomorphic trajectory} of the HHS $(X,\Omega, \mH)$ iff $\gamma$ satisfies the holomorphic integral curve equation:
 \begin{gather*}
  \gamma^\prime (z) = X_\mH (\gamma (z))\quad\forall z\in U,
 \end{gather*}
 where $U\subset\mathbb{C}$ is an open and connected subset and $\gamma^\prime$ is the complex derivative of $\gamma$. We call a holomorphic trajectory $\gamma:U\to X$ \textbf{maximal} iff for every holomorphic trajectory $\hat\gamma:\hat U\to X$ with $U\subset \hat U$ and $\hat\gamma\vert_U = \gamma$ one has $\hat U = U$ and $\hat \gamma = \gamma$. Sometimes, we call the integral curves of the real vector field $X^R_\mH$ the \textbf{real trajectories} of the HHS $(X,\Omega,\mH)$.
\end{Def}
Next, let us consider the existence and uniqueness of holomorphic trajectories:%In particular, we want to show that for any given initial value $x_0\in X$ and $z_0\in\mathbb{C}$ a holomorphic trajectory $\gamma^{z_0, x_0}:U\to X$ with $\gamma^{z_0, x_0} (z_0) = x_0$ exists and that two holomorphic trajectories $\gamma^{z_0, x_0}_1:U_1\to X$ and $\gamma^{z_0, x_0}_2:U_2\to X$ for the same initial value are equal iff their domains $U_1$ and $U_2$ are equal:

\begin{Prop}[Existence and uniqueness of holomorphic trajectories]\label{prop:holo_traj}
 Let $(X,\Omega,\mH)$ be a HHS. Then, for any $z_0\in\mathbb{C}$ and $x_0\in X$, there exists an open and connected subset $U\subset\mathbb{C}$ with $z_0\in U$ and a holomorphic trajectory\ $\gamma^{z_0, x_0}:U\to X$ of $(X,\Omega,\mH)$ with $\gamma^{z_0, x_0} (z_0) = x_0$. Two holomorphic trajectories $\gamma^{z_0, x_0}_1:U_1\to X$ and $\gamma^{z_0, x_0}_2:U_2\to X$ with $\gamma^{z_0, x_0}_1 (z_0) = x_0 = \gamma^{z_0, x_0}_2 (z_0)$ locally coincide, in particular, they are equal iff their domains $U_1$ and $U_2$ are equal. Furthermore, the holomorphic trajectory $\gamma^{z_0, x_0}$ depends holomorphically on $z_0$ and $x_0$.
\end{Prop}

\begin{proof}
 Let $(X,\Omega,\mH)$ be a HHS with Hamiltonian vector field $X_\mH = 1/2(X^R_\mH - iJ(X^R_\mH))$ and let $z_0\in\mathbb{C}$ and $x_0\in X$ be any points. To construct a holomorphic trajectory $\gamma^{z_0, x_0}$, we first realize that $t\mapsto\gamma^{z_0, x_0}(t + is)$ for fixed $s\in\mathbb{R}$ is a real trajectory. We can see this by taking the real part of the holomorphic integral curve equation. Thus, finding holomorphic trajectories amounts to finding analytic continuations of real trajectories. To accomplish this task, we observe that similarly $s\mapsto\gamma^{z_0, x_0}(t + is)$ for fixed $t\in\mathbb{R}$ is an integral curve of $J(X^R_\mH)$. %by considering the imaginary part of the holomorphic integral curve equation.
 Naively, one hopes that $\gamma^{z_0, x_0}(t + is)$ is given by:
 \begin{gather*}
   \gamma^{z_0, x_0}(t + is) = \varphi^{J(X^R_\mH)}_{s-s_0}\circ\varphi^{X^R_\mH}_{t-t_0} (x_0),
 \end{gather*}
 where $z_0 = t_0 + is_0$ and $\varphi^{J(X^R_\mH)}_{s-s_0}$ and $\varphi^{X^R_\mH}_{t-t_0}$ are the flows of the vector fields $J(X^R_\mH)$ and $X^R_\mH$ with times $s-s_0$ and $t-t_0$, respectively. In general, however, this expression is problematic: even though it is an integral curve of $J(X^R_\mH)$ for fixed $t$, it might not be an integral curve of $X^R_\mH$ for fixed $s$ anymore due to the composition with $\varphi^{J(X^R_\mH)}_{s-s_0}$. In order to avoid this problem, we need the composition of the flows to commute, at least for small times $t-t_0$ and $s-s_0$. This occurs if the vector fields $X^R_\mH$ and $J(X^R_\mH)$ themselves commute. In our situation, this is indeed the case, as we can use Corollary \autoref{cor:commute} and the fact that $X_\mH$ is a holomorphic vector field. Thus, we can define:
 \begin{gather*}
   \gamma^{z_0, x_0}(t + is)\coloneqq \varphi^{J(X^R_\mH)}_{s-s_0}\circ\varphi^{X^R_\mH}_{t-t_0} (x_0)\equiv \varphi^{X^R_\mH}_{t-t_0}\circ\varphi^{J(X^R_\mH)}_{s-s_0} (x_0)\equiv\varphi^{(t-t_0)X^R_\mH + (s-s_0)J(X^R_\mH)}_1 (x_0).
 \end{gather*}
 The expressions above are well-defined for $|t-t_0|,|s-s_0|<\varepsilon$ with $\varepsilon$ small enough and all identical due to the commutativity of $X^R_\mH$ and $J(X^R_\mH)$.\\
 Let us check that the given expressions for $\gamma^{z_0, x_0}$ indeed define a holomorphic trajectory. By construction, the map $\gamma^{z_0, x_0}$ is holomorphic, as it satisfies the Cauchy-Riemann equations. Hence, we only need to compute the complex derivative $\gamma^{z_0, x_0\ \prime}$.\linebreak If $\phi = (z_1,\ldots, z_{2n}):V\to W\subset\mathbb{C}^{2n}$ is a holomorphic chart of $X$ near $x_0$,\linebreak then we can define the complex derivative $\gamma^{z_0, x_0\ \prime}(z)$ for suitable $z$\linebreak using $(\gamma^{z_0, x_0}_1(z),\ldots, \gamma^{z_0, x_0}_{2n}(z))\coloneqq \phi\circ\gamma^{z_0, x_0} (z)$:
 \begin{gather*}
 \gamma^{z_0, x_0\ \prime} (z)\coloneqq\sum^{2n}_{j = 1} \gamma^{z_0, x_0\ \prime}_j (z)\cdot\left.\pa_{z_j}\right\vert_{\gamma^{z_0, x_0} (z)},
 \end{gather*}
 where $\gamma^{z_0, x_0\ \prime}_j (z)$ is the usual complex derivative of a holomorphic map from $\mathbb{C}$ to $\mathbb{C}$. A straightforward calculation reveals that the complex derivative $\gamma^{z_0, x_0\ \prime}(z)$ equates to:
 \begin{gather*}
  \gamma^{z_0, x_0\ \prime}(z) = \frac{1}{2}\left(\frac{\pa \gamma^{z_0, x_0}}{\pa t}(z) - i\cdot\frac{\pa \gamma^{z_0, x_0}}{\pa s}(z)\right).
 \end{gather*}
 By definition of $\gamma^{z_0, x_0}$, we have:
 \begin{gather*}
  \frac{\pa \gamma^{z_0, x_0}}{\pa t}(z) = X^R_\mH (\gamma^{z_0, x_0}(z));\quad \frac{\pa \gamma^{z_0, x_0}}{\pa s}(z) = J\left(X^R_\mH (\gamma^{z_0, x_0}(z))\right).
 \end{gather*}
 Putting everything together gives:
 \begin{gather*}
  \gamma^{z_0, x_0\ \prime}(z) = \frac{1}{2}\left(X^R_\mH (\gamma^{z_0, x_0}(z)) - i\cdot J\left(X^R_\mH (\gamma^{z_0, x_0}(z))\right)\right) = X_\mH (\gamma^{z_0, x_0}(z)).
 \end{gather*}
 Thus, $\gamma^{z_0, x_0}$ is indeed a holomorphic trajectory. Clearly, $\gamma^{z_0, x_0}$ satisfies $\gamma^{z_0, x_0}(z_0) = x_0$ proving the existence in Proposition \autoref{prop:holo_traj}.\\
 To show local uniqueness given an initial value, we recall that $\gamma^{z_0, x_0}$ is just an integral curve of $X^R_\mH$ along the $t$-axis satisfying $\gamma^{z_0, x_0}(z_0) = x_0$. Hence, every other holomorphic trajectory $\hat\gamma^{z_0, x_0}$ with $\hat\gamma^{z_0, x_0}(z_0) = x_0$ agrees with $\gamma^{z_0, x_0}$ for $s = s_0$ and $t$ near $t_0$. This allows us to apply the identity theorem for holomorphic functions to the coordinates of $\gamma^{z_0, x_0}$ and $\hat\gamma^{z_0, x_0}$ in a holomorphic chart near $x_0$ giving us the local uniqueness. In order to show that $\gamma^{z_0, x_0}$ and $\hat\gamma^{z_0, x_0}$ coincide completely iff their domains are equal, we cover the images of $\gamma^{z_0, x_0}$ and $\hat\gamma^{z_0, x_0}$ with holomorphic charts and repeatedly apply the identity theorem.\\
 Lastly, we need to show that $\gamma^{z_0, x_0}$ depends analytically on $z_0$ and $x_0$. For $z_0$, this is trivial, since $\gamma^{z_1, x_0}(z)$ and $\gamma^{z_2, x_0}(z)$ for $z_1\neq z_2$ only differ by a translation in $z$. For $x_0$, this is true if and only if the flows $\varphi^{X^R_\mH}_{t-t_0}$ and $\varphi^{J(X^R_\mH)}_{s-s_0}$ of $X^R_\mH$ and $J(X^R_\mH)$ are holomorphic maps from and to $X$. As explained in Chapter IX of \cite{kobayashi1969}, the $J$-preserving vector fields on $X$ are exactly those real vector fields on $X$ whose flow is holomorphic. Remembering that, by Proposition \autoref{prop:holo_vec_field_equiv_J_pre_vec_field}, the vector fields $X^R_\mH$ and $J(X^R_\mH)$ are $J$-preserving concludes the proof.
\end{proof}

\begin{Rem}\label{rem:holo_traj_alpha}
 In the last proof, we have used that a holomorphic trajectory $\gamma (t+is)$ of a HHS $(X,\Omega,\mH)$ is an integral curve of $X^R_\mH$ for fixed $s$ and an integral curve of $J(X^R_\mH)$ for fixed $t$. We can generalize this observation. If we express $t+is$ in polar coordinates, $t+is = re^{i\alpha}$, then $\gamma (re^{i\alpha})$ is an integral curve of $\cos(\alpha) X^R_\mH + \sin (\alpha)J(X^R_\mH)$ for fixed $\alpha$.
\end{Rem}

The properties we have found so far seem to indicate that holomorphic trajectories of a HHS exhibit the same behavior as trajectories of a RHS. However, this is not entirely true. In sharp contrast to the real case, the maximal holomorphic trajectories, given an initial value, do \underline{not} need to be unique, as the following counterexample demonstrates.

\begin{Ex}[Central problem in one complex dimension]\label{ex:holo_cen_prob}\normalfont
 {\textcolor{white}{Easter Egg}}\linebreak Let $X\coloneqq T^{\ast, (1,0)}\mathbb{C}^\times\cong\mathbb{C}^\times\times\mathbb{C}$ be the holomorphic cotangent bundle of $\mathbb{C}^\times\coloneqq\mathbb{C}\backslash\{0\}$ together with the standard form $\Omega = \Omega_\text{can} = dP\wedge dQ$, $(Q,P)\in X$, and the natural Hamiltonian\footnote{The given Hamiltonian is even regular, i.e., $d\mH\neq 0$ for all points of $X$.} $\mH (Q,P)\coloneqq \frac{P^2}{2} - \frac{1}{8Q^2}$. Physically speaking, the HHS $(X,\Omega, \mH)$ is the complexification of the RHS describing a single particle in one-dimensional position space subject to the almost Kepler-like central potential $V(q) = -\frac{1}{8q^2}$. The Hamiltonian vector field $X_\mH$ of the HHS $(X,\Omega,\mH)$ is given by
 \begin{gather*}
  X_\mH (Q,P) = P\cdot\pa_Q - \frac{1}{4Q^3}\cdot\pa_P.
 \end{gather*}
 Hence, the holomorphic trajectories $\gamma (z) = (Q(z), P(z))$ of $(X,\Omega, \mH)$ satisfy
 \begin{gather*}
  Q^\prime (z) = P (z);\quad P^\prime (z) = -\frac{1}{4Q^3 (z)}
 \end{gather*}
 or, combining both equations:
 \begin{gather*}
  Q^{\prime\prime} (z) = -\frac{1}{4Q^3 (z)}.
 \end{gather*}
 We want to determine the holomorphic trajectories $\gamma$ satisfying $\gamma (z_0) = x_0 = (Q_0,P_0)$ for $z_0, P_0\in\mathbb{C}$ and $Q_0\in\mathbb{C}^\times$. After translation in $z$, we can assume $z_0 = 0$. A straightforward computation reveals that, locally, the desired solutions are given by
 \begin{gather*}
  Q(z) = \sqrt{Q^2_0 + 2Q_0 P_0\cdot z + 2E_0\cdot z^2};\quad P(z) = Q^\prime (z),
 \end{gather*}
 where $E_0\coloneqq \mH (Q_0, P_0)$ and $\sqrt{\cdot}$ is chosen such that $\sqrt{Q^2_0} = Q_0$. Two square roots mapping $Q^2_0$ to $Q_0$ coincide on a small neighborhood of $Q^2_0$, however, they do not need to have the same domain. Let us make this precise by choosing values for $Q_0$ and $P_0$. Pick $Q_0 = 1$ and $P_0 = \frac{1}{2}$. Then, $E_0 = 0$ and $Q(z) = \sqrt{z + 1}$. Here, all square roots are admissible that coincide with the standard square root for real positive numbers. For instance, one can choose
 \begin{gather*}
  \sqrt{\cdot}^1:\{z\in\mathbb{C}\mid \text{Im}(z) \neq 0\text{ or }\text{Re}(z) > 0\}\to\mathbb{C},\quad z = re^{i\alpha}\mapsto \sqrt{r}e^{\frac{i\alpha}{2}}\ \left(\alpha \in (-\pi, \pi)\right)
 \end{gather*}
 or
 \begin{gather*}
  \sqrt{\cdot}^2:\{z\in\mathbb{C}\mid \text{Re}(z) \neq 0\text{ or }\text{Im}(z) > 0\}\to\mathbb{C},\quad z = re^{i\alpha}\mapsto \sqrt{r}e^{\frac{i\alpha}{2}}\ \left(\alpha \in \left(-\frac{\pi}{2}, \frac{3\pi}{2}\right)\right).
 \end{gather*}
 Using these square roots, we can define the holomorphic trajectories $\gamma_1: U_1\to X$ and $\gamma_2: U_2\to X$ given by
 \begin{align*}
  Q_1:U_1\coloneqq\{z\in\mathbb{C}\mid \text{Im}(z+1) \neq 0\text{ or }\text{Re}(z+1) > 0\}\to\mathbb{C},\quad Q_1(z)&\coloneqq \sqrt{z + 1}^1,\\
  Q_2:U_2\coloneqq\{z\in\mathbb{C}\mid \text{Re}(z+1) \neq 0\text{ or }\text{Im}(z+1) > 0\}\to\mathbb{C},\quad Q_2(z)&\coloneqq \sqrt{z + 1}^2.
 \end{align*}
 Clearly, $\gamma_1$ and $\gamma_2$ are maximal holomorphic trajectories satisfying $\gamma_1 (0) = (1,\frac{1}{2}) = \gamma_2 (0)$. However, their domains $U_1$ and $U_2$ differ showing that maximal trajectories are not unique, even given an initial value. In particular, the trajectories $\gamma_1$ and $\gamma_2$ yield different values for $z\in\{z\in\mathbb{C}\mid\text{Re}(z+1)<0\text{ and }\text{Im}(z+1)<0\}\subset U_1\cap U_2$, namely $\gamma_1 (z) = -\gamma_2 (z)$.
\end{Ex}

Even though the square root, which spoils the uniqueness of maximal trajectories in the previous example, is not well-defined on all of $\mathbb{C}\backslash\{0\}$, it is well-defined on the $2:1$ covering $z\mapsto z^2$ of $\mathbb{C}\backslash\{0\}$. In the same vein, the maximal trajectories of a HHS become unique after ``passing them down to a covering''. To understand this idea, recall that, for a RHS $(M,\omega, H)$, the energy hypersurface $H^{-1}(E)$ is foliated by maximal trajectories of $(M,\omega, H)$ if $E$ is a regular value of $H$. Similarly, there exists a foliation for regular energy hypersurfaces of a HHS $(X,\Omega, \mH)$. However, the leaves of the foliation are ``more than just'' the maximal trajectories this time:

\newpage

\begin{Prop}[Holomorphic foliation of a regular hypersurface]\label{prop:holo_foli}
 Let $(X,\Omega,\mH)$ be a HHS with complex structure $J$, Hamiltonian vector field $X_\mH = 1/2 (X^R_\mH - i J(X^R_\mH))$, and regular value $E$ of $\mH$. Then, the energy hypersurface $\mH^{-1}(E)$ admits a holomorphic\footnote{If the reader is unfamiliar with the notion of a holomorphic foliation, confer the proof of Proposition \autoref{prop:pseudo-holo_foli} for its definition.} foliation. The leaf $L_{x_0}$ of this foliation through a point $x_0\in\mH^{-1}(E)$ is given by
 \begin{align*}
  L_{x_0}\coloneqq \{y\in X\mid &y = \varphi^{X^R_\mH}_{t_1}\circ\varphi^{J(X^R_\mH)}_{s_1}\circ\varphi^{X^R_\mH}_{t_2}\circ\varphi^{J(X^R_\mH)}_{s_2}\circ\ldots\circ\varphi^{X^R_\mH}_{t_n}\circ\varphi^{J(X^R_\mH)}_{s_n} (x_0);\\
  &t_1,\ldots,t_n, s_1,\ldots, s_n\in\mathbb{R};\ n\in\mathbb{N}\},
 \end{align*}
 where $\varphi^{X^R_\mH}_{t_j}$ and $\varphi^{J(X^R_\mH)}_{s_j}$ are the flows of $X^R_\mH$ and $J(X^R_\mH)$ for time $t_j$ and $s_j$, respectively. Every holomorphic trajectory of $(X,\Omega,\mH)$ with energy $E$ is completely contained in one such leaf.
\end{Prop}

\begin{proof}
 Take the assumptions and notations from above. $E$ is a regular value of $\mH$, hence, $\mH^{-1}(E)$ is a complex submanifold of $X$. The (real) tangent space of $\mH^{-1} (E)$ consists of vectors $W$ in the (real) tangent space of $X$ satisfying $d\mH (W) = 0$. Using the holomorphicity of $\mH$, $d\mH\circ J = i\cdot d\mH$, we obtain
 \begin{align*}
  d\mH (X^R_\mH) &= d\mH (X_\mH) = -\Omega (X_\mH, X_\mH) = 0,\\
  d\mH (J(X^R_\mH)) &= i\cdot d\mH (X^R_\mH) = 0
 \end{align*}
 showing that $X^R_\mH$ and $J(X^R_\mH)$ live in the tangent space of $\mH^{-1}(E)$ at each point of $\mH^{-1}(E)$. This allows us to restrict $X^R_\mH$ and $J(X^R_\mH)$ to real vector fields on $\mH^{-1}(E)$.\\
 As $\mH$ is regular on $\mH^{-1} (E)$, neither $X^R_\mH$ nor $J(X^R_\mH)$ vanish on $\mH^{-1} (E)$. Furthermore, as real vector fields, they are $\mathbb{R}$-linearly independent at every point of $\mH^{-1}(E)$, since there is no real number which squares to $-1$. By Corollary \autoref{cor:commute}, $X^R_\mH$ and $J(X^R_\mH)$ also commute. This allows us to apply Frobenius' theorem to the distribution spanned by the real vector real fields $X^R_\mH$ and $J(X^R_\mH)$ giving us a foliation of $\mH^{-1}(E)$ whose leaves take the form described in Proposition \autoref{prop:holo_foli}.
 %As the vector space spanned by $X^R_\mH$ and $J(X^R_\mH)$ at each point of $\mH^{-1}(E)$ is closed under $J$, this foliation is $J$-holomorphic.
 As $X^R_\mH$ and $J(X^R_\mH)$, the vector fields generating the foliation, are real and imaginary part of a holomorphic vector field, the foliation itself is holomorphic by the holomorphic version of Frobenius' theorem (cf. Theorem 2.26 in \cite{voisin2002}). Comparing the construction of holomorphic trajectories in the proof of Proposition \autoref{prop:holo_traj} with the form of the leaves in Proposition \autoref{prop:holo_foli} reveals that a holomorphic trajectory is completely contained in one leaf concluding the proof.
\end{proof}

\begin{Rem}[Flows of $X^R_\mH$ and $J(X^R_\mH)$ do not commute globally]\label{rem:flows_do_not_commute}
 One might be tempted to set $n$ in the definition of the leaves in Proposition \autoref{prop:holo_foli} to $1$, since $X^R_\mH$ and $J(X^R_\mH)$\linebreak as well as their flows commute. However, this is only locally the case.\linebreak To illustrate this, consider Example \autoref{ex:holo_cen_prob} again. Choose the initial value\linebreak $x_0 = (Q_0, P_0) = (1, 1/2)$ and set $n=2$,  $t_1 = 0$, $s_1 = -2$, $t_2 = -2$, and $s_2 = 1$:
 \begin{gather*}
  \varphi^{X^R_\mH}_{0}\circ\varphi^{J(X^R_\mH)}_{-2}\circ\varphi^{X^R_\mH}_{-2}\circ\varphi^{J(X^R_\mH)}_{1} \left(1, \frac{1}{2}\right) = \left(\sqrt{2}\cdot e^{i\frac{5\pi}{8}}, \frac{1}{2\sqrt{2}}\cdot e^{-i\frac{5\pi}{8}}\right).
 \end{gather*}
 If we exchange the order of $s_1$ and $s_2$, the result is different:
 \begin{gather*}
  \varphi^{X^R_\mH}_{0}\circ\varphi^{J(X^R_\mH)}_{1}\circ\varphi^{X^R_\mH}_{-2}\circ\varphi^{J(X^R_\mH)}_{-2} \left(1, \frac{1}{2}\right) = -\left(\sqrt{2}\cdot e^{i\frac{5\pi}{8}}, \frac{1}{2\sqrt{2}}\cdot e^{-i\frac{5\pi}{8}}\right).
 \end{gather*}
\end{Rem}

In light of Proposition \autoref{prop:holo_foli}, one might say that the leaves of a HHS $(X,\Omega, \mH)$ should be considered to be the holomorphic counterpart to the maximal trajectories of a RHS $(M,\omega, H)$. Let us investigate this statement further. To do that, we first need to generalize the notion of holomorphic trajectories in such a way that we can view any Riemann surface as the domain of a trajectory, not only subsets of $\mathbb{C}$:

\begin{Def}[Geometric trajectory]\label{def:geo_traj}
 Let $(X,\Omega,\mH)$ be a HHS with regular value $E$ of $\mH$ and foliation $L = \{L_{x_0}\}_{x_0\in I}$ ($I$: some index set) of $\mH^{-1}(E)$ as in Proposition \autoref{prop:holo_foli}. Further, let $\Sigma$ be a Riemann surface, i.e., a connected, complex one-dimensional manifold. We call a holomorphic map $\gamma:\Sigma\to X$ a \textbf{geometric trajectory} of energy $E$ iff $\gamma$ is an immersion and the image of $\gamma$ is completely contained in one leaf $L_{x_0}$ of the foliation $L$.
\end{Def}

The definition of a geometric trajectory is reasonable, as every geometric trajectory is locally a holomorphic trajectory:

\begin{Prop}[Geometric trajectories are locally holomorphic trajectories]\label{prop:geo_traj}
 Let\linebreak $(X,\Omega,\mH)$ be a HHS with regular value $E$ of $\mH$, $\Sigma$ be a Riemann surface, and $\gamma:\Sigma\to X$ be a geometric trajectory of energy $E$. Then, for every $s_0\in\Sigma$, there exists an open neighborhood $V\subset\Sigma$ of $s_0$ and a holomorphic chart $\varphi:V\to U\subset\mathbb{C}$ of $\Sigma$ such that $\gamma\circ\varphi^{-1}:U\to X$ is a holomorphic trajectory of the HHS $(X,\Omega,\mH)$.
\end{Prop}

\begin{proof}
 Invoke the assumptions and notations from above. As $\gamma$ is a holomorphic immersion whose image is completely contained in one leaf and the leaves of $L$ are generated by the Hamiltonian vector field $X_\mH$, there exists for every $s\in\Sigma$ an uniquely determined vector $Y_\mH (s)\in T^{(1,0)}_s\Sigma$ such that:
 \begin{gather*}
  d\gamma\vert_s (Y_\mH (s)) = X_\mH (\gamma (s)).
 \end{gather*}
 These vectors form a holomorphic vector field $Y_\mH$ on $\Sigma$. Now pick $s_0\in\Sigma$ and $z_0\in\mathbb{C}$. By Proposition \autoref{prop:holo_traj} and Remark \autoref{rem:holo_vec_fields}, there exists an open and connected subset $U\subset\mathbb{C}$ such that $\varphi^{-1}:U\to\Sigma$ is a holomorphic integral curve of $Y_\mH$ satisfying $\varphi^{-1}(z_0) = s_0$. After shrinking $U$ if necessary, $\varphi^{-1}$ becomes a biholomorphism onto its image $\varphi^{-1}(U)\eqqcolon V$. Thus, $\varphi:V\to U$ is a holomorphic chart of $\Sigma$ near $s_0$. Furthermore, the curve $\gamma\circ\varphi^{-1}:U\to X$ is an integral curve of $X_\mH$:
 \begin{gather*}
  \left(\gamma\circ\varphi^{-1}\right)^\prime (z) = d\gamma\vert_{\varphi^{-1}(z)}\left(\varphi^{-1\,\prime} (z)\right) = d\gamma\vert_{\varphi^{-1}(z)}\left(Y_\mH(\varphi^{-1} (z))\right) = X_\mH (\gamma\circ\varphi^{-1}(z))\quad\forall z\in U,
 \end{gather*}
 hence, a holomorphic trajectory concluding the proof.
\end{proof}

Let us now assume that $(X,\Omega,\mH)$ is a HHS with regular value $E$ of $\mH$ and foliation $L = \{L_{x_0}\}_{x_0\in I}$ of $\mH^{-1}(E)$. Pick a leaf $L_{x_0}$. If $L_{x_0}$ is a complex submanifold of $X$, the inclusion $L_{x_0}\hookrightarrow X$ is clearly a geometric trajectory. If $L_{x_0}$ is not a complex submanifold of $X$, we can always equip $L_{x_0}$ with the structure of a complex manifold by choosing a suitable atlas such that the inclusion $L_{x_0}\hookrightarrow X$ becomes a geometric trajectory. The atlas in question consists of maps $\gamma^{-1}$ where $\gamma:U\to X$ is an injective holomorphic trajectory whose image is contained in $L_{x_0}$. In contrast to maximal trajectories, the geometric trajectories $L_{x_0}\hookrightarrow X$ are unique given an initial value $x_0\in \mH^{-1}(E)\subset X$ such that $x_0\in L_{x_0}$. In fact, the uniqueness can be expressed as a universal property: pick $x_0\in\mH^{-1}(E)$. Then, every geometric trajectory $\gamma:\Sigma\to X$ with initial value $x_0\in\gamma (\Sigma)$ factors uniquely through the geometric trajectory $L_{x_0}\hookrightarrow X$ and the geometric trajectory $L_{x_0}\hookrightarrow X$ is unique up to biholomorphisms with that property.\\
Often, the geometric trajectories $L_{x_0}\hookrightarrow X$ can be understood as coverings of or to be covered by maximal trajectories. Example \autoref{ex:holo_cen_prob} exemplifies this behavior. To see this, we first need to determine the leaves in Example \autoref{ex:holo_cen_prob}. We achieve this by applying the following proposition:

\begin{Prop}[Energy hypersurfaces of low-dimensional systems]\label{prop:low_dim}
 Let $(X,\Omega,\mH)$ be a HHS with $\text{\normalfont dim}_\mathbb{C}(X) = 2$ and regular value $E$ of $\mH$. Then, the leaves of $\mH^{-1} (E)$ are its connected components and, in particular, complex submanifolds of $X$.
\end{Prop}

\begin{proof}
 Take the notations and assumptions from above. $X$ is a complex two-dimensional manifold, hence, $\mH^{-1}(E)$ is a complex one-dimensional one. Likewise, the leaves of $\mH^{-1}(E)$ are one-dimensional complex manifolds, immersed in $\mH^{-1}(E)$. Thus, the leaves are open in $\mH^{-1}(E)$. By definition of a foliation, $\mH^{-1}(E)$ decomposes into a disjoint union of leaves. Therefore, the leaves are also closed in $\mH^{-1}(E)$ concluding the proof.
\end{proof}

Return to Example \autoref{ex:holo_cen_prob}. By Proposition \autoref{prop:low_dim}, the leaves in this example are just the connected components of the energy hypersurfaces $\mH^{-1}(E)$. For $E\neq 0$, the energy hypersurface is connected and only consists of one leaf, namely itself. For $E = 0$, we have:
\begin{gather*}
 \mH(Q,P) = \frac{P^2}{2} - \frac{1}{8Q^2} = \frac{1}{2}\left(P-\frac{1}{2Q}\right)\left(P+\frac{1}{2Q}\right) \stackrel{!}{=} 0.
\end{gather*}
We see that there are two connected components and, consequentially, two leaves this time: one with $Q\cdot P = \frac{1}{2}$ and another one with $Q\cdot P = -\frac{1}{2}$. We have already determined the maximal trajectories of the leaf with $Q\cdot P = \frac{1}{2}$: they are given by $Q (z) = \sqrt{z+1}$ and $P = Q^\prime (z)$ with appropriate square roots. Precomposing the maximal trajectories with the $2:1$ covering $z\mapsto z^2 -1$ gives us a well-defined map from $\mathbb{C}\backslash\{0\}$ to the leaf with $Q\cdot P = \frac{1}{2}$. In fact, this map is a biholomorphism. The leaf can be understood in this sense as the double cover of the maximal trajectories. For an example where a leaf is covered by a maximal trajectory, confer Example \autoref{ex:complex_torus}.\\
Before we conclude the section on holomorphic trajectories, we quickly add two comments. The first one concerns the notion of holomorphic and geometric trajectories: throughout the rest of this paper, we will not distinguish between holomorphic and geometric trajectories anymore and use both names interchangeably. The second one concerns holomorphic vector fields on general complex manifolds: 

\begin{Rem}\label{rem:holo_vec_fields}
 The previous results are not only true for holomorphic Hamiltonian vector fields and holomorphic trajectories of HHSs, but also for holomorphic vector fields and holomorphic integral curves on any complex manifold.
\end{Rem}

\subsection{Action Functionals and Principles for HHSs}
\label{subsec:holo_action_fun_and_prin}

As in the real case, we wish to link the holomorphic trajectories of a HHS $(X,\Omega,\mH)$ to critical points of some action functional. To achieve this, let us first study the holomorphic symplectic $2$-form $\Omega$ on $X$. By definition, $\Omega$ is non-degenerate on $T^{(1,0)}X$, but vanishes on $T^{(0,1)}X$. Every complex form $\Omega$ together with complex vectors $V$ and $W$ satisfies $\overline{\Omega (V,W)} = \overline{\Omega} (\overline{V},\overline{W})$, hence, the complex conjugate $\overline{\Omega}$ of the holomorphic symplectic form $\Omega$ is non-degenerate on $T^{(0,1)}X$, but vanishes on $T^{(1,0)}X$\footnote{Here, we have also used that the complex conjugation maps $T^{(1,0)}X$ to $T^{(0,1)}X$ and $T^{(0,1)}X$ to $T^{(1,0)}X$.}. In total, this implies that the real and imaginary part of $\Omega$,
\begin{gather*}
 \Omega_R = \frac{1}{2}(\Omega + \overline{\Omega})\quad\text{and}\quad\Omega_I = \frac{-i}{2}(\Omega - \overline{\Omega}),
\end{gather*}
are non-degenerate on the entire complexified tangent bundle $T_\mathbb{C}X = T^{(1,0)}X\oplus T^{(0,1)}X$ of $X$. Clearly, $\Omega_R$ and $\Omega_I$ are real $2$-forms on $X$ and, hence, must be already non-degenerate on the real tangent bundle $TX$. As $\Omega$ is holomorphic and closed, $\Omega_R$ and $\Omega_I$ are smooth and closed. Putting everything together, we find that the real and imaginary part $\Omega_R$ and $\Omega_I$ of $\Omega$, respectively, are symplectic $2$-forms on $X$ viewed as a real manifold.\\
Let us return to the HHS $(X,\Omega, \mH)$. Clearly, the real and imaginary part of the Hamiltonian $\mH = \mH_R + i\mH_I$ are smooth real functions on $X$. Thus, any HHS $(X,\Omega,\mH)$ gives rise to four \underline{RHSs}: $(X,\Omega_R,\mH_R)$, $(X,\Omega_I,\mH_R)$, $(X,\Omega_R,\mH_I)$, and $(X,\Omega_I,\mH_I)$\footnote{As we will see later on, it might be more appropriate to say that $(X,\Omega,\mH)$ gives rise to only two RHSs, since $(X,\Omega_R,\mH_R)$ and $(X,\Omega_I,\mH_I)$ as well as $(X,\Omega_I,\mH_R)$ and $(X,\Omega_R,-\mH_I)$ are subject to the same dynamics.}. Our next task is to determine the Hamiltonian vector fields of the four RHSs.\\
We start with the RHS $(X,\Omega_R,\mH_R)$. We write the holomorphic Hamiltonian vector field $X_\mH$ as $X_\mH = 1/2 (X^R_\mH - iJ(X^R_\mH))$ and compute $\iota_{X^R_\mH}\Omega_R$:
\begin{align*}
 \iota_{X^R_\mH}\Omega_R &= \frac{1}{2}\left(\iota_{X^R_\mH}\Omega + \iota_{X^R_\mH}\overline{\Omega}\right) = \frac{1}{2}\left(\iota_{X^R_\mH}\Omega + \iota_{\overline{X^R_\mH}}\overline{\Omega}\right) = \frac{1}{2}\left(\iota_{X^R_\mH}\Omega + \overline{\iota_{X^R_\mH}\Omega}\right)\\
 &= \frac{1}{2}\left(\iota_{X_\mH}\Omega + \overline{\iota_{X_\mH}\Omega}\right) = -\frac{1}{2}\left(d\mH + \overline{d\mH}\right) = -\frac{1}{2}d\left(\mH + \overline{\mH}\right) = - d\mH_R,
\end{align*}
where we used that $X^R_\mH$ is a real vector field on $X$, i.e., $\overline{X^R_\mH} = X^R_\mH$, and that\linebreak $\iota_{X_\mH}\Omega = \iota_{X^R_\mH}\Omega$ due to Equation \eqref{eq:J-anticompatible}. We deduce from the expression above that $X^R_\mH$ is the real Hamiltonian vector field of the RHS $(X,\Omega_R,\mH_R)$. Similarly, one can show the following proposition:

\begin{Prop}[HHS as four RHSs]\label{prop:four_RHS}
 Let $(X,\Omega = \Omega_R + i\Omega_I,\mH = \mH_R + i\mH_I)$ be a HHS with Hamiltonian vector field $X_\mH = 1/2 (X^R_\mH - iJ(X^R_\mH))$. Then:
 \begin{enumerate}
  \item $(X,\Omega_R, \mH_R)$ is a RHS with Hamiltonian vector field $X^R_\mH$.
  \item $(X,\Omega_R, \mH_I)$ is a RHS with Hamiltonian vector field $-J(X^R_\mH)$.
  \item $(X,\Omega_I, \mH_R)$ is a RHS with Hamiltonian vector field $J(X^R_\mH)$.
  \item $(X,\Omega_I, \mH_I)$ is a RHS with Hamiltonian vector field $X^R_\mH$.
 \end{enumerate}
\end{Prop}

\begin{Rem}[Cauchy-Riemann-like relations]\label{rem:cauchy-riemann}
 At first glance, one might be confused why the Hamiltonian vector fields of $(X,\Omega_R,\mH_R)$ and $(X,\Omega_I,\mH_I)$ coincide, while the Hamiltonian vector fields of $(X,\Omega_R,\mH_I)$ and $(X,\Omega_I,\mH_R)$ differ by a sign. However, this observation is just a consequence of the analyticity of the HHS $(X,\Omega,\mH)$ and one might think of it as Cauchy-Riemann-like relations:
 \begin{gather*}
  X^{\Omega_R}_{\mH_R} = X^{\Omega_I}_{\mH_I};\quad J\left(X^{\Omega_R}_{\mH_R}\right) = X^{\Omega_I}_{\mH_R} = - X^{\Omega_R}_{\mH_I},
 \end{gather*}
 where $X^{\Omega_a}_{\mH_b}$ is the Hamiltonian vector field of the RHS $(X,\Omega_a,\mH_b)$.
\end{Rem}

The upshot of Proposition \autoref{prop:four_RHS} is that the real trajectories of the HHS $(X,\Omega,\mH)$ are just the trajectories of the RHSs $(X,\Omega_R,\mH_R)$ and $(X,\Omega_I,\mH_I)$. In particular, the real trajectories of $(X,\Omega,\mH)$ are critical points of an action functional, at least if\linebreak $(X,\Omega = d\Lambda, \mH)$ is exact\footnote{A HHS $(X,\Omega,\mH)$ is exact iff $\Omega$ has a holomorphic primitive $\Lambda$.}. Of course, the same is true for the trajectories of $(X, \Omega_R, \mH_I)$ and $(X, \Omega_I, \mH_R)$:

\begin{Prop}[Action principle for real trajectories]\label{prop:action_fct_of_real_traj}
 Let $(X,\Omega = d\Lambda, \mH)$ be an exact HHS with Hamiltonian vector field $X_\mH = 1/2 (X^R_\mH - iJ(X^R_\mH))$ and decompositions\linebreak $\Omega = \Omega_R + i\Omega_I$, $\Lambda = \Lambda_R + i\Lambda_I$, and $\mH = \mH_R + i\mH_I$. Let $I_0\subset\mathbb{R}$ be an interval. We set $\mathcal{P}_{I_0}\coloneqq C^\infty (I_0,X)$ and define the action functionals $\mathcal{A}^{\Lambda_R}_{\mH_R}:\mathcal{P}_{I_0}\to\mathbb{R}$, $\mathcal{A}^{\Lambda_R}_{-\mH_I}:\mathcal{P}_{I_0}\to\mathbb{R}$\footnote{Note the different signs in the definition of $\mathcal{A}^{\Lambda_R}_{-\mH_I}$ due to the Cauchy-Riemann-like relations.}, $\mathcal{A}^{\Lambda_I}_{\mH_R}:\mathcal{P}_{I_0}\to\mathbb{R}$, $\mathcal{A}^{\Lambda_I}_{\mH_I}:\mathcal{P}_{I_0}\to\mathbb{R}$, $\mathcal{A}^{\Lambda}_{\mH}:\mathcal{P}_{I_0}\to\mathbb{C}$, and $\mathcal{A}^{\Lambda}_{i\mH}:\mathcal{P}_{I_0}\to\mathbb{C}$ by
 \begin{align*}
  \mathcal{A}^{\Lambda_R}_{\mH_R}[\gamma]&\coloneqq \int\limits_{I_0}\gamma^\ast\Lambda_R - \int\limits_{I_0}\mH_R\circ\gamma (t)dt,\quad \mathcal{A}^{\Lambda_R}_{-\mH_I}[\gamma]\coloneqq \int\limits_{I_0}\gamma^\ast\Lambda_R + \int\limits_{I_0}\mH_I\circ\gamma (t)dt,\\
  \mathcal{A}^{\Lambda_I}_{\mH_R}[\gamma]&\coloneqq \int\limits_{I_0}\gamma^\ast\Lambda_I - \int\limits_{I_0}\mH_R\circ\gamma (t)dt,\quad \mathcal{A}^{\Lambda_I}_{\mH_I}[\gamma]\coloneqq \int\limits_{I_0}\gamma^\ast\Lambda_I - \int\limits_{I_0}\mH_I\circ\gamma (t)dt,\\
  \mathcal{A}^{\Lambda}_{\mH}[\gamma]&\coloneqq \mathcal{A}^{\Lambda_R}_{\mH_R}[\gamma] + i\mathcal{A}^{\Lambda_I}_{\mH_I}[\gamma] = \int\limits_{I_0}\gamma^\ast\Lambda - \int\limits_{I_0}\mH\circ\gamma (t)dt,\\
  \mathcal{A}^{\Lambda}_{i\mH}[\gamma]&\coloneqq \mathcal{A}^{\Lambda_R}_{-\mH_I}[\gamma] + i\mathcal{A}^{\Lambda_I}_{\mH_R}[\gamma] = \int\limits_{I_0}\gamma^\ast\Lambda - i\int\limits_{I_0}\mH\circ\gamma (t)dt,
 \end{align*}
 where $\gamma\in\mathcal{P}_{I_0}$. Now, let $\gamma\in\mathcal{P}_{I_0}$ be any smooth path in $X$. Then, $\gamma$ is a real trajectory of $(X,\Omega,\mH)$ if and only if $\gamma$ is a ``critical point'' of the action functionals $\mathcal{A}^{\Lambda_R}_{\mH_R}$, $\mathcal{A}^{\Lambda_I}_{\mH_I}$, and $\mathcal{A}^{\Lambda}_{\mH}$. Similarly, $\gamma$ is a (real) integral curve of $J(X^R_\mH)$ if and only if $\gamma$ is a ``critical point'' of the action functionals $\mathcal{A}^{\Lambda_R}_{-\mH_I}$, $\mathcal{A}^{\Lambda_I}_{\mH_R}$, and $\mathcal{A}^{\Lambda}_{i\mH}$.
\end{Prop}

\begin{proof}
 Proposition \autoref{prop:action_fct_of_real_traj} is a consequence of Proposition \autoref{prop:four_RHS} and the action principle for RHSs.
\end{proof}

\begin{Rem}[Meaning of ``critical point'']\label{rem:critical_point}
 ``Critical point'' does not denote an actual critical point. ``Critical point'' in Proposition \autoref{prop:action_fct_of_real_traj} means that the first derivative of the action functionals vanishes at $\gamma$ \underline{only} for all variations of $\gamma$ \textbf{which keep the endpoints of $\mathbf{\gamma}$ fixed!} Often, one wishes to view trajectories as actual critical points of some action functional, not just with fixed endpoints. There are two main ways to achieve this:
 \begin{enumerate}
  \item One can only consider paths which start and end at points where the primitive of the symplectic form vanishes, usually Lagrangian submanifolds of the symplectic manifold.
  \item One can only consider periodic paths such that the boundary terms in the first derivative of the action cancel each other.
 \end{enumerate}
\end{Rem}

\begin{Rem}[Action functional for ``tilted'' trajectories]\label{rem:tilted_traj}
 The observation that the integral curves of $X^R_\mH$ and $J(X^R_\mH)$ are linked to the ``critical points'' of the action functionals $\mathcal{A}^{\Lambda}_{\mH}$ and $\mathcal{A}^{\Lambda}_{i\mH}$, respectively, can be generalized to ``tilted'' trajectories. For any $\alpha\in\mathbb{R}$, $\gamma$ is an integral curve of $\cos (\alpha)\cdot X^R_\mH + \sin (\alpha)\cdot J(X^R_\mH)$ if and only if $\gamma$ is a ``critical point'' of the action functional $\mathcal{A}^{\Lambda}_{e^{i\alpha}\mH}:\mathcal{P}_{I_0}\to\mathbb{C}$ defined by
 \begin{gather*}
  \mathcal{A}^{\Lambda}_{e^{i\alpha}\mH}[\gamma]\coloneqq \int\limits_{I_0}\gamma^\ast\Lambda - e^{i\alpha}\int\limits_{I_0}\mH\circ\gamma (t)dt.
 \end{gather*}
 Of course, the same action principle holds true if we only consider the real or imaginary part of $\mathcal{A}^{\Lambda}_{e^{i\alpha}\mH}$. This fact will be of great importance later on and in \autoref{app:action_functionals}.
\end{Rem}

As the holomorphic trajectories of the HHS $(X,\Omega,\mH)$ are analytic continuations of its real trajectories, one might be content with finding action functionals for the real trajectories. However, it is also possible to construct an action functional for the holomorphic trajectories of $(X,\Omega,\mH)$ by averaging the action functionals of the four underlying RHSs:
% The idea behind the construction is simple: Recall from the proof of Proposition \autoref{prop:holo_traj} that $\gamma$ is a holomorphic trajectory of a HHS $(X,\Omega, \mH)$ iff $\gamma (t + is)$ is an integral curve of $X^R_\mH$ for fixed $s$ and an integral curve of $J(X^R_\mH)$ for fixed $t$. By Proposition \autoref{prop:action_fct_of_real_traj}, the map $\gamma_s: t\mapsto \gamma (t+is)$, for any fixed $s$, is an integral curve of $X^R_\mH$ iff $\gamma_s$ is a ``critical point'' of $\mathcal{A}^{\Lambda_R}_{\mH_R}$. Likewise, the map $\gamma_t: s\mapsto \gamma (t+is)$, for any fixed $t$, is an integral curve of $J(X^R_\mH)$ iff $\gamma_t$ is a ``critical point'' of $\mathcal{A}^{\Lambda_R}_{-\mH_I}$. Similar statements still hold if we average the action functionals $\mathcal{A}^{\Lambda_R}_{\mH_R}[\gamma_s]$ and $\mathcal{A}^{\Lambda_R}_{-\mH_I}[\gamma_t]$ over $s$ and $t$, respectively. For instance, $\gamma_s: t\mapsto \gamma (t+is)$ is an integral curve of $X^R_\mH$ for \underline{all} $s$ iff $\gamma$ is a ``critical point'' of the functional
% \begin{gather*}
%  \gamma\mapsto \int \mathcal{A}^{\Lambda_R}_{\mH_R}[\gamma_s] ds.
% \end{gather*}
% Thus, averaging the action functionals $\mathcal{A}^{\Lambda_R}_{\mH_R}[\gamma_s]$ and $\mathcal{A}^{\Lambda_R}_{-\mH_I}[\gamma_t]$ over $s$ and $t$, respectively, and taking a suitable complex linear combination of them afterwards gives us the desired action functional for holomorphic trajectories:

\begin{Lem}[Action principle for holomorphic trajectories]\label{lem:holo_action_prin}
 Let $(X,\Omega = d\Lambda,\mH)$ be an exact HHS with $\Lambda = \Lambda_R + i\Lambda_I$. Furthermore, let $R\coloneqq [t_1,t_2] + i[s_1,s_2]\subset\mathbb{C}$ be a rectangle in the complex plane with real numbers $t_1< t_2$ and $s_1< s_2$. Denote the space of smooth maps from $R$ to $X$ by $\mathcal{P}_R$ and define the action functional $\mathcal{A}^R_\mH:\mathcal{P}_R\to\mathbb{C}$ by
 \begin{gather*}
  \mathcal{A}^R_\mH[\gamma]\coloneqq \int\limits^{t_2}_{t_1}\int\limits^{s_2}_{s_1}\left[\Lambda_R\vert_{\gamma (t+is)}\left(2\frac{\pa\gamma}{\pa z}(t+is)\right) - \mH\circ\gamma (t+is)\right] ds\, dt\ \text{with}\ \frac{\pa \gamma}{\pa z}\coloneqq \frac{1}{2}\left(\frac{\pa \gamma}{\pa t} - i\frac{\pa \gamma}{\pa s}\right)
 \end{gather*}
 for every $\gamma\in\mathcal{P}_R$. Now, let $\gamma\in\mathcal{P}_R$ be a smooth map from $R$ to $X$. Then, $\gamma$ is a holomorphic trajectory of $(X,\Omega,\mH)$ iff $\gamma$ is a ``critical point''\footnote{``Critical point'' means that only those variations are allowed which keep $\gamma$ fixed on the boundary $\partial R$.} of $\mathcal{A}^R_\mH$.
\end{Lem}

\begin{proof}
 Take the notations from above and decompose $\mH = \mH_R + i\mH_I$. Furthermore, let $\gamma\in\mathcal{P}_R$ be a smooth map and let $\gamma_s:[t_1,t_2]\to X$ and $\gamma_t:[s_1,s_2]\to X$ be defined by $\gamma_s (t) = \gamma (t+is) = \gamma_t (s)$ for any $s\in [s_1,s_2]$ and $t\in [t_1,t_2]$. Recall the action functionals from Proposition \autoref{prop:action_fct_of_real_traj}. A short calculation reveals that $\mathcal{A}^R_\mH [\gamma]$ can be expressed as
 \begin{gather*}
  \mathcal{A}^R_\mH [\gamma] = \int\limits^{s_2}_{s_1} \mathcal{A}^{\Lambda_R}_{\mH_R}[\gamma_s] ds - i\int\limits^{t_2}_{t_1} \mathcal{A}^{\Lambda_R}_{-\mH_I}[\gamma_t] dt.
 \end{gather*}
 $\gamma$ is a ``critical point'' of $\mathcal{A}^R_\mH$ iff $\gamma$ is a ``critical point'' of its real and imaginary part.\\
 Now consider the real part of $\mathcal{A}^R_\mH$. For any $s\in [s_1,s_2]$, $\gamma_s$ is a ``critical point'' of $\mathcal{A}^{\Lambda_R}_{\mH_R}$ iff $\gamma_s$ is an integral curve of $X^R_\mH$, where $X_\mH = 1/2(X^R_\mH - iJ(X^R_\mH))$ is the Hamiltonian vector field of $(X,\Omega,\mH)$. Explicitly writing down the first derivative of the functional $\gamma\mapsto\int \mathcal{A}^{\Lambda_R}_{\mH_R}[\gamma_s] ds$ shows that this property is preserved under averaging: $\gamma$ is a ``critical point'' of $\gamma\mapsto\int \mathcal{A}^{\Lambda_R}_{\mH_R}[\gamma_s] ds$ iff $\gamma_s$ is an integral curve of $X^R_\mH$ for every $s\in [s_1,s_2]$.\\
 Similarly, we find for the imaginary part of $\mathcal{A}^R_\mH$ that $\gamma$ is a ``critical point'' of\linebreak $\gamma\mapsto\int \mathcal{A}^{\Lambda_R}_{-\mH_I}[\gamma_t] dt$ iff $\gamma_t$ is an integral curve of $J(X^R_\mH)$ for every $t\in [t_1,t_2]$. Combining our results so far, we find that $\gamma$ is a ``critical point'' of $\mathcal{A}^R_\mH$ iff $\gamma_s$ is an integral curve of $X^R_\mH$ for every $s\in [s_1,s_2]$ and $\gamma_t$ is an integral curve of $J(X^R_\mH)$ for every $t\in [t_1,t_2]$. Recalling from \autoref{subsec:holo_traj} that holomorphic trajectories of $(X,\Omega,\mH)$ are exactly those smooth maps $\gamma$ that are integral curves of $X^R_\mH$ in $t$-direction and $J(X^R_\mH)$ in $s$-direction concludes the proof.
\end{proof}

To get a better understanding of Lemma \autoref{lem:holo_action_prin}, several remarks are in order:

\begin{Rem}\label{rem:several_remarks}
 {\textcolor{white}{Easter Egg}}
 \begin{enumerate}
  \item Note that the action functional $\mathcal{A}^R_\mH$ in Lemma \autoref{lem:holo_action_prin} only uses the real part $\Lambda_R$ and \underline{not} $\Lambda_I$. Of course, a similar action functional including $\Lambda_I$ exists, but we will not use it for reasons that become apparent in \autoref{sec:PHHS}.
  \item As before, one might wish to express holomorphic trajectories as actual critical points of some functional. Again, there are two main ways to achieve this: one may only consider smooth maps $\gamma$ from the rectangle $R$ to $X$ which\ldots
  \begin{enumerate}
   \item \dots map the boundary $\partial R$ to points in $X$ where the $1$-form $\Lambda_R$ vanishes, usually Lagrangian submanifolds of $X$.
   \item \dots are doubly-periodic, i.e., periodic in both $t$- and $s$-direction.
  \end{enumerate}
  Theoretically, one can even imagine a mix of both methods: one only considers maps $\gamma$ which are periodic in one direction and map the boundary orthogonal to the remaining direction to Lagrangian submanifolds of $X$.
  \item In more geometrical terms, the action $\mathcal{A}^R_\mH$ can be expressed as
  \begin{gather*}
   \mathcal{A}^R_\mH[\gamma] =  \iint\limits_{R}\left[\Lambda_R\vert_{\gamma (t+is)}\left(2\frac{\pa \gamma}{\pa z}(t+is)\right) - \mH\circ\gamma (t+is)\right] dt\wedge ds,
  \end{gather*}
  where $dt\wedge ds$ is the standard area form on $\mathbb{C}\cong\mathbb{R}^2$. If $\gamma$ is a ``critical point'' of $\mathcal{A}^R_\mH$ or simply a holomorphic curve, we find:
  \begin{gather*}
   \frac{\pa\gamma}{\pa z} (t+is) = \gamma^\prime (t+is) \in T^{(1,0)}X.
  \end{gather*}
  Using $\Lambda_R (V) = i\Lambda_I (V)$ for $V\in T^{(1,0)}X$, we obtain that the action at such $\gamma$ is given by
  \begin{gather*}
   \mathcal{A}^R_\mH[\gamma] =  \iint\limits_{R}\left[\Lambda\left(\gamma^\prime (t+is)\right) - \mH\circ\gamma (t+is)\right] dt\wedge ds,
  \end{gather*}
  where the expression in rectangular brackets is holomorphic in $z = t + is$.
  \item Upon closer inspection of Lemma \autoref{lem:holo_action_prin}, one might wonder whether Lemma \autoref{lem:holo_action_prin} is still true if one restricts the domain $\mathcal{P}_R$ of $\mathcal{A}^R_\mH$ to the holomorphic curves from $R$ to $X$ instead of varying over all smooth curves from $R$ to $X$. Clearly, this is not the case, as the values $\gamma$ attains at $\partial R$ completely determine one holomorphic curve, so variation over this space is not viable. A different perspective is offered by the action functional $\mathcal{A}^R_\mH$ itself. By writing $dt\wedge ds = i/2 \cdot dz\wedge d\bar{z}$ and recalling Point 3, $\mathcal{A}^R_\mH [\gamma]$ can be written as the integral of a form admitting a primitive for holomorphic $\gamma$. By Stokes' theorem, $\mathcal{A}^R_\mH[\gamma]$ then only depends on the values of $\gamma$ on the boundary $\partial R$. Since these values are kept fixed during variation, the action functional never changes in the variational process and gives us no information. An additional explanation for this behavior is presented in \autoref{app:action_functionals}.
 \end{enumerate}
\end{Rem}

We can define action functionals like $\mathcal{A}^R_\mH$ not only for rectangles, but for all kinds of domains in $\mathbb{C}$. A large selection of them is explored in \autoref{app:action_functionals}. Here, let us quickly introduce one generalization of $\mathcal{A}^R_\mH$, namely the action functional $\mathcal{A}^{P_\alpha}_\mH$ for parallelograms $P_\alpha$. For the sake of simplicity, we assume that the first vector spanning the parallelogram $P_\alpha$ is parallel to the real axis such that we can write $P_\alpha = [t_1,t_2] + e^{i\alpha}\cdot[r_1,r_2]$ for some angle $\alpha\in\mathbb{R}\backslash\{n\cdot\pi\mid n\in\mathbb{Z}\}$ and some real numbers $t_1 < t_2$ and $r_1 < r_2$. Using the standard area form $dt\wedge ds$, the generalization of $\mathcal{A}^R_\mH$ to $\mathcal{A}^{P_\alpha}_\mH$ is straightforward:
\begin{gather*}
 \mathcal{A}^{P_\alpha}_\mH[\gamma]\coloneqq \iint\limits_{P_\alpha}\left[\Lambda_R\vert_{\gamma (t+is)}\left(2\frac{\pa\gamma}{\pa z}(t+is)\right) - \mH\circ\gamma (t+is)\right] dt\wedge ds\quad\forall\gamma\in\mathcal{P}_{P_\alpha},
\end{gather*}
where we used the coordinates $z = t + is$ and defined $\partial\gamma/\partial z$ as in Lemma \autoref{lem:holo_action_prin}. To show that $\gamma$ is a ``critical point'' of $\mathcal{A}^{P_\alpha}_\mH$ iff $\gamma$ is a holomorphic trajectory, we express $\mathcal{A}^{P_\alpha}_\mH$ in the ``tilted'' coordinates $z = t + r\cdot e^{i\alpha}$ ($\alpha$ fixed):
\begin{align*}
 \mathcal{A}^{P_\alpha}_\mH[\gamma]&= \int\limits^{t_2}_{t_1}\int\limits^{r_2}_{r_1}\left[\Lambda_R\vert_{\gamma (t+re^{i\alpha})}\left(2\frac{\pa\gamma}{\pa z}(t+re^{i\alpha})\right) - \mH\circ\gamma (t+re^{i\alpha})\right]\, \sin(\alpha)\,dr\, dt\\
 &= \int\limits^{t_2}_{t_1}\int\limits^{r_2}_{r_1}\left[\Lambda_R\vert_{\gamma (t+re^{i\alpha})}\left(ie^{-i\alpha}\cdot \frac{d\gamma_r}{d t}(t) - i\cdot \frac{d\gamma_t}{dr}(r)\right)\right.\\
 &\qquad\quad \left.- \left(ie^{-i\alpha}\mH_R - i\text{Re}(e^{i\alpha}\mH)\right)\circ\gamma (t+re^{i\alpha})\right]\, dr\, dt\\
 &= ie^{-i\alpha}\int\limits^{r_2}_{r_1}\mathcal{A}^{\Lambda_R}_{\mH_R} [\gamma_r] dr - i\int\limits^{t_2}_{t_1}\text{Re} (\mathcal{A}^\Lambda_{e^{i\alpha}\mH} [\gamma_t]) dt,
\end{align*}
where, for any $t\in [t_1, t_2]$ and $r\in [r_1,r_2]$, the curves $\gamma_r:[t_1,t_2]\to X$ and $\gamma_t:[r_1,r_2]\to X$ are defined by $\gamma_r (t) = \gamma (t + re^{i\alpha}) = \gamma_t (r)$, $\text{Re}(\cdot)$ denotes the real part, and $\mathcal{A}^{\Lambda_R}_{\mH_R}$ and $\mathcal{A}^{\Lambda}_{e^{i\alpha}\mH}$ are the action functionals from Proposition \autoref{prop:action_fct_of_real_traj} and Remark \autoref{rem:tilted_traj}, respectively. For $\alpha\neq n\cdot \pi$, $n\in\mathbb{Z}$, the complex numbers $ie^{-i\alpha}$ and $-i$ form a $\mathbb{R}$-linear basis of $\mathbb{C}$. Thus, $\gamma$ is a ``critical point'' of $\mathcal{A}^{P_\alpha}_\mH$ iff $\gamma$ is a ``critical point'' of the functionals $\gamma\mapsto\int\mathcal{A}^{\Lambda_R}_{\mH_R} [\gamma_r] dr$ and $\gamma\mapsto\int\text{Re} (\mathcal{A}^\Lambda_{e^{i\alpha}\mH} [\gamma_t]) dt$. The rest now follows as in proof of Lemma \autoref{lem:holo_action_prin} by exploiting Proposition \autoref{prop:action_fct_of_real_traj}, Remark \autoref{rem:tilted_traj}, and the fact that holomorphic trajectories are exactly those smooth curves which are integral curves of $X^R_\mH$ in $t$-direction and integral curves of $\cos (\alpha) X^R_\mH + \sin (\alpha) J(X^R_\mH)$ in $r$-direction ($z = t + re^{i\alpha}$) for $\alpha\in\mathbb{R}\backslash\{n\cdot\pi\mid n\in\mathbb{Z}\}$. Summing up our results, we have just shown:

\begin{Prop}[Action principle for parallelograms]\label{prop:holo_action_prin_para}
 Let $(X,\Omega = d\Lambda,\mH)$ be an exact HHS with $\Lambda = \Lambda_R + i\Lambda_I$. For $\alpha\in\mathbb{R}\backslash\{n\cdot\pi\mid n\in\mathbb{Z}\}$, let $P_\alpha\coloneqq [t_1,t_2] + e^{i\alpha}[r_1,r_2]\subset\mathbb{C}$ be a parallelogram in the complex plane with real numbers $t_1< t_2$ and $r_1< r_2$. Denote the space of smooth maps from $P_\alpha$ to $X$ by $\mathcal{P}_{P_\alpha}$ and define the action functional $\mathcal{A}^{P_\alpha}_\mH:\mathcal{P}_{P_\alpha}\to\mathbb{C}$ by
\begin{gather*}
 \mathcal{A}^{P_\alpha}_\mH[\gamma]\coloneqq \iint\limits_{P_\alpha}\left[\Lambda_R\vert_{\gamma (t+is)}\left(2\frac{\pa\gamma}{\pa z}(t+is)\right) - \mH\circ\gamma (t+is)\right] dt\wedge ds\ \text{with}\ \frac{\pa\gamma}{\pa z}\coloneqq \frac{1}{2}\left(\frac{\pa\gamma}{\pa t} - i\frac{\pa\gamma}{\pa s}\right)
\end{gather*}
 for every $\gamma\in\mathcal{P}_{P_\alpha}$. Now, let $\gamma\in\mathcal{P}_{P_\alpha}$ be a smooth map from $P_\alpha$ to $X$. Then, $\gamma$ is a holomorphic trajectory of $(X,\Omega,\mH)$ iff $\gamma$ is a ``critical point''\footnote{Again, ``critical point'' means that only those variations are allowed which keep $\gamma$ fixed on the boundary $\partial P_\alpha$.} of $\mathcal{A}^{P_\alpha}_\mH$.
\end{Prop}

\begin{Rem}\label{rem:action_para}
 Proposition \autoref{prop:holo_action_prin_para} is a direct generalization of Lemma \autoref{lem:holo_action_prin}, since one obtains Lemma \autoref{lem:holo_action_prin} from Proposition \autoref{prop:holo_action_prin_para} by setting $\alpha = \pi/2$.
\end{Rem}

Before we conclude \autoref{subsec:holo_action_fun_and_prin}, let us inspect Point 2b) of Remark \autoref{rem:several_remarks} more closely. If a holomorphic curve $\gamma:P_\alpha\to X$ whose domain is a parallelogram $P_\alpha\subset\mathbb{C}$ is doubly-periodic, i.e., periodic in $t$- and $r$-direction for $z = t + re^{i\alpha}$, then we can also view $\gamma$ as a holomorphic map from a complex torus to $X$. In this sense, we can interpret holomorphic trajectories whose domains are complex tori as the holomorphic analogue to periodic orbits of RHSs. In contrast to periodic orbits of RHSs, however, the domains of two holomorphic periodic orbits do not need to be isomorphic. Indeed, the complex structure of such a torus is determined by the shape of the parallelogram $P_\alpha$. Therefore, the action functional $\mathcal{A}^{P_\alpha}_\mH$ is only sensitive to certain holomorphic periodic orbits, namely those whose domains share the complex structure induced by $P_\alpha$.\\
Non-constant holomorphic periodic orbits are rather rare and do not exist in most HHSs $(X,\Omega,\mH)$. For instance, take $X$ to be the standard example $\mathbb{C}^{2n}$. Due to the compactness of a complex torus $\mathbb{C}/\Gamma$, the maximum principle applies and any holomorphic map\linebreak $\gamma:\mathbb{C}/\Gamma\to X$ has to be constant. The same result applies if $X$ is Brody hyperbolic\footnote{A complex manifold $X$ is Brody hyperbolic iff every holomorphic map $\gamma:\mathbb{C}\to X$ defined on all of $\mathbb{C}$ is constant.}. Furthermore, if $X$ is compact, then all holomorphic trajectories are constant, since all Hamiltonians are constant by the maximum principle. Still, there are examples of HHSs $(X,\Omega,\mH)$ where a plethora of holomorphic periodic orbits exists.

\begin{Ex}[Natural Hamiltonians on complex tori $\mathbb{C}^n/\Gamma$]\label{ex:complex_torus}\normalfont
 Let $n\in\mathbb{N}$ be a natural number and $\Gamma\subset\mathbb{C}^n$ be a lattice, i.e.
 \begin{gather*}
  \Gamma\coloneqq \left\{\sum\limits^{2n}_{j=1} k_j\cdot e_j\middle| k_j\in\mathbb{Z}\right\},
 \end{gather*}
 where the vectors $e_1,\ldots, e_{2n}\in\mathbb{C}^n$ form an $\mathbb{R}$-linear basis of $\mathbb{C}^n$. Then, $\mathbb{C}^n/\Gamma$ is a complex torus of complex dimension $n$. Now consider the holomorphic cotangent bundle\linebreak $X\coloneqq T^{\ast, (1,0)}(\mathbb{C}^n/\Gamma)\cong \mathbb{C}^n/\Gamma\times\mathbb{C}^n$ with coordinates $([Q_1,\ldots, Q_j], P_1,\ldots, P_j)\in \mathbb{C}^n/\Gamma\times\mathbb{C}^n$ and canonical $2$-form $\Omega = \sum^n_{j=1} dP_j\wedge dQ_j$. We want to determine all natural Hamiltonians $\mH = \mathcal{T} + \mathcal{V}$ on the HSM $(X,\Omega)$. The potential energy $\mathcal{V}$ factors through a holomorphic function on $\mathbb{C}^n/\Gamma$. As $\mathbb{C}^n/\Gamma$ is compact, all holomorphic functions on it are constant due to the maximum principle. Since changing the Hamiltonian by a constant does not change the dynamics of the system, we can set the potential energy to zero without loss of generality. To compute the kinetic energy $\mathcal{T}$, we need to classify all holomorphic metrics $g$ on $\mathbb{C}^n/\Gamma$. The projection $\mathbb{C}^n\to\mathbb{C}^n/\Gamma$ gives rise to $n$ linearly independent holomorphic $1$-forms $dQ_j$, $j = 1,\ldots, n$, on the torus $\mathbb{C}^n/\Gamma$ which we have already used to express the canonical form $\Omega$. Using these $1$-forms, we can write $g$ as
 \begin{gather*}
  g = \sum^{2n}_{i,j = 1} g_{ij}dQ_i\otimes dQ_j,
 \end{gather*}
 where $g_{ij}$ are holomorphic functions on the torus. As before, these functions have to be constant implying that the space of holomorphic metrics $g$ on $\mathbb{C}^n/\Gamma$ is isomorphic to the space of symmetric and non-degenerate $\mathbb{C}$-bilinear forms on the complex vector space $\mathbb{C}^n$. By a standard result from linear algebra, every symmetric and non-degenerate $\mathbb{C}$-bilinear form on $\mathbb{C}^n$ can be brought into the standard form $g_{ij} = \delta_{ij}$\footnote{Here, $\delta_{ij}$ is the Kronecker delta!} by a $\mathbb{C}$-linear transformation. Hence, after transforming the lattice $\Gamma$ if necessary, we can assume that the metric $g$ is given by $g = \sum^{2n}_{j = 1} dQ^2_j$. In total, it suffices to investigate the dynamics of the HHS $(\mathbb{C}^n/\Gamma\times\mathbb{C}^n, \sum^n_{j=1} dP_j\wedge dQ_j, \mH)$ with Hamiltonian
 \begin{gather*}
  \mH (Q_1,\ldots, Q_n, P_1,\ldots, P_n) = \sum^{2n}_{j = 1} \frac{P^2_j}{2}
 \end{gather*}
 for all lattices $\Gamma\subset\mathbb{C}^n$ in order to study all natural Hamiltonians on a complex torus.\\
 Clearly, the Hamilton equations related to this problem are given by
 \begin{gather*}
  Q^\prime_j (z) = P_j (z);\quad P^\prime_j (z) = 0
 \end{gather*}
 and, given the initial value $\gamma (0) = ([Q^0_1,\ldots, Q^0_n], P^0_1,\ldots, P^0_n)$, obviously solved by the holomorphic trajectory $\gamma:\mathbb{C}\to X$,
 \begin{gather*}
  \gamma (z)\coloneqq ([Q^0_1 + z\cdot P^0_1,\ldots, Q^0_n + z\cdot P^0_n], P^0_1,\ldots, P^0_n).
 \end{gather*}
 Let us now define $P^0\coloneqq (P^0_1,\ldots, P^0_n)\in\mathbb{C}^n$ and consider different values for $P^0$:
 \begin{enumerate}
  \item If $P^0 = 0$, then $\gamma$ is simply a constant curve.
  \item If $P^0 \neq 0$ and $z\cdot P^0\notin\Gamma$ for every $z\in\mathbb{C}\backslash\{0\}$, then $\gamma$ is a regular holomorphic trajectory with no periodicity.
  \item If $P^0\neq 0$ and $z_1\cdot P^0\in\Gamma$ for at least one $z_1\neq 0$, then $\gamma$ is a regular holomorphic trajectory which is periodic in at least one direction. In this case, we can view $\gamma$ as a holomorphic map from a complex cylinder to $X$.
  \item If $P^0\neq 0$ and $z_1\cdot P^0, z_2\cdot P^0\in\Gamma$ for two $\mathbb{R}$-linearly independent complex numbers $z_1, z_2\in\mathbb{C}$, then $\gamma$ is a regular, doubly-periodic holomorphic trajectory. In this case, $\gamma$ is holomorphic periodic orbit.
 \end{enumerate}
 We observe that the topology and the complex structure of the domain of $\gamma$ changes depending on the momentum $P^0$.
\end{Ex}

\begin{Rem}[General Hamiltonians on a complex torus]\label{rem:ham_on_com_tor}
 As it turns out, Example \autoref{ex:complex_torus} covers all possible Hamiltonians on the HSM $(T^{\ast, (1,0)}(\mathbb{C}^n/\Gamma), \sum^{2n}_{j = 1} dP_j\wedge dQ_j)$. Let $\mH$ be any holomorphic function on $T^{\ast, (1,0)}(\mathbb{C}^n/\Gamma)$. Since $T^{\ast, (1,0)}(\mathbb{C}^n/\Gamma)$ is isomorphic to $\mathbb{C}^n/\Gamma\times\mathbb{C}^n$, $\mH$ cannot depend on the $Q_j$-coordinates due to maximum principle. This allows us to repeat the discussion from Example \autoref{ex:complex_torus} by replacing $P^0_j$ with $\partial\mH/\partial P_j (P^0)$ in the solution to the Hamilton equations.
\end{Rem}

\subsection{Application of HHSs: Lefschetz and Almost Toric Fibrations}
\label{subsec:Lefschetz}

One interesting aspect of HHSs is their interplay with two important structures in (symplectic) geometry: Lefschetz fibrations and almost toric fibrations. In this subsection, we briefly explore the connection between these structures. We begin the investigation by giving a short introduction to Lefschetz and almost toric fibrations.

\subsubsection*{Lefschetz Fibrations}

Let us first recall the definition of a Lefschetz fibration:

\begin{Def}[Lefschetz fibration]
 Let $X$ be a smooth $2m$-manifold and $C$ be a smooth $2$-manifold (both possibly with boundary). We call a smooth surjective map $\pi:X\to C$ a \textbf{Lefschetz fibration} iff the following three conditions are satisfied:
 \begin{enumerate}
  \item $\partial X = \pi^{-1} (\partial C)$
  \item All points on the boundary $\partial X$ are regular points of $\pi$.
  \item For every critical point $p\in\text{Crit}(\pi)\subset\text{Int}(X)$, there exists a smooth chart\linebreak $\psi_X:U_X\to V_X\subset\R^{2m}\cong \Cx^m$ near $p$ and a smooth chart $\psi_C:U_C\to V_C\subset\R^{2}\cong \Cx$ near $\pi (p)$ such that
   \begin{gather*}
    \psi_C\circ\pi\circ\psi_X^{-1} (z_1,\ldots,z_m) = \sum^m_{j=1}z^2_j.
   \end{gather*}
 \end{enumerate}
\end{Def}

Roughly speaking, a Lefschetz fibration generalizes the notion of a fiber bundle over a surface where we now also allow singular fibers. This aspect is captured by the following proposition:

\begin{Prop}[Lefschetz fibrations as fiber bundles]\label{prop:Lef_fiber_bundles}
 Let $\pi:X\to C$ be a Lefschetz fibration and $C^\ast$ be the set of regular values of $\pi$. Further, assume that $X$ is connected. If $\pi:\pi^{-1}(C^\ast)\to C^\ast$ is proper, then $\pi^{-1}(C^\ast)\stackrel{\pi}{\to} C^\ast$ is a fiber bundle. In particular, if $X$ (and then also $C$) is compact, $\pi^{-1}(C^\ast)\stackrel{\pi}{\to} C^\ast$ is a fiber bundle.
\end{Prop}

\begin{proof}
 This follows directly from the fact\footnote{Due to Ehresmann, cf. \cite{Ehresmann1952}.} that every smooth, surjective, and proper submersion between connected manifolds is a fiber bundle.
\end{proof}

Most authors include additional conditions in the definition of a Lefschetz fibration. Often, $X$ is assumed to be a compact, connected, and oriented fourfold. On one hand, this has historic reasons: While Lefschetz introduced these fibrations to study the topology of complex surfaces, Donaldson and Gompf brought Lefschetz fibrations to the attention of the symplectic community by showing that every compact symplectic fourfold admits the structure of a Lefschetz fibration (after blowing up if necessary) and vice versa\footnote{Under mild conditions.}. For a comprehensive overview of the history of fourfolds and Lefschetz fibrations, confer the introduction in \cite{Fuller2003}.\\
On the other hand, the case $m=2$ has a rich structure and is well understood: If $X$ is a closed fourfold, then the regular fibers of $\pi:X\to C$ are closed surfaces. Under mild conditions\footnote{Both $X$ and $C$ are oriented and connected and, additionally, $C$ is simply connected, cf. \cite{Naylor2016}.}, the regular fibers of $\pi:X\to C$ are even oriented and connected, hence, are surfaces of genus $g$. In this case, the singular fibers are pinched surfaces of genus $g$ (cf. \cite{Naylor2016}). In our investigation, the case $m=2$ also plays an important role.

\subsubsection*{Almost Toric Fibrations}

Almost toric fibrations generalize the notion of toric fibrations which themselves can be understood as the momentum map of an effective Hamiltonian torus action. To capture this idea, let us recall the famous convexity theorem of Atiyah, Guillemin, and Sternberg (confer, for instance, \cite{Audin2004} and \cite{Symington2002}):

\begin{Thm}[Convexity theorem]
 Let $(X^{2m},\omega)$ be a connected and closed symplectic manifold with an effective Hamiltonian $T^m$-action on it. Then, the image of the associated momentum map $\pi: X\to\R^m$ is a convex polytope.
\end{Thm}

The polytope has a natural stratification. The highest dimensional stratum, i.e., its interior is the set of regular values of $\pi$. The regular level sets are Lagrangian tori and one can express $\omega$ and $\pi$ in a neighborhood of any regular point as\footnote{This is just the equally famous Arnold-Liouville theorem.}
\begin{gather*}
 \omega = \sum^m_{j=1} dx_j\wedge dy_j;\quad \pi_j = x_j.
\end{gather*}
For $k<m$, the points on the $k$-dimensional stratum are critical values. Similarly as before, one can show that in a neighborhood of such a critical point we can write (cf. \cite{Symington2002} and \cite{Leung2003}):
\begin{gather}
 \omega = \sum^m_{j=1} dx_j\wedge dy_j;\quad \pi_j = x_j\text{ for }1\leq j\leq k;\quad \pi_j = x^2_j + y^2_j\text{ for }k<j\leq m.\label{eq:toric}
\end{gather}
We see that a toric fibration is a symplectic manifold viewed as a fiber bundle whose regular fibers are Lagrangian tori and whose singular fibers take the local form as described by \autoref{eq:toric}.\\
Almost toric fibrations, introduced by M. Symington in \cite{Symington2002}, exemplify a similar structure, but the local description of their singular fibers is broadened. The following definitions are taken from \cite{Symington2002} and \cite{Leung2003}:

\begin{Def}[Lagrangian and almost toric fibrations]\label{def:alm_tor}
 Let $C^m$ be a smooth and $(X^{2m},\omega)$ be a symplectic manifold (both possibly with boundary). Further, let $\pi:X\to C$ be a smooth and surjective map. We call $\pi$ a \textbf{Lagrangian fibration} of $(X,\omega)$ iff there exists an open and dense subset $C^\ast\subset C$ such that $\pi^{-1}(C^\ast)\stackrel{\pi}{\to} C^\ast$ is a fiber bundle with Lagrangian fibers. We call $\pi$ an \textbf{almost toric fibration} iff $\pi$ is a Lagrangian fibration and every critical point of $\pi$ has a neighborhood in which $\omega$ and $\pi$, after choosing charts, take the following form ($0\leq k<m$):
 \begin{enumerate}
  \item $\omega = \sum^m_{j=1} dx_j\wedge dy_j$,
  \item $\pi_j = x_j$ for $1\leq j\leq k$,
  \item $\pi_j = x^2_j + y^2_j$ or $(\pi_j, \pi_{j+1}) = (x_jy_j + x_{j+1}y_{j+1}, x_jy_{j+1} - x_{j+1}y_j)$ for $k<j\leq m$.
 \end{enumerate}
\end{Def}

\begin{Rem}[Fibers of almost toric fibrations]\label{rem:toric_fibers}
 If the regular fibers of an almost toric fibration are compact and connected, then they are diffeomorphic to a torus by the Arnold-Liouville theorem, explaining the name ``almost toric fibration''.
\end{Rem}

Let us now investigate the relation between Lefschetz and almost toric fibrations using HHSs. The idea is that we define a compatible holomorphic symplectic structure on a Lefschetz fibration. Such a fibration can be interpreted as a HHS. At the same time, a holomorphic symplectic Lefschetz fibration in real dimension four is, under mild conditions, an almost toric fibration, as we will show later on.\\
First, we formulate the notion of a Lefschetz fibration in a holomorphic setup:

\begin{Def}[Holomorphic Lefschetz fibration]
 Let $X$ be a complex $m$-dimensional and $C$ be a complex one-dimensional manifold (both possibly with boundary\footnote{A complex manifold with boundary is defined similarly to a real manifold with boundary, even though there are subtle differences.}). We say that a surjective holomorphic map $\pi:X\to C$ is a \textbf{holomorphic Lefschetz fibration} iff the following three conditions are satisfied:
 \begin{enumerate}
  \item $\partial X = \pi^{-1} (\partial C)$
  \item All points on the boundary $\partial X$ are regular points of $\pi$.
  \item Every critical point $p\in\text{Crit}(\pi)\subset\text{Int}(X)$ is non-degenerate.
 \end{enumerate}
\end{Def}

Here, a critical point $p\in\text{Crit}(\pi)$ is called non-degenerate iff its (complex) Hessian\footnote{Since $p$ is a critical point, the non-degeneracy of the Hessian is well-defined, i.e., independent of the choice of charts.} at $p$ is non-degenerate in some holomorphic charts of $X$ and $C$. As the name implies, a holomorphic Lefschetz fibration is also a (smooth) Lefschetz fibration:

\begin{Prop}[Holomorphic Lefschetz fibrations are Lefschetz fibrations]
 Let\linebreak $\pi:X\to C$ be a holomorphic Lefschetz fibration. Then $\pi:X\to C$ is a (smooth) Lefschetz fibration in the usual sense. In particular, the charts $\psi_X$ and $\psi_C$ can be chosen to be holomorphic.
\end{Prop}

\begin{proof}
 This is a direct consequence of the holomorphic Morse lemma.
\end{proof}

\begin{Lem}[Holomorphic Morse lemma]
 Let $X$ be a complex manifold, $f:X\to\Cx$ be a holomorphic function, and $p\in X$ be a non-degenerate critical point of $f$. Then, there exists a holomorphic chart $\psi_X:U_X\to V_X\subset\Cx^m$ of $X$ near $p$ with $\psi_X (p) = 0$ such that:
 \begin{gather*}
  f\circ\psi^{-1}_X (z_1,\ldots, z_m) = f(p) + \sum^m_{j=1} z^2_j\quad\forall (z_1,\ldots, z_m)\in V_X.
 \end{gather*}
\end{Lem}

\begin{Rem}[Signature of Hessian]
 In contrast to the real Morse lemma, we do not need to care about the signature of the (complex) Hessian in the holomorphic Morse lemma, since all non-degenerate symmetric $\Cx$-bilinear forms on $\Cx^m$ are isomorphic.
\end{Rem}

\begin{proof}
 The holomorphic Morse lemma can be shown in the same way as the usual Morse lemma (confer, for instance, the proof in \cite{Audin2014} on page 12 ff.), where, of course, we use the appropriate theorems from complex analysis instead of their smooth counterparts\footnote{For instance, the holomorphic implicit function theorem instead of the implicit function theorem.}.
\end{proof}

Now, we put a compatible symplectic structure on a holomorphic Lefschetz fibration:

\begin{Def}[Holomorphic symplectic Lefschetz fibration]
 We call a holomorphic Lefschetz fibration $\pi:X\to C$ \textbf{symplectic} iff $X$ admits the structure of a HSM $(X,\Omega)$ such that every critical point $p\in\text{Crit}(\pi)$ has holomorphic Morse Darboux charts near it, i.e., there are holomorphic charts $\psi_X = (z_1,\ldots, z_{2n}):U_X\to V_X\subset\Cx^{2n}$ of $X$ near $p$ and $\psi_C:U_C\to V_C\subset\Cx$ of $C$ near $\pi (p)$ satisfying:
 \begin{enumerate}
  \item $\Omega\vert_{U_X} = \sum^n_{j=1} dz_{j+n}\wedge dz_j$,
  \item $\psi_C\circ\pi\vert_{U_X} = \sum^{2n}_{j=1} z^2_j$.
 \end{enumerate}
\end{Def}

If $\pi:X\to C$ is a holomorphic symplectic Lefschetz fibration with underlying HSM $(X,\Omega)$, we can locally interpret $(X,\Omega,\pi)$ as a HHS after choosing a holomorphic chart $\psi_C$ of $C$. Of course, the Hamiltonian vector field of this HHS is not well-defined, since the Hamilton function depends on the choice of $\psi_C$. However, two Hamiltonian vector fields only differ by a holomorphic function:
\begin{gather*}
 X_{\hat \psi_C\circ\pi}\vert_p = (\hat\psi_C\circ\psi^{-1}_C)^\prime (\psi_C\circ\pi(p))\cdot X_{\psi_C\circ\pi}\vert_p\quad\text{for }p\in X.
\end{gather*}
Hence, the two Hamiltonian vector fields have the same trajectories, just parameterized differently. This implies that the holomorphic foliation by the regular level sets of the HHS is still well-defined, even though the Hamiltonian vector field is not. In particular, if $X$ is complex two-dimensional or, equivalently, real four-dimensional and the level sets of $\pi:X\to C$ are connected, then the leaves of this foliation are just the regular level sets themselves allowing us to interpret a holomorphic symplectic Lefschetz fibration as the holomorphic foliation of a HHS.\\
The next step is to link holomorphic symplectic Lefschetz fibrations to almost toric fibrations. We show that every holomorphic symplectic Lefschetz fibration in real dimension four is also an almost toric fibration if $\pi$ is proper.

\begin{Prop}[Holomorphic symplectic Lefschetz fibrations in real dimension four]\label{prop:holo_lef_toric}
 Let $\pi:X\to C$ be a holomorphic symplectic Lefschetz fibration with underlying HSM $(X,\Omega = \Omega_R + i\Omega_I)$. Assume that $X$ is connected and has real dimension four. If $\pi:X\to C$ is proper, then $\pi:X\to C$ is an almost toric fibration of $(X,\Omega_R)$. In particular, if $X$ is compact, then $\pi:X\to C$ is an almost toric fibration of $(X,\Omega_R)$.
\end{Prop}

\begin{proof}
 The set $C^\ast$ of regular values of $\pi$ is open and dense in $C$. By the same argument as in the proof of Proposition \autoref{prop:Lef_fiber_bundles}, $\pi^{-1}(C^\ast)\stackrel{\pi}{\to}C^\ast$ is a fiber bundle. We now consider the fibers of $\pi^{-1}(C^\ast)\stackrel{\pi}{\to}C^\ast$. The HHS $(X,\Omega,\pi)$ is complex two-dimensional, hence, the regular level sets of $\pi$ are complex Lagrangian submanifolds of $(X,\Omega)$. Thus, they are also Lagrangian submanifolds of $(X,\Omega_R)$, as $\Omega = \Omega_R + i\Omega_I$ is a holomorphic $2$-form. This implies that $\pi:X\to C$ is a Lagrangian fibration of $(X,\Omega_R)$.\\
 \pagebreak
 To conclude the proof, we need to consider a critical point $p$ of $\pi$. By definition, there are holomorphic charts $\psi_X = (z_1, z_2):U_X\to V_X\subset\Cx^2$ of $X$ near $p$ and $\psi_C:U_C\to V_C\subset\Cx$ of $C$ near $\pi (p)$ satisfying:
 \begin{align*}
  \Omega\vert_{U_X} &= dz_2\wedge dz_1,\\
  \psi_C\circ\pi\vert_{U_X} &= z^2_1 + z^2_2.
 \end{align*}
 In the new coordinates $z_1\eqqcolon (\hat z_2 - i\hat z_1)/\sqrt{2}$, $z_2\eqqcolon (\hat z_1 - i\hat z_2)/\sqrt{2}$, and $\hat \psi_C \coloneqq -\psi_C/2i$, we find:
 \begin{alignat*}{4}
  \Omega\vert_{\hat U_X} &= \, d\hat z_1\wedge d\hat z_2 &&= &&(d\hat x_1\wedge d\hat x_2 - d\hat y_1\wedge d\hat y_2) &&+ i(d\hat x_1\wedge d\hat y_2 + d\hat y_1\wedge d\hat x_2) ,\\
  \hat \psi_C\circ\pi\vert_{\hat U_X} &= \hat z_1\hat z_2 &&= &&(\hat x_1 \hat x_2 - \hat y_1\hat y_2) &&+ i(\hat x_1\hat y_2 + \hat y_1 \hat x_2),
 \end{alignat*}
 where we have used the decomposition $\hat z_j = \hat x_j + i\hat y_j$. In particular, we have\linebreak $\Omega_R\vert_{\hat U_X} = d\hat x_1\wedge d\hat x_2 - d\hat y_1\wedge d\hat y_2$ in these coordinates. Setting $x_1 \coloneqq \hat x_1$, $x_2\coloneqq -\hat y_1$, $y_1\coloneqq \hat x_2$, $y_2\coloneqq \hat y_2$, $\pi_1 \equiv \text{Re}(\hat \psi_C\circ\pi\vert_{\hat U_X})$, and $\pi_2 \equiv \text{Im}(\hat \psi_C\circ\pi\vert_{\hat U_X})$ reproduces the local structure near a critical point as in the definition of an almost toric fibration.
\end{proof}

\begin{Rem}[Proposition \autoref{prop:holo_lef_toric} for $(X,\Omega_I)$]
 With the same assumptions as in Proposition \autoref{prop:holo_lef_toric}, we find that $\pi:X\to C$ is also an almost toric fibration of $(X,\Omega_I)$ if $\pi:X\to C$ is proper. In that regard, a holomorphic symplectic Lefschetz fibration gives rise to two different almost toric fibrations.
\end{Rem}

\begin{Rem}[Critical points are non-toric]
 Observe that the critical points of a holomorphic symplectic Lefschetz fibration are non-toric, i.e., their local description matches the non-toric case in Definition \autoref{def:alm_tor}.
\end{Rem}

One might wonder how many manifolds $X$ Proposition \autoref{prop:holo_lef_toric} is applicable to. In the case that $X$ is closed, the answer is already known: Leung and Symington classified in \cite{Leung2003} all closed almost toric four-folds up to diffeomorphisms. They have shown that the only examples which are locally Lefschetz, i.e., whose critical points are non-toric are the K3 surface with base $C = S^2$ and its $\mathbb{Z}_2$-quotient, i.e., the Enriques surface with base $C = \R P^2$. Since $\R P^2$ is not orientable, it cannot admit a complex structure and, thus, the Enriques surface cannot be a holomorphic symplectic Lefschetz fibration. Hence, the only possible example of a closed holomorphic symplectic Lefschetz fibration in four dimensions is the K3 surface\footnote{By Remark \autoref{rem:toric_fibers}, the regular fibers of the K3 surface are tori.}.\\
To conclude this section, we generalize the last question and ask ourselves whether there is an obstruction for a holomorphic Lefschetz fibration to be symplectic. Of course, the space $X$ of a holomorphic Lefschetz fibration $\pi:X\to C$ needs to admit the structure of a HSM in order for $\pi:X\to C$ to be symplectic. This itself is a non-trivial condition. But even if $X$ is a HSM, it is not clear whether the Lefschetz fibration admits Morse Darboux charts near critical points. This question is especially interesting, since we already know that every HSM admits Darboux charts near any point and every holomorphic Lefschetz fibration admits Morse charts near critical points.\\
In the complex case, this is a non-trivial problem and not completely understood\footnote{At least not by the author.}. To gain some intuition, we consider the real case in dimension two instead.

\newpage

\textbf{Setting:} Let $M$ be a smooth manifold of real dimension two with symplectic form $\omega$ on it, $L$ be a smooth manifold of real dimension one, $f\in C^\infty (M,L)$, and $p\in M$ be a non-degenerate\footnote{Confer Definition \autoref{def:morse_index}.} critical point of $f$.\\

\textbf{Question:} Do smooth charts $\psi_M = (x,y):U_M\to V_M\subset\R^2$ of $M$ near $p$ and\linebreak $\psi_L:U_L\to V_L\subset\R$ of $L$ near $f(p)$ exist such that
\begin{enumerate}
 \item $\omega\vert_{U_M} = dx\wedge dy$,
 \item $\psi_L\circ f\vert_{U_M} = x^2 + y^2$?
\end{enumerate}

Of course, the answer to the question in full generality is negative due to the Morse index of $p$ which is absent in the complex case:

\begin{Def}[Morse index $\mu_f (p)$]\label{def:morse_index}
 Let $M$ be a smooth $m$-manifold, $L$ be a smooth $1$-manifold, $f\in C^\infty (M,L)$, and $p\in M$ be a critical point of $f$. We say that $p$ is non-degenerate iff the Hessian of $\psi_L\circ f\circ \psi^{-1}_M$ at $\psi_M (p)$ is non-degenerate for some charts $\psi_M$ of $M$ near $p$ and $\psi_L$ of $L$ near $f(p)$. The \textbf{Morse index} $\mu_f (p)$ is the number $\min\{k, m-k\}$ where $k$ is the usual Morse index of $\psi_L\circ f\circ \psi^{-1}_M$ at $\psi_M (p)$ for charts $\psi_M$ and $\psi_L$, i.e, the number of negative eigenvalues of its Hessian.
\end{Def}

\begin{Rem*}
 Even though the non-degeneracy of the Hessian is independent of the choice of charts, the number of negative eigenvalues of the Hessian is not. A orientation reversing transformation of the chart $\psi_L$, for instance $\psi_L\mapsto -\psi_L$, changes the number from $k$ to $m-k$, explaining the definition of the Morse index.
\end{Rem*}

The Morse index of $(x,y)\mapsto x^2 + y^2$ is $0$. Since the Morse index is invariant under change of charts, a necessary condition for the existence of Morse Darboux charts near $p$ is $\mu_f (p) = 0$. The question remains whether this condition is sufficient. Aside from regularity issues, this seems to be the case. To make this precise, consider the following two lemmata which are proven in \autoref{app:morse_darboux_lem}:

\begin{Lem}[Morse Darboux lemma I]\label{lem:morse_darboux_lem_I}
 Let $(M^2,\omega)$ be a symplectic $2$-manifold, $L^1$ be a smooth $1$-manifold, $f\in C^\infty (M,L)$, and $p\in M$ be a non-degenerate critical point of $f$ with Morse index $\mu_f (p) = 0$. Further, let $T>0$ be a positive real number. Then, there exists a $C^1$-chart $\psi_L:U_L\to V_L\subset\R$ of $L$ near $f(p)$ which is smooth on $U_L\backslash\{f(p)\}$ such that all non-constant trajectories near $p$ of the RHS $(U_M, \omega\vert_{U_M}, H)$ with $U_M\coloneqq f^{-1}(U_L)$ and $H\coloneqq \psi_L\circ f\vert_{U_M}$ are $T$-periodic.
\end{Lem}

\newpage
% 
% \begin{Lem}[Morse Darboux lemma II]\label{lem:morse_darboux_lem_II}
%  Let $(M^2,\omega)$ be a symplectic manifold and let\linebreak $H\in C^\infty (M,\R)$ be a smooth function on $M$ with non-degenerate critical point $p\in M$. Further, let $T>0$ be a positive real number. Then, the following statements are equivalent:
%  \begin{enumerate}
%   \item There exists a topological chart $\psi_M = (x,y):U_M\to V_M\subset\R^2$ of $M$ near $p$ which is smooth on $U_M\backslash\{p\}$ such that ($\psi_M (p) = 0$):
%   \begin{enumerate}[label = (\alph*)]
%    \item $H\vert_{U_M} = H(p) \pm \frac{\pi}{T}(x^2 + y^2)$,
%    \item $\omega\vert_{U_M} = dx\wedge dy$.
%   \end{enumerate}
%   \item There exists an open neighborhood $U_M\subset M$ of $p$ such that:
%   \begin{enumerate}[label = (\alph*)]
%    \item For every metric $d$ on $M$ compatible with its topology and every $\varepsilon >0$, there is a non-constant trajectory $\gamma$ of the RHS $(U_M, \omega\vert_{U_M}, H\vert_{U_M})$ such that $\sup_{t\in\R} d(\gamma (t), p) <\varepsilon$.
%    \item All non-constant trajectories of the RHS $(U_M, \omega\vert_{U_M}, H\vert_{U_M})$ are $T$-periodic.
%   \end{enumerate}
%   \item There exists a number $E_0 > 0$ such that:
%   \begin{enumerate}[label = (\alph*)]
%    \item $\mu_H (p) \neq 1$,
%    \item $\int_{U(E)} \omega = T\cdot E$ for every $E\in [0, E_0]$, where $U(E)$ is the connected component containing $p$ of the set $\{q\in M\mid |H(q)-H(p)|\leq E\}$.
%   \end{enumerate}
%  \end{enumerate}
% \end{Lem}

\begin{Lem}[Morse Darboux lemma II]\label{lem:morse_darboux_lem_II}
 Let $(M^2,\omega)$ be a symplectic $2$-manifold and let\linebreak $H\in C^\infty (M,\R)$ be a smooth function on $M$ with non-degenerate critical point $p\in M$ of Morse index $\mu_H (p)\neq 1$. Further, let $T>0$ be a positive real number. Then, the following statements are equivalent:
 \begin{enumerate}
  \item There exists a topological chart $\psi_M = (x,y):U_M\to V_M\subset\R^2$ of $M$ near $p$ which is smooth on $U_M\backslash\{p\}$ such that ($\psi_M (p) = 0$):
  \begin{enumerate}[label = (\alph*)]
   \item $H\vert_{U_M} = H(p) \pm \frac{\pi}{T}(x^2 + y^2)$,
   \item $\omega\vert_{U_M} = dx\wedge dy$.
  \end{enumerate}
  \item There exists an open neighborhood $U_M\subset M$ of $p$ such that all non-constant trajectories of the RHS $(U_M, \omega\vert_{U_M}, H\vert_{U_M})$ are $T$-periodic.
  \item There exists a number $E_0 > 0$ such that $\int_{U(E)} \omega = T\cdot E$ for every $E\in [0, E_0]$,\linebreak where $U(E)$ is the connected component containing $p$ of the set\linebreak $\{q\in M\mid |H(q)-H(p)|\leq E\}$.
 \end{enumerate}
\end{Lem}

% \begin{Rem*}
%  Condition (a) in Statement 2 and 3 of Lemma \autoref{lem:morse_darboux_lem_II} can be omitted if we include $\mu_H (p)\neq 1$ in the assumptions. We worded Lemma \autoref{lem:morse_darboux_lem_II} this way so that the equivalence of Statement 1 and 2 has a (weaker) holomorphic analogue (cf. Lemma \autoref{lem:holo_morse_darboux_lem_II}).
% \end{Rem*}

With Lemma \autoref{lem:morse_darboux_lem_I} and \autoref{lem:morse_darboux_lem_II} in mind, the rough idea to show the existence of Morse Darboux charts is clear: First, we use Lemma \autoref{lem:morse_darboux_lem_I} to find a chart $\psi_L$ in which all trajectories near $p$ have period $T = \pi$ and the usual Morse index of $\psi_L\circ f$ is $0$. Afterwards, we wish to apply Lemma \autoref{lem:morse_darboux_lem_II} to find a chart $\psi_M$ in which $\omega$ and $H = \psi_L\circ f$ assume their respective standard form. However, the regularity issues come into play here: We cannot apply Lemma \autoref{lem:morse_darboux_lem_II}\footnote{Note, however, that we \underline{can} apply Lemma \autoref{lem:morse_darboux_lem_II} after Lemma \autoref{lem:morse_darboux_lem_I} if all objects in Lemma \autoref{lem:morse_darboux_lem_I} are real analytic, as the chart $\psi_L$ and the Hamilton function $\psi_L\circ f$ are real analytic in this case. Confer Remark \autoref{rem:no_reg_issue} in \autoref{app:morse_darboux_lem} for details.}, since the Hamilton function $\psi_L\circ f$ is, in general, only $C^1$. This is troublesome for several reasons, not the least of which is that the notion of a Morse index does not even make sense for $C^1$-functions.\\
At this point, it is clear that a deeper analysis of this problem is in order, especially if one aims to also discuss the complex case. We stop this excursion here and move on to the main part of this paper.

\newpage
\section{Pseudo-Holomorphic Hamiltonian Systems}
\label{sec:PHHS}

In \autoref{subsec:holo_action_fun_and_prin}, we have seen that the existence of non-constant holomorphic periodic orbits is forbidden for a large class of HHSs $(X,\Omega, \mH)$. One possible obstruction to the existence of such orbits is the integrability of the almost complex structure $J$ of $X$. In this section, we drop the integrability of $J$ leading us to the notion of a pseudo-holomorphic Hamiltonian system (PHHS). We demonstrate in \autoref{subsec:def_PHHS} that PHHSs exhibit, by design, (almost) the same properties as found for HHSs in \autoref{sec:HHS}, in particular with respect to the existence and uniqueness of pseudo-holomorphic trajectories and with respect to action functionals and principles. In \autoref{subsec:rel_HHS_PHHS}, we explore the relation between HHSs and PHHSs and show that we recover a HHS from a PHHS if we restore the integrability of its almost complex structure $J$.

\subsection{PHHS: Basic Definitions, Notions, and Properties}
\label{subsec:def_PHHS}

In \autoref{subsec:holo_action_fun_and_prin}, we have found that most HHSs $(X,\Omega,\mH)$ do not possess non-constant holomorphic periodic orbits. Often, their existence was forbidden by the maximum principle. For instance, consider $X = \mathbb{R}^{4}\cong \mathbb{C}^2$ with standard complex structure $J = i$. As holomorphic periodic orbits are holomorphic maps $\gamma:\mathbb{C}/\Gamma\to X$ and a complex torus $\mathbb{C}/\Gamma$ is compact, the maximum principle applies and every holomorphic periodic orbit in $\mathbb{C}^2$ is constant.\\
In his beautiful paper \cite{moser1995} from 1995, Moser showed that the same argument does \underline{not} apply if we equip $\mathbb{R}^4$ with a different almost complex structure $J$. Let $\gamma$ be any smooth embedding of the $2$-torus into $\mathbb{R}^4$, e.g. the inclusion $S^1\times S^1 \subset \mathbb{R}^2\times\mathbb{R}^2\equiv \mathbb{R}^4$. Then, the image of $\gamma$ is a $2$-dimensional submanifold of $\mathbb{R}^4$ and its tangent bundle can be continued to a smooth $2$-dimensional distribution $D$ on $\mathbb{R}^4$. This can be seen as follows: the tangent bundle of $S^1\times S^1\subset\mathbb{R}^4$ is spanned by the two vector fields $V_1$ and $V_2$ on $S^1\times S^1$:
\begin{align*}
 V_1:&S^1\times S^1\to\mathbb{R}^4,\ (x_1,x_2,x_3,x_4)\mapsto (-x_2,x_1,0,0);\\ V_2:&S^1\times S^1\to\mathbb{R}^4,\ (x_1,x_2,x_3,x_4)\mapsto (0,0, -x_4, x_3).
\end{align*}
We show that there are two linearly independent vector fields $\hat V_1$ and $\hat V_2$ on $\mathbb{R}^4$ continuing $V_1$ and $V_2$, i.e., $\hat V_i\vert_{S^1\times S^1} = V_i$. To construct $\hat V_1$ and $\hat V_2$, we first define the functions $r_1\coloneqq \sqrt{x^2_1 + x^2_2}$, $r_2\coloneqq \sqrt{x^2_3 + x^2_4}$, and $R\coloneqq (1-r^2_1)^2 + (1-r^2_2)^2$. Next, define the vector fields $\hat V_1$, $W_1$, and $W_2$ on $\mathbb{R}^4$ as follows:
\begin{align*}
 \hat V_1:& \mathbb{R}^4\to\mathbb{R}^4,\ (x_1,x_2,x_3,x_4)\mapsto (-x_2,x_1,R,0);\\
 W_1:& \mathbb{R}^4\to\mathbb{R}^4,\ (x_1,x_2,x_3,x_4)\mapsto (-x_2x_4R,\, x_1x_4R,\, -r^2_1x_4,\, r^2_1x_3);\\
 W_2:& \mathbb{R}^4\to\mathbb{R}^4,\ (x_1,x_2,x_3,x_4)\mapsto (x_1,x_2,0,R).
\end{align*}
One easily checks that $\hat V_1$ is a continuation of $V_1$, vanishes nowhere, and is orthogonal to $W_1$ and $W_2$ with respect to the standard metric on $\mathbb{R}^4$. Furthermore, we notice that $W_1$ is a continuation of $V_2$. However, $W_1$ vanishes for $x_1 = x_2 = 0$ or $x_3 = x_4 = 0$. To rectify this, we take an appropriate linear combination of $W_1$ and $W_2$. For that, we first observe that $W_1$ does not vanish on $S^1\times S^1$. Hence, we can pick an open neighborhood $U\subset\mathbb{R}^4$ of $S^1\times S^1$ such that $W_1$ does not vanish on $U$. Next, we choose a partition of unity $\{f_1, f_2\}$ on $\mathbb{R}^4$ subordinate to the open covering $\{U, \mathbb{R}^4\backslash (S^1\times S^1)\}$ of $\mathbb{R}^4$, i.e, two smooth functions $f_1,f_2\in C^\infty (\mathbb{R}^4,\mathbb{R}_{\geq 0})$ satisfying:
\begin{enumerate}
 \item $f_1(x) + f_2(x) = 1\quad\forall x\in\mathbb{R}^4$,
 \item $\text{supp}(f_1)\subset U$ and $\text{supp}(f_2)\subset\mathbb{R}^4\backslash (S^1\times S^1)$.
\end{enumerate}
Now define the vector field $\hat V_2$ by $\hat V_2\coloneqq f_1\cdot W_1 + f_2\cdot W_2$. By construction, the vector field $\hat V_2$ is a continuation of $V_2$. Moreover, one can show that $\hat V_2$ vanishes nowhere by considering $\hat V_2$ separately on $S^1\times S^1$, $U\backslash (S^1\times S^1)$, and $\mathbb{R}^4\backslash U$. As $\hat V_1$ is orthogonal to $W_1$ and $W_2$, $\hat V_1$ is also orthogonal to the vector field $\hat V_2$. Two orthogonal vector fields which vanish nowhere are linearly independent, hence, the vector fields $\hat V_1$ and $\hat V_2$ are the desired continuations of $V_1$ and $V_2$. The distribution $D$ is now just the span of $\hat V_1$ and $\hat V_2$ at any point $(x_1,x_2,x_3,x_4)\in\mathbb{R}^4$.\\
Return to Moser's construction. Choose a Riemannian metric $g$ on $\mathbb{R}^4$ and consider\linebreak the orthogonal complement $D^\perp$ of $D$ with respect to $g$. $D^\perp$ is also a smooth\linebreak $2$-dimensional distribution on $\mathbb{R}^4$. Moreover, $D$ and $D^\perp$ span the tangent bundle of $\mathbb{R}^4$,\linebreak $T\mathbb{R}^4 = D\oplus D^\perp$. Now we can construct the almost complex structure $J$ on $\mathbb{R}^4$ as follows: choose orientations for $D$ and $D^\perp$ and define $J$ to be the $90^{\circ}$-rotation in $D$ and $D^\perp$ with respect to $g$ and the given orientations. After choosing a suitable complex structure $j$ on the $2$-torus, $\gamma$ becomes a pseudo-holomorphic\footnote{A smooth map $f:(X_1,J_1)\to (X_2, J_2)$ between smooth manifolds $X_1$ and $X_2$ with almost complex structures $J_1$ and $J_2$ is called \textbf{pseudo-holomorphic} iff $df\circ J_1 = J_2\circ df$. We call a pseudo-holomorphic map $f$ \textbf{holomorphic} if further $J_1$ and $J_2$ are integrable.} embedding, i.e., $d\gamma\circ j = J\circ d\gamma$. The almost complex structure $J$ constructed this way will, in general, \underline{not} be integrable.\\
Moser's example indicates that Hamiltonian systems with non-integrable almost complex structures $J$ might be richer than HHSs when it comes to (pseudo-)holomorphic periodic orbits. However, the generalization of HHSs is not straightforward, as the complex structure $J$ only enters most definitions regarding HHSs implicitly. To that end, let us recapitulate which objects and relations are essential to the definitions and discussions in \autoref{sec:HHS}. A HHS consists of six objects: a smooth manifold $X$ together with an integrable almost complex structure $J$ on it, two real $2$-forms $\Omega_R$ and $\Omega_I$ on $X$ which assemble to a holomorphic symplectic form $\Omega = \Omega_R + i\Omega_I$, and two smooth real functions $\mH_R$ and $\mH_I$ on $X$ forming a holomorphic function $\mH = \mH_R + i\mH_I$ on $X$. Closely tracing every step of \autoref{sec:HHS}, we see that these six objects need to satisfy the following relations:
\begin{enumerate}
 \item $\Omega_R$ must be closed\footnote{For the action functionals of holomorphic trajectories in \autoref{subsec:holo_action_fun_and_prin}, we need a primitive of $\Omega_R$, but not of $\Omega_I$!}.
 \item $J$, $\Omega_R$, and $\Omega_I$ need to satisfy the relations induced by Equation \eqref{eq:J-anticompatible}.
 \item $J$ and the Hamiltonian vector fields of the underlying RHSs have to fulfill Cauchy-Riemann-like relations formulated in Remark \autoref{rem:cauchy-riemann}.
 \item All Hamiltonian vector fields must commute reproducing Corollary \autoref{cor:commute}.
\end{enumerate}
Now, one way to define a pseudo-holomorphic Hamiltonian system is to simply impose these relations for any almost complex structure $J$:

\begin{Def}[Pseudo-holomorphic Hamiltonian system]\label{def:PHHS_1}
 We call a collection\linebreak $(X,J;\Omega_R,\Omega_I;\mH_R,\mH_I)$ a \textbf{pseudo-holomorphic Hamiltonian system} (PHHS) iff $X$ is a smooth manifold, $J$ is a (not necessarily integrable) almost complex structure on $X$, $\Omega_R$ and $\Omega_I$ are two smooth, non-degenerate, alternating, real $2$-forms on $X$, and $\mH_R$ and $\mH_I$ are two smooth real functions on $X$ satisfying:
 \begin{enumerate}
  \item $\Omega_R$ is closed, $d\Omega_R = 0$.
  \item $\Omega_R(J\cdot,\cdot) = \Omega_R(\cdot,J\cdot) = -\Omega_I;\quad \Omega_I(J\cdot,\cdot) = \Omega_I(\cdot,J\cdot) = \Omega_R;$\\
  $\Omega_R (J\cdot,J\cdot) = -\Omega_R;\qquad\qquad\ \,\ \Omega_I (J\cdot,J\cdot) = -\Omega_I$.
  \item $X^{\Omega_R}_{\mH_R} = X^{\Omega_I}_{\mH_I}$ and $J\left(X^{\Omega_R}_{\mH_R}\right) = X^{\Omega_I}_{\mH_R} = - X^{\Omega_R}_{\mH_I}$, where $X^{\Omega_a}_{\mH_b}$ is defined by $\iota_{X^{\Omega_a}_{\mH_b}}\Omega_a = -d\mH_b$.
  \item $[X^{\Omega_a}_{\mH_b}, X^{\Omega_c}_{\mH_d}] = 0$ for all $a,b,c,d\in\{R,I\}$.
 \end{enumerate}
\end{Def}

\begin{Rem}[Property 3 in Definition \autoref{def:PHHS_1}]\label{rem:H_pseudo-holo}
 Note that we can replace Property 3 in Definition \autoref{def:PHHS_1} with the condition that $\mH\coloneqq \mH_R + i\cdot \mH_I$ is pseudo-holomorphic. Indeed, Property 2 and 3 imply that $\mH:X\to\mathbb{C}$ is pseudo-holomorphic:
 \begin{align*}
  d\mH\circ J &= -\Omega_R (X^{\Omega_R}_{\mH_R}, J\cdot) - i\cdot \Omega_R (X^{\Omega_R}_{\mH_I}, J\cdot)\\
  &= \Omega_I (X^{\Omega_I}_{\mH_I}, \cdot) + i\cdot\Omega_I (-X^{\Omega_I}_{\mH_R},\cdot) = -d\mH_I + i\cdot d\mH_R = i\cdot d\mH.
 \end{align*}
 Conversely, Property 2 and $d\mH\circ J = i\cdot d\mH$ imply the Cauchy-Riemann-like equations in Property 3.
\end{Rem}

\begin{Rem}[The form $\Omega$, Part I]\label{rem:omega_I}
 As for complex manifolds, we can decompose the complexified tangent\footnote{Of course, similar remarks apply to the complexified cotangent bundle $T^{\ast}_\mathbb{C}X$ of $X$.} bundle $T_\mathbb{C}X$ of a manifold $X$ with almost complex structure $J$ into a direct sum of bundles $T^{(1,0)}X$ and $T^{(0,1)}X$ which are fiberwise given by the eigenspaces of $J$ with eigenvalue $i$ and $-i$, respectively. The difference is that now $T^{(1,0)}X$ is not a holomorphic , but merely a smooth complex vector bundle over $X$. Still, if we define the complex $2$-form $\Omega$ to be $\Omega_R + i\Omega_I$ for PHHSs, we find that $\Omega$ is of type $(2,0)$, i.e., vanishes on $T^{(0,1)}X$.
\end{Rem}

One might be confused why we only require $\Omega_R$ to be closed. The reason is that if we were to include the closedness of $\Omega_I$ into the definition of a PHHS, the almost complex structure $J$ would automatically be integrable rendering our construction pointless. The proof of this statement is given in \autoref{subsec:rel_HHS_PHHS}, when we explore the relation between HHSs and PHHSs.\\
Definition \autoref{def:PHHS_1} is convoluted, redundant, and rather unwieldy. For a better approach to PHHSs, let us first define pseudo-holomorphic symplectic manifolds:

\begin{Def}[Pseudo-holomorphic symplectic manifolds]\label{def:PHSM}
 We call a triple $(X,J;\Omega_R)$ \textbf{pseudo-holomorphic symplectic manifold} (PHSM) iff $(X,\Omega_R)$ is a symplectic manifold and $J$ is an almost complex structure on $X$ which is also $\Omega_R$-anticompatible, i.e. $\Omega_R (J\cdot, J\cdot) = -\Omega_R$. A PHSM $(X,J;\Omega_R)$ is called \textbf{proper} iff $J$ is not integrable.
\end{Def}

\begin{Rem}[The form $\Omega$, Part II]\label{rem:omega_II}
 Every PHSM $(X,J;\Omega_R)$ possesses forms $\Omega_I$ and $\Omega$ defined by
 \begin{gather*}
  \Omega_I\coloneqq -\Omega_R (J\cdot,\cdot);\quad \Omega\coloneqq \Omega_R + i\Omega_I.
 \end{gather*}
 It is easy to see that $\Omega_I$ is a smooth, non-degenerate, alternating $2$-form on $X$ which is also anticompatible with $J$. Furthermore, $\Omega$ is also anticompatible with $J$, satisfies $\Omega (J\cdot,\cdot) = \Omega (\cdot, J\cdot) = i\Omega$, i.e., $\Omega$ is of type $(2,0)$, and is non-degenerate on $T^{(1,0)}X$. However, neither $\Omega_I$ nor $\Omega$ are necessarily closed.
\end{Rem}

Now we can give an alternative definition of a PHHS:

\begin{Def}[Pseudo-holomorphic Hamiltonian system]\label{def:PHHS_2}
 We call a collection\linebreak $(X,J;\Omega_R, \mH_R)$ a \textbf{pseudo-holomorphic Hamiltonian system} (PHHS) iff $(X,J;\Omega_R)$ is a PHSM, $\mH_R:X\to\mathbb{R}$ is a smooth function on $X$, and the $1$-form $\Omega_R (J(X^{\Omega_R}_{\mH_R}),\cdot)$ is exact, where $\Omega_R (X^{\Omega_R}_{\mH_R},\cdot)\coloneqq -d\mH_R$. We call a PHHS $(X,J;\Omega_R,\mH_R)$ \textbf{proper} iff $J$ is not integrable.
\end{Def}

For both definitions to agree, it is obvious that the condition ``$\Omega_R (J(X^{\Omega_R}_{\mH_R}),\cdot)$ is exact'' is necessary, since by Definition \autoref{def:PHHS_1}: $\Omega_R (J(X^{\Omega_R}_{\mH_R}),\cdot) = d\mH_I$. It is also sufficient as guaranteed by the following proposition:

\begin{Prop}[PHHSs well-defined]\label{prop:def_PHHS_agree}
 Definition \autoref{def:PHHS_1} and Definition \autoref{def:PHHS_2} coincide.
\end{Prop}

\begin{proof}
 Clearly, every PHHS as in Definition \autoref{def:PHHS_1} also fulfills Definition \autoref{def:PHHS_2}. Now let\linebreak $(X,J;\Omega_R,\mH_R)$ be a PHHS as in Definition \autoref{def:PHHS_2}. Then, we define $\Omega_I\coloneqq -\Omega_R (J\cdot,\cdot)$ and take $\mH_I$ to be a primitive of the $1$-form $\Omega_R (J(X^{\Omega_R}_{\mH_R}),\cdot)$. We need to check that these data satisfy the properties 1-4 in Definition \autoref{def:PHHS_1}. Property 1 is trivially true by definition. Verifying Property 2 is a short and easy computation. To check Property 3, we recall Remark \autoref{rem:H_pseudo-holo}. It suffices to verify that the map $\mH = \mH_R + i\mH_I:X\to\mathbb{C}$ is pseudo-holomorphic which follows immediately:
 \begin{align*}
  d\mH_R\circ J &= -\Omega_R (X^{\Omega_R}_{\mH_R},J\cdot) = -\Omega_R (J(X^{\Omega_R}_{\mH_R}),\cdot) = -d\mH_I,\\
  d\mH_I\circ J &= \Omega_R (J(X^{\Omega_R}_{\mH_R}),J\cdot) = -\Omega_R (X^{\Omega_R}_{\mH_R},\cdot) = d\mH_R,\\
  \Rightarrow d\mH\circ J &= i\cdot d\mH.
 \end{align*}
 Lastly, we need to check Property 4. For this, remember that any symplectic manifold $(X,\Omega_R)$ admits a Poisson bracket $\{\cdot,\cdot\}:C^{\infty}(X,\mathbb{R})\times C^{\infty}(X,\mathbb{R})\to C^{\infty}(X,\mathbb{R})$ given by
 \begin{gather*}
  \{F,G\}\coloneqq \Omega_R (X_F, X_G),
 \end{gather*}
 where $X_F$ and $X_G$ are the Hamiltonian vector fields of the functions $F$ and $G$. Furthermore, recall that the map $X_\cdot:C^{\infty}(X,\mathbb{R})\to\Gamma (TX)$ is a Lie algebra homomorphism:
 \begin{gather*}
  X_{\{F,G\}} = [X_F, X_G]\quad\forall F,G\in C^{\infty}(X,\mathbb{R}).
 \end{gather*}
 Hence, it suffices to prove that $\{\mH_R,\mH_I\}$ vanishes in order to show that $X^{\Omega_R}_{\mH_R}$ and $X^{\Omega_R}_{\mH_I}$ commute. Let us calculate $\{\mH_R,\mH_I\}$ using Property 2 and 3:
 \begin{gather*}
  \{\mH_R,\mH_I\} = \Omega_R (X^{\Omega_R}_{\mH_R}, X^{\Omega_R}_{\mH_I}) = -\Omega_R (X^{\Omega_R}_{\mH_R}, J(X^{\Omega_R}_{\mH_R})) = \Omega_I (X^{\Omega_R}_{\mH_R},X^{\Omega_R}_{\mH_R}) = 0.
 \end{gather*}
 Commutativity of the remaining Hamiltonian vector fields follows from commutativity of $X^{\Omega_R}_{\mH_R}$ and $X^{\Omega_R}_{\mH_I}$ as well as Property 3 concluding the proof.
\end{proof}

Now that we have found a compact definition of PHHSs, we should briefly mention some examples of PHHSs. Of course, every HHS is, by design, a PHHS with integrable $J$. Even though the set of proper PHHSs is much larger than the set of HHSs, finding them is a bit more involved and, thus, relegated to \autoref{sec:deformation}. Partially, this is due to the fact that there are no ``standard'' examples of proper PHHSs like cotangent bundles\footnote{At least no canonical ones! In \autoref{sec:deformation}, we will equip the holomorphic cotangent bundle of a complex manifold with a non-canonical PHHS-structure.} as there are for RHSs and HHSs. On a deeper level, this is caused by the absence of a Darboux-like theorem. Clearly, there cannot be a counterpart to Darboux's theorem for PHSMs, since $J$ is usually not integrable and, hence, there are no coordinates in which $J$ assumes the standard form, let alone coordinates in which both $J$ and $\Omega = \Omega_R + i\Omega_I$ assume some standard form. Still, one can bring $J$ and $\Omega$ into standard form using local (usually non-integrable) frames:

\begin{Lem}[PHSMs in local frames]\label{lem:PHSM_in_sta_form}
 Let $(X,J;\Omega_R)$ be a PHSM with $\Omega\coloneqq \Omega_R - i\Omega_R (J\cdot,\cdot)$ and let $x_0\in X$ be any point. Then, there exists an open neighborhood $U\subset X$ of $x_0$ and a local frame $\theta^Q_1,\ldots, \theta^Q_n,\theta^P_1,\ldots,\theta^P_n$ of the smooth complex vector bundle $T^{\ast, (1,0)}X$ on $U$ such that $\Omega$ on $U$ can be expressed as
 \begin{gather*}
  \Omega\vert_U = \sum^{n}_{j} \theta^P_j\wedge \theta^Q_j.
 \end{gather*}
 In particular, the real dimension of $X$ is a multiple of $4$. The local frame can be chosen to be integrable (after shrinking $U$ if necessary) if and only if $J$ is integrable near $x_0$.
\end{Lem}

\begin{Rem}[$J$ in standard form]\label{rem:J_sta}
 $J$ is also in standard form in the dual frame of $\theta^Q_1,\ldots, \theta^Q_n,\theta^P_1,\ldots,\theta^P_n$. One can see this as follows: the real and imaginary part of the local frame $\theta^Q_1,\ldots, \theta^Q_n,\theta^P_1,\ldots,\theta^P_n$ ($\theta = \theta^x + i \theta^y$) give rise to a local frame of the real cotangent bundle $T^{\ast}X$. Its dual frame ${\hat e}^{Q, x}_{1},\ldots, {\hat e}^{P,y}_{n}$ is a local frame of the tangent bundle $TX$. By setting ${\hat e}\coloneqq 1/2({\hat e}^x - i{\hat e}^y)$, one obtains a local frame of $T^{(1,0)}X$. On $T^{(1,0)}X$, $J$ simply acts by $i$, thus, ${\hat e}^x$ and ${\hat e}^y$ satisfy $J({\hat e}^x) = {\hat e}^y$ and $J({\hat e}^y) = -{\hat e}^x$. This is the standard form of $J$.
\end{Rem}

\begin{proof}
 Lemma \autoref{lem:PHSM_in_sta_form} follows from the application of the symplectic Gram-Schmidt process, which can be found in any textbook on symplectic geometry, to a local frame of $T^{(1,0)}X$. Confer Proposition 2.8 in \cite{bogomolov2020} for the complex analogue of the symplectic Gram-Schmidt process. For completeness' sake, we repeat the explicit construction here. Let $(X,J;\Omega_R)$ be a PHSM with $\Omega\coloneqq \Omega_R - i\Omega_R (J\cdot,\cdot)$ and take the real dimension of $X$ to be\linebreak $\text{dim}_\mathbb{R} (X) = 2m$, $m\in\mathbb{N}$. The dimension of $X$ is even, as $X$ admits an almost complex structure $J$. Then, the complex rank of the complexified bundle $T_\mathbb{C}X$ is also given by $2m$. Now recall the decomposition $T_\mathbb{C}X = T^{(1,0)}X\oplus T^{(0,1)}X$. Since the complex vector bundles $T^{(1,0)}X$ and $T^{(0,1)}X$ are isomorphic via the complex conjugation $v + iw\mapsto v- iw$, their fibers have the same complex dimension, namely $m$. Now let $x_0\in X$ be any point and pick a local frame $v_1,\ldots, v_m$ of $T^{(1,0)}X$ on an open neighborhood $U\subset X$ of $x_0$. $\Omega$ is non-degenerate on $T^{(1,0)}X$ by Remark \autoref{rem:omega_II}, hence, there exists a vector ${\hat e}^Q_1\vert_{x_0}\in T^{(1,0)}_{x_0}X$ such that $\Omega\vert_{x_0} (v_1\vert_{x_0}, {\hat e}^Q_1\vert_{x_0})\neq 0$. $v_1,\ldots, v_m$ is a local frame of $T^{(1,0)}X$ near $x_0$, thus, we can write
 \begin{gather*}
  {\hat e}^Q_1\vert_{x_0} = \sum^m_{j = 1} c_j\cdot v_j\vert_{x_0}
 \end{gather*}
 for some constants $c_j\in\mathbb{C}$. Now define the local section ${\hat e}^Q_1 \coloneqq \sum_{j} c_j\cdot v_j$ of $T^{(1,0)}X$. After shrinking $U$ while preserving $x_0\in U$ if necessary, one obtains $\Omega\vert_x (v_1\vert_x, {\hat e}^Q_1\vert_x)\neq 0$ for every $x\in U$. Setting ${\hat e}^P_1\coloneqq v_1$ and changing the normalization of ${\hat e}^Q_1$ if necessary allows us to write $\Omega\vert_U ({\hat e}^P_1, {\hat e}^Q_1) = 1$.\\
 If $m = 2$, we simply define $\theta^Q_1, \theta^P_1$ to be the dual frame of ${\hat e}^Q_1, {\hat e}^P_1$. If $m> 2$, then we can pick one local section of the frame $v_1,\ldots, v_m$, say $v_2$, such that ${\hat e}^Q_1\vert_x$, ${\hat e}^P_1\vert_x$, and $v_2\vert_x$ are $\mathbb{C}$-linearly independent for every $x\in U$ after shrinking $U\ni x_0$ if necessary. We set:
 \begin{gather*}
  \hat v_2\coloneqq v_2 - \Omega\vert_U (v_2, {\hat e}^Q_1)\cdot {\hat e}^P_1 + \Omega\vert_U (v_2, {\hat e}^P_1)\cdot {\hat e}^Q_1.
 \end{gather*}
 Then, ${\hat e}^Q_1$, ${\hat e}^P_1$, and $\hat v_2$ are still $\mathbb{C}$-linearly independent on $U$ and $\hat v_2$ is $\Omega$-orthogonal to ${\hat e}^Q_1$ and ${\hat e}^P_1$, i.e., $\Omega\vert_U (\hat v_2, {\hat e}^Q_1) = \Omega\vert_U (\hat v_2, {\hat e}^P_1) = 0$. Again by the non-degeneracy of $\Omega$, we can find a vector $e^Q_2\vert_{x_0}\in T^{(1,0)}_{x_0}X$ such that $\Omega\vert_{x_0} (\hat v_2\vert_{x_0}, e^Q_2\vert_{x_0})\neq 0$. As before, we can write
 \begin{gather*}
  {e}^Q_2\vert_{x_0} = \sum^m_{j = 1} d_j\cdot v_j\vert_{x_0}
 \end{gather*}
 for some constants $d_j\in\mathbb{C}$ and define the local section ${e}^Q_2 \coloneqq \sum_{j} d_j\cdot v_j$ of $T^{(1,0)}$. After shrinking $U$ and changing the normalization of $e^Q_2$ if necessary, we obtain $\Omega\vert_U (\hat v_2, e^Q_2) = 1$. Now we set:
 \begin{gather*}
  {\hat e}^P_2\coloneqq \hat v_2;\quad {\hat e}^Q_2\coloneqq e^Q_2 - \Omega\vert_U (e^Q_2, {\hat e}^Q_1)\cdot {\hat e}^P_1 + \Omega\vert_U (e^Q_2, {\hat e}^P_1)\cdot {\hat e}^Q_1.
 \end{gather*}
 Proceeding inductively gives us a local frame ${\hat e}^Q_1,\ldots {\hat e}^Q_n$, ${\hat e}^P_1,\ldots, {\hat e}^P_n$ of $T^{(1,0)}X$ on some neighborhood $U$ of $x_0$ ($n\coloneqq m/2$) satisfying:
 \begin{gather*}
  \Omega\vert_U ({\hat e}^Q_i, {\hat e}^Q_j) = \Omega\vert_U ({\hat e}^P_i, {\hat e}^P_j) = 0;\quad \Omega\vert_U ({\hat e}^P_i, {\hat e}^Q_j) = \delta_{ij}.
 \end{gather*}
 Thus, the frame $\theta^Q_1,\ldots, \theta^Q_n$, $\theta^P_1,\ldots, \theta^P_n$ dual to the frame ${\hat e}^Q_1,\ldots {\hat e}^Q_n$, ${\hat e}^P_1,\ldots, {\hat e}^P_n$ is the desired local frame of $T^{\ast, (1,0)}X$ near $x_0$ in which $\Omega$ takes the form:
 \begin{gather*}
  \Omega\vert_U = \sum^{n}_{j} \theta^P_j\wedge \theta^Q_j.
 \end{gather*}
 In particular, the real dimension of $X$ is $4n$. If the frame $\theta^Q_1,\ldots$ is integrable, then the frame ${\hat e}^Q_1,\ldots$ is also integrable and there exists a (holomorphic) chart\linebreak $\phi = (Q_1,\ldots, Q_n, P_1,\ldots, P_n):U\to V\subset\mathbb{C}^{2n}$ near $x_0$ such that:
 \begin{gather*}
  {\hat e}^Q_j\equiv \pa_{Q_j};\quad {\hat e}^P_j\equiv \pa_{P_j};\quad \theta^Q_j\equiv dQ_j;\quad \theta^P_j\equiv dP_j.
 \end{gather*}
 By Remark \autoref{rem:J_sta}, $J$ also assumes its standard form in this chart, thus, the Nijenhuis tensor of $J$ vanishes on an open neighborhood of $x_0$ and $J$ is integrable near $x_0$. The converse direction follows from Theorem \autoref{thm:rel_HSM_PHSM} (cf. \autoref{subsec:rel_HHS_PHHS}) and Darboux's theorem for HSMs (cf. Theorem \autoref{thm:holo_Darboux} and \autoref{app:darboux}).
\end{proof}

Next, let us investigate the dynamics of a PHHS $(X,J;\Omega_R,\mH_R)$. To do so, we need to introduce pseudo-holomorphic Hamiltonian vector fields and trajectories. Hereby, we imitate the definitions from \autoref{sec:HHS}. For the sake of simplicity, we always associate from now on with a PHHS $(X,J;\Omega_R,\mH_R)$ the forms $\Omega_I\coloneqq -\Omega_R (J\cdot,\cdot)$ and $\Omega\coloneqq \Omega_R + i\Omega_I$ as well as the functions $\mH_I$ and $\mH\coloneqq \mH_R + i\mH_I$, where $\mH_I$ is a primitive of $\Omega_R (J(X^{\Omega_R}_{\mH_R}),\cdot)$.

\begin{Def}[Pseudo-holomorphic Hamiltonian vector fields and trajectories]\label{def:pseudo-holo_ham_field_and_traj}
 Let\linebreak $(X,J;\Omega_R, \mH_R)$ be a PHHS. We call the smooth section $X_\mH$ of $T^{(1,0)}X$ defined by\linebreak $\iota_{X_\mH}\Omega = -d\mH$ the (pseudo-holomorphic) \textbf{Hamiltonian vector field} of the\linebreak PHHS $(X,J;\Omega_R, \mH_R)$. Furthermore, we call a pseudo-holomorphic map $\gamma:U\to X$ a \textbf{pseudo-holomorphic trajectory} of the PHHS $(X,J;\Omega_R, \mH_R)$ iff $\gamma$ satisfies the pseudo-holomorphic integral curve equation:
 \begin{gather*}
  \frac{\pa\gamma}{\pa z} (z)\coloneqq \frac{1}{2}\left(\frac{\pa \gamma}{\pa t} (z) - i\frac{\pa\gamma}{\pa s} (z)\right) = X_\mH (\gamma (z))\quad\forall z = t + is\in U,
 \end{gather*}
 where $U\subset\mathbb{C}$ is an open and connected subset with standard complex structure $j = i$. We call a pseudo-holomorphic trajectory $\gamma:U\to X$ \textbf{maximal} iff for every pseudo-holomorphic trajectory $\hat\gamma:\hat U\to X$ with $U\subset \hat U$ and $\hat\gamma\vert_U = \gamma$ one has $\hat U = U$ and $\hat \gamma = \gamma$.
\end{Def}

Alternatively, one can define pseudo-holomorphic Hamiltonian vector fields and trajectories in terms of the vector field $X^{\Omega_R}_{\mH_R}$:

\begin{Prop}[Alternative Definition of $X_\mH$ and $\gamma$]\label{prop:pseudo-holo_ham_field_and_traj}
 Let $(X,J;\Omega_R,\mH_R)$ be a PHHS with vector field $X^{\Omega_R}_{\mH_R}$ defined by $\Omega_R (X^{\Omega_R}_{\mH_R},\cdot) = -d\mH_R$. Then:
 \begin{gather*}
  X_\mH = \frac{1}{2} \left( X^{\Omega_R}_{\mH_R} - i\cdot J(X^{\Omega_R}_{\mH_R})\right).
 \end{gather*}
 Now, let $U\subset\mathbb{C}$ be an open and connected subset and $\gamma:U\to X$ be a map. Then, $\gamma$ is a pseudo-holomorphic trajectory iff $\gamma_s$ defined by $\gamma_s (t)\coloneqq \gamma (t + is)$ is an integral curve of $X^{\Omega_R}_{\mH_R}$ for every suitable $s\in\mathbb{R}$ and $\gamma:U\to X$ is pseudo-holomorphic.
\end{Prop}

\begin{proof}
 A straightforward calculation verifies that $X_\mH$ as in Definition \autoref{def:pseudo-holo_ham_field_and_traj} satisfies the equation above. To prove the statement about pseudo-holomorphic trajectories, consider the real part of the pseudo-holomorphic integral curve equation and observe that due to the pseudo-holomorphicity of $\gamma$:
 \begin{gather*}
  \frac{\pa\gamma}{\pa s} = J(\frac{\pa\gamma}{\pa t}).
 \end{gather*}
\end{proof}

We designed PHHSs in such a way that all properties we found for HHSs in \autoref{sec:PHHS} (almost) completely transfer to PHHSs. For instance, we find the following PHHS-counterpart to Proposition \autoref{prop:holo_traj}:

\begin{Prop}[Existence and uniqueness of pseudo-holomorphic trajectories]\label{prop:pseudo-holo_traj}
 Let $(X,J;\Omega_R,\mH_R)$ be a PHHS. Then, for any $z_0\in\mathbb{C}$ and $x_0\in X$, there exists an open and connected subset $U\subset\mathbb{C}$ and a pseudo-holomorphic trajectory $\gamma^{z_0, x_0}:U\to X$ of $(X,J;\Omega_R,\mH_R)$ with $\gamma^{z_0, x_0} (z_0) = x_0$. Two pseudo-holomorphic trajectories\linebreak $\gamma^{z_0, x_0}_1:U_1\to X$ and $\gamma^{z_0, x_0}_2:U_2\to X$ with $\gamma^{z_0, x_0}_1 (z_0) = x_0 = \gamma^{z_0, x_0}_2 (z_0)$ locally coincide, in particular, they are equal iff their domains $U_1$ and $U_2$ are equal. Furthermore, the pseudo-holomorphic trajectory $\gamma^{z_0, x_0}$ depends pseudo-holomorphically on $z_0$, but, in general, \underline{only} smoothly on $x_0$.
\end{Prop}

\begin{proof}
 The proof of Proposition \autoref{prop:pseudo-holo_traj} works very similarly to the proof of Proposition \autoref{prop:holo_traj}. By Definition \autoref{def:PHHS_1} and Proposition \autoref{prop:pseudo-holo_ham_field_and_traj}, the real and imaginary part of $X_\mH$ commute, hence, we can proceed as in the proof of Proposition \autoref{prop:holo_traj} to show that pseudo-holomorphic trajectories $\gamma^{z_0,x_0}$ given an initial value $\gamma^{z_0, x_0} (z_0) = x_0$ exist.\\
 To prove uniqueness, we observe that the formula
 \begin{gather*}
  \gamma^{z_0, x_0}(z)\coloneqq \varphi^{J(X^{\Omega_R}_{\mH_R})}_{s-s_0}\circ\varphi^{X^{\Omega_R}_{\mH_R}}_{t-t_0} (x_0)\equiv \varphi^{X^{\Omega_R}_{\mH_R}}_{t-t_0}\circ\varphi^{J(X^{\Omega_R}_{\mH_R})}_{s-s_0} (x_0)\equiv\varphi^{(t-t_0)X^{\Omega_R}_{\mH_R} + (s-s_0)J(X^{\Omega_R}_{\mH_R})}_1 (x_0),
 \end{gather*}
 where $z = t+is$, uniquely determines $\gamma^{z_0, x_0}$ on a small rectangle in $\mathbb{C}$ near $z_0 = t_0 + is_0$. The rest now follows by covering a path between $z_0$ and any point $z_1$ in $U_1\equiv U_2$ with a finite number of such rectangles.\\
 Lastly, let us consider the dependence of $\gamma^{z_0, x_0}$ on $z_0\in\mathbb{C}$ and $x_0\in X$. Again, $\gamma^{z_1, x_0} (z)$ and $\gamma^{z_2, x_0} (z)$ only differ by a translation in $z$. As pseudo-holomorphic trajectories are pseudo-holomorphic maps, the $z_0$-dependence is also pseudo-holomorphic. For the $x_0$-dependence, we need to consider the flow of $X^{\Omega_R}_{\mH_R}$ and $J(X^{\Omega_R}_{\mH_R})$. Both $X^{\Omega_R}_{\mH_R}$ and $J(X^{\Omega_R}_{\mH_R})$ are smooth vector fields, thus, their flows are smooth as well concluding the proof.
\end{proof}

\begin{Rem}[$x_0$-dependence]\label{rem:x_0-dependance}
 Note that holomorphic trajectories of HHSs depend holomorphically on $x_0$, while pseudo-holomorphic trajectories of PHHSs do \underline{not} generally depend pseudo-holomorphically on $x_0$. This distinction can be traced back to the Hamiltonian vector field $X_\mH$. For HHSs, $X_\mH$ is a holomorphic vector field, in particular its real and imaginary part are $J$-preserving vector fields (cf. Proposition \autoref{prop:holo_vec_field_equiv_J_pre_vec_field}) implying that the differential of their flows commute with $J$. For PHHSs, this does not need to be the case anymore: neither $X^{\Omega_R}_{\mH_R}$ nor $J(X^{\Omega_R}_{\mH_R})$ are required to be $J$-preserving! In fact, we study an example of a proper PHHS in \autoref{sec:deformation} where $X^{\Omega_R}_{\mH_R}$ is $J$-preserving, but $J(X^{\Omega_R}_{\mH_R})$ is not.
\end{Rem}

As for HHSs, the maximal trajectories of a PHHS, given an initial value, do not need to be unique, however, we can still pseudo-holomorphically foliate energy hypersurfaces $\mH^{-1}(E)$ of a PHHS:

\begin{Prop}[Pseudo-holomorphic foliation of a regular hypersurface]\label{prop:pseudo-holo_foli}
 Let\linebreak $(X,J;\Omega_R,\mH_R)$ be a PHHS with Hamiltonian vector field $X_\mH = 1/2 (X^{\Omega_R}_{\mH_R} - i J(X^{\Omega_R}_{\mH_R}))$ and regular\footnote{As before, a PHHS $(X,J;\Omega_R,\mH_R)$ is regular at the energy $E\in\mathbb{C}$ iff $d\mH$ or, equivalently, $d\mH_R$ does not vanish on $\mH^{-1}(E)$.} value $E$ of $\mH$. Then, the energy hypersurface $\mH^{-1}(E)$ admits a pseudo-holomorphic foliation. The leaf $L_{x_0}$ of this foliation through a point $x_0\in\mH^{-1}(E)$ is given by
 \begin{align*}
  L_{x_0}\coloneqq \{y\in X\mid &y = \varphi^{X^{\Omega_R}_{\mH_R}}_{t_1}\circ\varphi^{J(X^{\Omega_R}_{\mH_R})}_{s_1}\circ\varphi^{X^{\Omega_R}_{\mH_R}}_{t_2}\circ\varphi^{J(X^{\Omega_R}_{\mH_R})}_{s_2}\circ\ldots\circ\varphi^{X^{\Omega_R}_{\mH_R}}_{t_n}\circ\varphi^{J(X^{\Omega_R}_{\mH_R})}_{s_n} (x_0);\\
  &t_1,\ldots,t_n, s_1,\ldots, s_n\in\mathbb{R};\ n\in\mathbb{N}\},
 \end{align*}
 where $\varphi^{X^{\Omega_R}_{\mH_R}}_{t_j}$ and $\varphi^{J(X^{\Omega_R}_{\mH_R})}_{s_j}$ are the flows of $X^{\Omega_R}_{\mH_R}$ and $J(X^{\Omega_R}_{\mH_R})$ for time $t_j$ and $s_j$, respectively. Every pseudo-holomorphic trajectory of $(X,J;\Omega_R,\mH_R)$ with energy $E$ is completely contained in one such leaf.
\end{Prop}

\begin{proof}
 First, we need to clarify the notion of a pseudo-holomorphic foliation. In order to do that, recall the definition of a holomorphic foliation. A ($d$-dimensional) holomorphic foliation $\{L_{x_0}\}_{x_0\in I}$ ($I$: index set) of a complex manifold $X$ is a decomposition of\linebreak $X = \bigcup_{x_0\in I} L_{x_0}$ into a disjoint union of leaves $L_{x_0}$, path-connected subsets of $X$, such that for every point $x\in X$ there exists a holomorphic chart $\phi = (z_1,\ldots, z_n):U\to V\subset\mathbb{C}^n$ of $X$ near $x$ fulfilling: for every leaf $L_{x_0}$ with $U\cap L_{x_0}\neq\emptyset$, the connected components of $U\cap L_{x_0}$ are given by $z_{d+1} = c_{d+1}$,\ldots, $z_n = c_n$ for some constants $c_{d+1},\ldots, c_n\in\mathbb{C}$. Clearly, we cannot directly transfer this definition to the non-integrable case, since generic almost complex manifolds do not admit holomorphic charts. Therefore, we call $\{L_{x_0}\}_{x_0\in I}$ a pseudo-holomorphic foliation of an almost complex manifold $(X,J)$ iff $\{L_{x_0}\}_{x_0\in I}$ is a (smooth) foliation of $X$ and the tangent spaces of the leaves $L_{x_0}$ are closed under the action of $J$. We can now prove the last proposition in the same way as Proposition \autoref{prop:holo_foli} by applying the Frobenius theorem to the vector fields $X^{\Omega_R}_{\mH_R}$ and $J(X^{\Omega_R}_{\mH_R})$. 
\end{proof}

Similarly to HHSs, we can also define the notion of geometric trajectories for\linebreak PHHSs. We simply copy Definition \autoref{def:geo_traj} and replace the term ``holomorphic'' with\linebreak ``pseudo-holomorphic''. All results we found in \autoref{subsec:holo_traj} for geometric trajectories of HHSs still hold in the pseudo-holomorphic case. In particular, Proposition \autoref{prop:geo_traj} is still true for PHHSs. The proof is essentially the same as in the holomorphic case. However, the vector field $Y_\mH$ on the Riemann surface $\Sigma$ is a priori only a smooth section of $T^{(1,0)}\Sigma$. $Y_\mH$ becomes a holomorphic vector field on $\Sigma$ by noting that the real and imaginary part of the Hamiltonian vector field $X_\mH$ commute by construction. Thus, the real and imaginary part of $Y_\mH$ also commute, as the push-forward of $\gamma$ is a Lie algebra homomorphism, i.e., $\gamma_\ast [V,W] = [\gamma_\ast V, \gamma_\ast W]$ for vector fields\footnote{Precisely speaking, this is not correct, since $\gamma:\Sigma\to X$ is only an immersion and not a diffeomorphism, hence, the push-forward of $\gamma$ is not well-defined. Nevertheless, the argument still holds if we consider $\gamma_\ast V$ and $\gamma_\ast W$ to be sections of the pull-back bundle $\gamma^\ast TX$ and adjust the definition of the Lie bracket accordingly.} $V$ and $W$ on $\Sigma$. Now note that, for a Riemann surface $\Sigma$, the real and imaginary part of a smooth section $V$ of $T^{(1,0)}\Sigma$ commute if and only if $V$ is a holomorphic vector field on $\Sigma$. This is easily verified in holomorphic charts of $\Sigma$. This shows the holomorphicity of $Y_\mH$. The proofs for the remaining results regarding geometric trajectories work as in the holomorphic case after adjusting the language where need be.\\
Before we conclude this subsection, we want to formulate action functionals and principles for pseudo-holomorphic trajectories. By construction of PHHSs, this can be done in the same way as in \autoref{subsec:holo_action_fun_and_prin}. First, we note that a PHHS $(X,J;\Omega_R,\mH_R)$ decomposes into multiple RHSs. In contrast to HHSs, we only obtain two RHSs this time, namely $(X,\Omega_R,\mH_R)$ and $(X,\Omega_R,\mH_I)$, since $\Omega_I$ is, in general, not closed. However, this suffices to find action functionals for pseudo-holomorphic trajectories, as only two of the four underlying RHSs of a HHS are subject to different dynamics. If the PHHS $(X,J;\Omega_R,\mH_R)$ is exact, i.e., $\Omega_R = d\Lambda_R$, the two RHSs $(X,\Omega_R,\mH_R)$ and $(X,\Omega_R,\mH_I)$ are also exact and possess themselves action functionals. As in the case of HHSs, we can now average these action functionals over the remaining time variable and take suitable linear combinations afterwards to find the following action functional for pseudo-holomorphic trajectories:

\begin{Lem}[Action principle for pseudo-holomorphic trajectories]\label{lem:pseudo-holo_action_prin_para}
 {\textcolor{white}{Easter Egg}}\linebreak Let $(X,J;\Omega_R = d\Lambda_R,\mH_R)$ be an exact PHHS. For $\alpha\in\mathbb{R}\backslash\{n\cdot\pi\mid n\in\mathbb{Z}\}$, let\linebreak $P_\alpha\coloneqq [t_1,t_2] + e^{i\alpha}[r_1,r_2]\subset\mathbb{C}$ be a parallelogram in the complex plane with real numbers $t_1< t_2$ and $r_1< r_2$. Denote the space of smooth maps from $P_\alpha$ to $X$ by $\mathcal{P}_{P_\alpha}$ and define the action functional $\mathcal{A}^{P_\alpha}_\mH:\mathcal{P}_{P_\alpha}\to\mathbb{C}$ by
\begin{gather*}
 \mathcal{A}^{P_\alpha}_\mH[\gamma]\coloneqq \iint\limits_{P_\alpha}\left[\Lambda_R\vert_{\gamma (t+is)}\left(2\frac{\pa\gamma}{\pa z}(t+is)\right) - \mH\circ\gamma (t+is)\right] dt\wedge ds\ \text{with}\ \frac{\pa\gamma}{\pa z}\coloneqq \frac{1}{2}\left(\frac{\pa\gamma}{\pa t} - i\frac{d\gamma}{\pa s}\right)
\end{gather*}
 for every $\gamma\in\mathcal{P}_{P_\alpha}$. Now, let $\gamma\in\mathcal{P}_{P_\alpha}$ be a smooth map from $P_\alpha$ to $X$. Then, $\gamma$ is a pseudo-holomorphic trajectory of $(X,J;\Omega_R,\mH_R)$ iff $\gamma$ is a ``critical point''\footnote{``Critical point'' means that only those variations are allowed which keep $\gamma$ fixed on the boundary $\partial P_\alpha$.} of $\mathcal{A}^{P_\alpha}_\mH$.
\end{Lem}

\begin{proof}
 For the proof of Lemma \autoref{lem:pseudo-holo_action_prin_para}, repeat the steps from \autoref{subsec:holo_action_fun_and_prin} for PHHSs, in particular the proof of Proposition \autoref{prop:holo_action_prin_para}.
\end{proof}

As before, if we wish to view pseudo-holomorphic trajectories as actual critical points of some functional, we can achieve this by either mapping the boundary $\partial P_\alpha$ to a Lagrangian submanifold of $(X,\Omega_R)$ or by imposing periodicity on the curves $\gamma$.

\newpage
\subsection{Relation between HHSs and PHHSs}
\label{subsec:rel_HHS_PHHS}

At this point, it is not clear how PHHSs relate to HHSs and why PHHSs are a ``reasonable'' generalization of HHSs with regard to the integrability of $J$. In particular, we do not know yet why the notion of PHHSs introduced in \autoref{subsec:def_PHHS} should coincide with the notion of HHSs when we restore the integrability of $J$. A priori, there is no reason why the $2$-form $\Omega\coloneqq \Omega_R - i\Omega_R (J\cdot,\cdot)$ associated with a PHHS $(X,J;\Omega_R,\mH_R)$ should be holomorphic or even closed for integrable $J$. That this is indeed the case is guaranteed by the following theorem:

\begin{Thm}[Relation between HSMs and PHSMs]\label{thm:rel_HSM_PHSM}
 Let $(X,J;\Omega_R)$ be a PHSM with $2$-forms $\Omega_I\coloneqq -\Omega_R (J\cdot,\cdot)$ and $\Omega\coloneqq \Omega_R + i\Omega_I$. Further, let $x_0\in X$ be any point. Then, $J$ is integrable near $x_0$ if and only if $d\Omega_I$ vanishes near $x_0$. Moreover, the following statements are equivalent:
 \begin{enumerate}
  \item $(X,\Omega)$ is a HSM with complex structure $J$.
  \item $\Omega_I$ is closed, $d\Omega_I = 0$.
  \item $J$ is integrable.
 \end{enumerate}
\end{Thm}

\begin{proof}
 We only show the equivalence of Statement 1, 2, and 3. The first part of Theorem \autoref{thm:rel_HSM_PHSM} is then just a local version of the equivalence. Direction ``1$\Rightarrow$2'' is trivially true by definition of a HSM. Implication ``2$\Rightarrow$3'' is due to Verbitsky (confer Theorem 3.5 in \cite{verbitsky2013} and Proposition 2.12 in \cite{bogomolov2020}). For completeness' sake, we include the proof here. Let $(X,J;\Omega_R)$ be a PHSM with $2$-forms $\Omega_I\coloneqq -\Omega_R (J\cdot,\cdot)$ and $\Omega\coloneqq \Omega_R + i\Omega_I$. Further, assume $d\Omega_I = 0$. We want to show that $J$ is integrable. By the Newlander-Nirenberg theorem, $J$ is integrable if and only if $J$ has no torsion, i.e., its Nijenhuis tensor vanishes. Now we apply Theorem 2.8 in Chapter IX of \cite{kobayashi1969}. Thus, $J$ is integrable if and only if the space of smooth sections of $T^{(0,1)}X$ is closed under the commutator $[\cdot,\cdot]$. From \autoref{subsec:def_PHHS}, we know that $\Omega$ is non-degenerate on $T^{(1,0)}X$, but vanishes on $T^{(0,1)}X$. Hence, a complex vector field $V$ on $X$ is a smooth section of $T^{(0,1)}$ if and only if $\iota_V\Omega = 0$. Therefore, $J$ is integrable if and only if for every pair of two complex vector fields $V$ and $W$ on $X$ satisfying $\iota_V\Omega = \iota_W\Omega = 0$ one has $\iota_{[V,W]}\Omega = 0$. Now let $V$ and $W$ be two complex vector fields on $X$ with $\iota_V\Omega = \iota_W\Omega = 0$. Recall that the interior product $\iota$ applied to forms fulfills the relation (cf. Proposition 3.10 in Chapter I of \cite{kobayashi1963}):
 \begin{gather*}
  \iota_{[V,W]} = [L_V, \iota_W],
 \end{gather*}
 where $L_V$ is the Lie derivative of $V$. We can calculate $L_V\Omega$ by using Cartan's magic formula, $\iota_V\Omega = 0$, and $d\Omega = 0$:
 \begin{gather*}
  L_V\Omega = d\iota_V\Omega + \iota_Vd\Omega = 0.
 \end{gather*}
 In total, we obtain using $\iota_W\Omega = 0$:
 \begin{gather*}
  \iota_{[V,W]}\Omega = [L_V,\iota_W]\Omega = L_V(\iota_W\Omega) - \iota_W(L_V\Omega) = 0
 \end{gather*}
 proving the integrability of $J$.
 \newpage
 The remaining direction ``3$\Rightarrow$1'' can be proven as follows: let $(X,J;\Omega_R)$ be a PHSM with $2$-forms $\Omega_I\coloneqq -\Omega_R (J\cdot,\cdot)$ and $\Omega\coloneqq \Omega_R + i\Omega_I$. Further, assume that $J$ is integrable. Then, $X$ is a complex manifold with complex structure $J$. We need to show that $\Omega$ is a closed, holomorphic $2$-form which is non-degenerate on $T^{(1,0)}X$. By Remark \autoref{rem:omega_II}, $\Omega$ is non-degenerate on $T^{(1,0)}X$ and of type $(2,0)$. Hence, $\Omega$ can be written in a holomorphic chart $\phi = (z_1,\ldots, z_{2m}):U\to V\subset\mathbb{C}^{2m}$ of $X$ as
 \begin{gather*}
  \Omega\vert_U = \sum\limits^{2m}_{i,j = 1} \Omega_{ij} dz_i\wedge dz_j,
 \end{gather*}
 where the coefficients $\Omega_{ij} = -\Omega_{ji}:U\to\mathbb{C}\cong\mathbb{R}^2$ are smooth functions on $U$. From $d\Omega_R = 0$, we deduce:
 \begin{align*}
  0 = 2d\Omega_R\vert_U = (\partial + \bar{\partial})(\Omega + \overline{\Omega})\vert_U &= \sum\limits^{2m}_{i,j,k = 1}\left(\frac{\partial\Omega_{ij}}{\partial z_k} dz_k\wedge dz_i\wedge dz_j + \frac{\partial\Omega_{ij}}{\partial \bar{z}_k} d\bar{z}_k\wedge dz_i\wedge dz_j\right.\\
  &\qquad\qquad + \left.\frac{\partial\overline{\Omega}_{ij}}{\partial z_k} dz_k\wedge d\bar{z}_i\wedge d\bar{z}_j + \frac{\partial\overline{\Omega}_{ij}}{\partial \bar{z}_k} d\bar{z}_k\wedge d\bar{z}_i\wedge d\bar{z}_j\right).
 \end{align*}
 This equation implies:
 \begin{gather*}
  \frac{\partial\Omega_{ij}}{\partial \bar{z}_k} = 0\quad\forall i,j,k\in\{1,\ldots, 2m\}.
 \end{gather*}
 Thus, the coefficients $\Omega_{ij}$ are holomorphic functions on $U$. As the last argument can be repeated for any holomorphic chart of $X$, the form $\Omega$ itself is holomorphic. Therefore, its exterior derivative $d\Omega$ is also a holomorphic form. In particular, $d\Omega$ satisfies:
 \begin{gather*}
  d\Omega (J\cdot,\cdot,\cdot) = i\cdot d\Omega.
 \end{gather*}
 As the exterior derivative is $\mathbb{C}$-linear, the decomposition of $d\Omega$ into real and imaginary part amounts to $d\Omega = d\Omega_R + id\Omega_I$. Combining this decomposition with the previous equation gives us
 \begin{gather*}
  d\Omega_I = -d\Omega_R (J\cdot,\cdot,\cdot) = 0,
 \end{gather*}
 where we have used the closedness of $\Omega_R$ again. This shows that $\Omega$ has the desired properties concluding the proof.
\end{proof}

\begin{Rem}[Closedness of $\Omega_R$]\label{rem:closedness_of_Omega_R}
 Note that the closedness of $\Omega_R$ is crucial for Theorem \autoref{thm:rel_HSM_PHSM}: if $(X,\Omega = \Omega_R + i\Omega_I)$ is a HSM and $f:X\to\mathbb{R}_+$ is a positive and smooth function on $X$, then $f\cdot\Omega_R$ is still a non-degenerate $2$-form on $X$ which is anticompatible with the integrable complex structure $J$, however, neither $f\cdot \Omega_R$ nor $f\cdot \Omega_I$ are necessarily closed. In fact, $f\cdot \Omega$ is, in general, not even holomorphic.
\end{Rem}

Of course, we can also formulate Theorem \autoref{thm:rel_HSM_PHSM} for Hamiltonian systems:

\begin{Cor}[Relation between HHSs and PHHSs]\label{cor:rel_HHS_PHHS}
 Let $(X,J;\Omega_R, \mH_R)$ be a PHHS with $2$-forms $\Omega_I\coloneqq -\Omega_R (J\cdot,\cdot)$ and $\Omega\coloneqq \Omega_R + i\Omega_I$ as well as a function $\mH\coloneqq \mH_R + i\mH_I$, where $\mH_I$ is any primitive of the $1$-form $\Omega_R (J(X^{\Omega_R}_{\mH_R}),\cdot)$. Then, the following statements are equivalent:
 \begin{enumerate}
  \item $(X,\Omega, \mH)$ is a HHS with complex structure $J$.
  \item $\Omega_I$ is closed, $d\Omega_I = 0$.
  \item $J$ is integrable.
 \end{enumerate}
\end{Cor}

\begin{proof}
 Corollary \autoref{cor:rel_HHS_PHHS} is a direct consequence of Theorem \autoref{thm:rel_HSM_PHSM} and Remark \autoref{rem:H_pseudo-holo}.
\end{proof}

We can interpret Theorem \autoref{thm:rel_HSM_PHSM} and Corollary \autoref{cor:rel_HHS_PHHS} as follows: the integrability of the almost complex structure $J$ of a PHSM $(X,J;\Omega_R)$ or a PHHS $(X,J;\Omega_R,\mH_R)$ is completely measured by the closedness of the imaginary part $\Omega_I\coloneqq -\Omega_R (J\cdot,\cdot)$ and vice versa. Moreover, these quantities are the only local invariants of a PHSM or a PHHS: we know by Darboux's theorem for HSMs (cf. Theorem \autoref{thm:holo_Darboux}) that any HSM can locally be brought into standard form. Similarly, we will see in \autoref{subsec:deforming_HHS} that (regular) HHSs can locally also be brought into standard form. The existence of coordinates in which some geometrical object assumes a standard form implies that said geometrical object exhibits no local invariant. In this sense, Theorem \autoref{thm:rel_HSM_PHSM} and Corollary \autoref{cor:rel_HHS_PHHS} state that the Nijenhuis tensor $N_J$ of $J$ or, equivalently, the exterior derivative $d\Omega_I$ are the only local invariants of PHSMs and (regular) PHHSs. For general PHHSs, the Nijenhuis tensor and the behavior of the Hamiltonian near singular points are the only local invariants.

\newpage
\section{Construction of Proper PHHSs and Deformation of HHSs}
\label{sec:deformation}

This section serves two purposes. The first is to present examples of proper PHHSs. In fact, we provide a general method for constructing PHHSs out of HHSs (cf. \autoref{subsec:constructing_PHHS}), which, for instance, allows us to equip the holomorphic cotangent bundle of a complex manifold with a (non-canonical) PHHS-structure. Secondly, we study the ``size'' of the set of proper PHHSs within the set of all PHHSs. The main result of this investigation is that proper PHHSs are generic (cf. \autoref{subsec:deforming_HHS}). To prove this, we deform HHSs by proper PHHSs.
% We obtain such a deformation by locally expressing the given HHS in standard form and applying the construction from \autoref{subsec:constructing_PHHS} afterwards to deform this standard form by proper PHHSs.

\subsection{Constructing proper PHHSs out of HHSs}
\label{subsec:constructing_PHHS}

In this subsection, we wish to construct examples of proper PHHSs. The basic idea of the construction is to start with a HHS $(X,\Omega,\mH)$ and turn it into a PHHS by transforming its complex structure $J$ into an almost complex structure $J_g \coloneqq I_g\circ J\circ I_g$ using a $\Omega_R$-\underline{compatible}\footnote{Here, $I_g$ is $\Omega_R$-compatible iff $\Omega_R (I_g\cdot, I_g\cdot) = \Omega_R$. In the literature on symplectic geometry, this term is used differently. Usually, if $(X,\Omega_R)$ is a symplectic manifold and $I_g$ is an almost complex structure on $X$, one says that $I_g$ is $\Omega_R$-compatible if $-g\coloneqq \Omega_R (\cdot,I_g\cdot)$ is a Riemannian metric. This is stronger than our condition, since our condition only implies that $-g$ is a semi-Riemannian metric, but not necessarily positive definite. Also note that we use $g\coloneqq \Omega_R (I_g\cdot,\cdot)$ instead of $-g$ in the following computations, since the sign is not important for us.} almost complex structure $I_g$. Let us make this idea precise:\\
Pick a HHS $(X,\Omega = \Omega_R + i\Omega_I,\mH = \mH_R + i\mH_I)$ with complex structure $J$. Next, choose an almost complex structure $I_g$ on $X$ satisfying $\Omega_R (I_g\cdot, I_g\cdot) = \Omega_R$ and consider the smooth $(1,1)$-tensor field $J_g\coloneqq I_g\circ J\circ I_g$ on $X$. Clearly, $J_g$ is an almost complex structure on $X$:
\begin{gather*}
 J^2_g = \left(I_g\circ J\circ I_g\right)\circ \left(I_g\circ J\circ I_g\right) = - I_g\circ J^2\circ I_g = I^2_g = -1.
\end{gather*}
Moreover, $J_g$ is still $\Omega_R$-anticompatible:
\begin{gather*}
 \Omega_R (J_g\cdot, J_g\cdot) = \Omega_R (J\circ I_g\cdot, J\circ I_g\cdot) = -\Omega_R (I_g\cdot, I_g\cdot) = -\Omega_R.
\end{gather*}
Thus, $(X,J_g;\Omega_R)$ is a PHSM. If $I_g$ was chosen such that the $1$-form $\Omega_R (J_g (X^{\Omega_R}_{\mH_R}),\cdot)$ is exact, then $(X,J_g;\Omega_R,\mH_R)$ is even a PHHS. The PHHS constructed this way is, in general, proper. To check that, it suffices by Corollary \autoref{cor:rel_HHS_PHHS} to compute the exterior derivative of\linebreak $\Omega^g_I\coloneqq -\Omega_R (J_g\cdot,\cdot)$.\\
In practice, one can find $\Omega_R$-compatible almost complex structures $I_g$ by picking suitable semi-Riemannian metrics $g\coloneqq \Omega_R (I_g\cdot,\cdot)$. Given any HHS $(X,\Omega,\mH)$, however, an arbitrary $\Omega_R$-compatible almost complex structure $I_g$ is usually not compatible with $\mH_R$ in the sense explained above. Most often, it is simpler to \underline{first} pick an almost complex structure $I_g$ and \underline{afterwards} pick a suitable real function $\mH_R$. Since $X$ is contractible in most examples we wish to study, e.g. local considerations and $X = \mathbb{C}^{2m}$, one can often find suitable $\mH_R$ by solving the differential equation $d[\Omega_R (J_g (X^{\Omega_R}_{\mH_R}),\cdot)] = 0$.\\
To illustrate the construction, let us consider the simplest non-trivial example:

\begin{Ex}[PHHS on $X = \mathbb{C}^2$]\label{ex:constructing_PHHS}\normalfont
 Let $(X,\Omega,\mH)$ be the HHS with $X = \mathbb{C}^2$, $\Omega = dz_2\wedge dz_1$, where $(z_1,z_2)\in\mathbb{C}^2$, and $\mH (z_1,z_2) = i\cdot z_1$. With the decomposition $z_j = x_j + i\cdot y_j$, we obtain:
 \begin{gather*}
  \Omega_R = dx_2\wedge dx_1 - dy_2\wedge dy_1;\quad \mH_R = -y_1.
 \end{gather*}
 The complex structure $J$ of the HHS $(X,\Omega,\mH)$ is simply given by multiplication with $i$:
 \begin{gather*}
  J (\partial_{x_j})\coloneqq \partial_{y_j};\quad J (\partial_{y_j})\coloneqq -\partial_{x_j}.
 \end{gather*}
 Now we pick the following semi-Riemannian metric $g$ on $\mathbb{C}^2$:
 \begin{gather*}
  g(\pa_{x_1},\pa_{x_1}) = g(\pa_{x_2},\pa_{x_2})^{-1} = f;\quad g(\pa_{y_1},\pa_{y_1}) = g(\pa_{y_2},\pa_{y_2})^{-1} = h;\\
  g(\pa_{x_1},\pa_{x_2}) = g(\pa_{y_1},\pa_{y_2}) = g(\pa_{x_i},\pa_{y_j}) = 0,
 \end{gather*}
 where $f,h:\mathbb{C}^2\to\mathbb{R}$ are smooth, nowhere-vanishing functions. The corresponding almost complex structure $I_g$ is given by:
 \begin{gather*}
  I_g (\pa_{x_1}) = f\pa_{x_2},\ I_g (\pa_{x_2}) = -f^{-1}\pa_{x_1};\quad I_g (\pa_{y_1}) = -h\pa_{y_2},\ I_g (\pa_{y_2}) = h^{-1}\pa_{y_1}.
 \end{gather*}
 One easily checks that $I_g$ is indeed an almost complex structure and $\Omega_R$-compatible. We now employ the notation $r\coloneqq f/h$ and set $J_g\coloneqq I_g\circ J\circ I_g$. Then, $J_g$ is given by
 \begin{gather*}
  J_g (\pa_{x_1}) = r\pa_{y_1},\ J_g (\pa_{x_2}) = r^{-1}\pa_{y_2};\quad J_g (\pa_{y_1}) = -r^{-1}\pa_{x_1},\ J_g (\pa_{y_2}) = -r\pa_{x_2}.
 \end{gather*}
 By construction, $J_g$ is an almost complex structure and $\Omega_R$-anticompatible. Thus,\linebreak $(X,J_g;\Omega_R)$ is a PHSM. We find for the induced $2$-form $\Omega^g_I\coloneqq -\Omega_R (J_g\cdot,\cdot)$:
 \begin{gather*}
  \Omega^g_I = r^{-1}dx_2\wedge dy_1 + rdy_2\wedge dx_1.
 \end{gather*}
 Hence, the exterior derivative of $\Omega^g_I$ can be expressed as
 \begin{gather*}
  d\Omega^g_I %= -r^{-2}dr\wedge dx_2\wedge dy_1 + dr\wedge dy_2\wedge dx_1
  = dr\wedge\left( dy_2\wedge dx_1 - r^{-2}dx_2\wedge dy_1\right).
 \end{gather*}
 In general, the exterior derivative $d\Omega^g_I$ does not vanish. For instance, we can set\linebreak $f (z_1, z_2) = 1$ and $h(z_1, z_2) = e^{x_1}$ resulting in $r (z_1, z_2) = e^{-x_1}$. This choice yields
 \begin{gather*}
  d\Omega^g_I = e^{x_1} dx_1\wedge dx_2\wedge dy_1
 \end{gather*}
 which does not vanish at any point of $X$. Thus, $(X,J_g;\Omega_R)$ is a proper PHSM for the presented choice of $f$ and $h$.\\
 %Next, we want to define a smooth real function $H$ on $M$ such that $(M,\omega, H, J_g)$ becomes a proper $J_g$-anticompatible Hamiltonian system. This means that the function $H$ needs to be chosen such that $\iota_{J_g (X_H)}\omega$ is exact where $X_H$ is the Hamiltonian vector field of $(M,\omega, H)$. If the first cohomology of $M$ vanishes, which is the case for $M = \mathbb{C}^2$, as $\mathbb{C}^2$ is contractible, it suffices to check that $\iota_{J_g (X_H)}\omega$ is closed. There is no generic way to ensure this, so we will only work in the example $M = \mathbb{C}^2$ from now on.\\
 Next, we check whether the $1$-form $\Omega_R (J_g (X^{\Omega_R}_{\mH_R}),\cdot)$ is exact. Since $X = \mathbb{C}^2$ is contractible, it suffices to check whether $\Omega_R (J_g (X^{\Omega_R}_{\mH_R}),\cdot)$ is closed. For the sake of generality, we first perform the computation for any smooth real function $H:\mathbb{C}^2\to\mathbb{R}$ and afterwards insert\linebreak $H = \mH_R = -y_1$. We start by determining the Hamiltonian vector field $X_H\equiv X^{\Omega_R}_H$. It can be written as
 \begin{gather*}
  X_H = (\pa_{x_2}H)\cdot\pa_{x_1} - (\pa_{x_1}H)\cdot\pa_{x_2} - (\pa_{y_2}H)\cdot\pa_{y_1} + (\pa_{y_1}H)\cdot\pa_{y_2}.
 \end{gather*}
 Applying $J_g$ to $X_H$ yields:
 \begin{gather*}
  J_g (X_H) = (r^{-1}\pa_{y_2}H)\cdot \pa_{x_1} - (r\pa_{y_1}H)\cdot \pa_{x_2} + (r\pa_{x_2}H)\cdot \pa_{y_1} - (r^{-1}\pa_{x_1}H)\cdot\pa_{y_2}.
 \end{gather*}
 Contracting $J_g (X_H)$ with $\Omega_R$ gives:
 \begin{align*}
  \Omega_R (J_g (X_H),\cdot) &= -(r\pa_{y_1}H)\cdot dx_1 - (r^{-1}\pa_{y_2}H)\cdot dx_2 + (r^{-1}\pa_{x_1}H)\cdot dy_1 + (r\pa_{x_2}H)\cdot dy_2\\
  &\equiv w_{x_1} dx_1 + w_{x_2} dx_2 + w_{y_1} dy_1 + w_{y_2} dy_2.
 \end{align*}
 For $\Omega_R (J_g (X_H),\cdot)$ to be closed, the coefficients $w$ need to satisfy the following conditions:
 \begin{gather*}
  \pa_{x_i} w_{x_j} = \pa_{x_j} w_{x_i};\quad \pa_{y_i} w_{y_j} = \pa_{y_j} w_{y_i};\quad \pa_{x_i} w_{y_j} = \pa_{y_j} w_{x_i}.
 \end{gather*}
 Let us now set $H = \mH_R = -y_1$. Then, all coefficients except for $w_{x_1} = r$ vanish. In particular, $\Omega_R (J_g (X_H),\cdot)$ is exact iff $r$ only depends on $x_1$. In this case, the primitive $\mH^g_I$ of $\Omega_R (J_g (X_H),\cdot)$ is given by $R$, where $R$ only depends on $x_1$ and satisfies $\pa_{x_1}R = r$. For instance, we can again set $f (z_1, z_2) = 1$ and $h(z_1, z_2) = e^{x_1}$ leading to\linebreak $r (z_1, z_2) = e^{-x_1}$. Then, $(X,J_g;\Omega_R,\mH_R)$ is a proper PHHS for this choice, where\linebreak $\mH^g_I = -e^{-x_1}$ is one possibility for the imaginary part of the Hamiltonian.
\end{Ex}

Of course, the construction presented above is not the only way to obtain a PHSM out of a HSM. For instance, the simple structure of the example $(X = \mathbb{C}^2, \Omega = dz_2\wedge dz_1)$ allows us to transform the standard complex structure $J$ into a new $\Omega_R$-anticompatible almost complex structure $J_\varphi$ by rotation of the axes:
\begin{alignat*}{2}
 J_\varphi (\pa_{x_1})&\coloneqq \cos (\varphi)\pa_{y_1} - \sin (\varphi)\pa_{y_2},\quad &&J_\varphi (\pa_{x_2})\coloneqq \sin (\varphi)\pa_{y_1} + \cos (\varphi)\pa_{y_2};\\
 J_\varphi (\pa_{y_1})&\coloneqq -\cos (\varphi)\pa_{x_1} - \sin (\varphi)\pa_{x_2},\quad &&J_\varphi (\pa_{y_2})\coloneqq \sin (\varphi)\pa_{x_1} - \cos (\varphi)\pa_{x_2},
\end{alignat*}
where $\varphi:\mathbb{C}^2\to\mathbb{R}$ is any smooth function. The resulting PHSM $(X, J_\varphi; \Omega_R)$ is also, in general, proper.\\
The reason why we focus our attention on the construction involving $I_g$ is that this construction possesses an interesting connection to hyperk{\"a}hler manifolds. Consider Example \autoref{ex:constructing_PHHS} again. If we choose $g$ to be the standard Euclidean metric $\delta$ on $\mathbb{C}^2\cong\mathbb{R}^4$, i.e., $f \equiv h\equiv 1$, then $I_\delta$ anti-commutes with the standard complex structure $J$. Thus, $J_\delta$ coincides with $J$ and our construction does not change the base HSM $(\mathbb{C}^2, dz_2\wedge dz_1)$. In fact, $(\mathbb{C}^2, dz_2\wedge dz_1)$ comes from the hyperk{\"a}hler structure on $\mathbb{C}^2$ defined by $\delta$, $I_\delta$, $J$, and $K\coloneqq I_\delta\circ J$. Here, we call a collection $(X,g,I,J,K)$ a \textbf{hyperk{\"a}hler manifold} iff $X$ is a smooth manifold with Riemannian metric $g$ and integrable almost complex structures $I$, $J$, and $K$ on it such that $I$, $J$, and $K$ fulfill the relation $I^2 = J^2 = K^2 = IJK = -1$ and $\omega^a\coloneqq -g(a\cdot,\cdot)$ is a symplectic $2$-form on $X$ for every $a\in\{I,J,K\}$. Every hyperk{\"a}hler manifold $(X,g,I,J,K)$ admits three holomorphic symplectic forms, namely:
\begin{gather*}
 \Omega^I\coloneqq \omega^J + i\omega^K;\quad \Omega^J\coloneqq\omega^K + i\omega^I;\quad \Omega^K\coloneqq\omega^I + i\omega^J,
\end{gather*}
where $\Omega^a$ is holomorphic with respect to $a\in\{I,J,K\}$. For the hyperk{\"a}hler manifold $(X = \mathbb{C}^2,\delta, I_\delta, J, K)$, the holomorphic symplectic form $\Omega = dz_2\wedge dz_1$ coincides with $-i\Omega^J$.\\
Now, note that the space of pairs $(f,h)$ of smooth, nowhere-vanishing functions on $\mathbb{C}^2$ has four connected components which are isomorphic via $(f,h)\mapsto (-f,h)$ and\linebreak $(f,h)\mapsto (f,-h)$. Furthermore, each connected component is contractible. Thus, we can reach any pair $(f,h)$ in this space by a path starting at $(1,1)$ and a discrete transformation. Next, consider such a path in this space, i.e., a smooth $1$-parameter family of nowhere-vanishing functions $f^\varepsilon$ and $h^\varepsilon$ on $\mathbb{C}^2$ such that $f^0\equiv h^0 \equiv 1$. This family induces a $1$-parameter family $J^\varepsilon$ of almost complex structures on $\mathbb{C}^2$ satisfying $J^0 = J$.
%which we can interpret as a deformation of the HHS $(\mathbb{C}^2,dz_2\wedge dz_1)$ by PHHSs\footnote{Confer \autoref{subsec:deforming_HHS} for the precise definition of deformations.}.
By construction, $J^\varepsilon$ itself is defined via a $1$-parameter family of almost complex structures $I^\varepsilon$ satisfying $I^0 = I_\delta$. With our previous knowledge, we can interpret $I^\varepsilon$ as some sort of deformation of the hyperk{\"a}hler manifold $(X = \mathbb{C}^2,\delta, I_\delta, J, K)$. Hence, we can say that every almost complex structure $J_g$ as in Example \autoref{ex:constructing_PHHS} giving rise to a PHSM comes from such a ``deformation'' of a hyperk{\"a}hler manifold up to a discrete transformation $J_g\mapsto -J_g$. In general, this also seems to be the best application of our construction: pick a HSM $(X,\Omega = -i\Omega^J)$ coming from a hyperk{\"a}hler manifold $(X,g,I,J,K)$ and then deform $I$ as described to find PHSMs.\\
Next, we show that the construction $J_g = I_g\circ J\circ I_g$ of PHSMs is applicable to a rather large class of complex manifolds $X$, namely the class of holomorphic cotangent bundles $X = T^{\ast, (1,0)}Y$ of complex manifolds $Y$. For this, let $Y$ be a complex manifold and $X\coloneqq T^{\ast, (1,0)}Y$ be its holomorphic cotangent bundle. The complex structure of $Y$ induces a canonical complex structure $J$ on $X$. Furthermore, $X$ as a holomorphic cotangent bundle possesses a canonical HSM structure with $2$-form $\Omega_\text{can} = \Omega_R + i\Omega_I$. We observe that the real part $\Omega_R$ of $\Omega_\text{can}$ can be identified with the canonical symplectic $2$-form $\omega_\text{can}$ on the real cotangent bundle $T^\ast Y$:

\begin{Prop}[$F^\ast\omega_\text{can} = \Omega_R$]\label{prop:F}
 Let $Y$ be a complex manifold. Then, the map\linebreak $F:T^{\ast, (1,0)}Y\to T^\ast Y$
 \begin{gather*}
  F(\alpha)\coloneqq \text{\normalfont Re} (\alpha)\quad\forall \alpha\in T^{\ast, (1,0)}Y,
 \end{gather*}
 where $\text{\normalfont Re} (\alpha)$ denotes the real part of $\alpha$, is a bundle isomorphism between smooth real vector bundles. In particular, $F$ is a diffeomorphism between the smooth manifolds $T^{\ast, (1,0)}Y$ and $T^\ast Y$ satisfying $F^\ast \omega_\text{can} = \Omega_R$.
\end{Prop}

\begin{proof}
 Take the notations from above. Let $(Q_1 = Q_{x,1} + iQ_{y,1},\ldots, Q_n = Q_{x,n} + iQ_{y,n}) = \psi$ be a holomorphic chart of $Y$. Then, $T^{\ast, (1,0)}\psi$ and $T^{\ast}\psi$ are (smooth) charts of $T^{\ast, (1,0)}Y$ and $T^\ast Y$, respectively:
 \begin{align*}
  \left(T^{\ast, (1,0)}\psi\right)^{-1} (\tilde{Q}_{x,1},\ldots, P_{y,n}) &\coloneqq \sum^n_{j=1} P_j dQ_j\vert_{\psi^{-1}(\tilde{Q}_{1},\ldots, \tilde{Q}_{n})}\\
  = \sum^n_{j = 1} &P_{x,j}dQ_{x,j}\vert_{\ldots} - P_{y,j}dQ_{y,j}\vert_{\ldots} + i\sum^n_{j=1}P_{x,j}dQ_{y,j}\vert_{\ldots} + P_{y,j}dQ_{x,j}\vert_{\ldots},\\
  \left(T^{\ast}\psi\right)^{-1} (\tilde{q}_{x,1},\ldots, p_{y,n}) &\coloneqq \sum^n_{j = 1} p_{x,j}dQ_{x,j}\vert_{\psi^{-1}(\tilde{q}_{1},\ldots, \tilde{q}_{n})} + p_{y,j}dQ_{y,j}\vert_{\psi^{-1}(\tilde{q}_{1},\ldots, \tilde{q}_{n})}.
 \end{align*}
 We denote the components of $T^{\ast, (1,0)}\psi$ by $T^{\ast, (1,0)}\psi = (Q_{x,1},\ldots, P_{y,n})$, while we denote the components of $T^{\ast}\psi$ by $T^{\ast}\psi = (q_{x,1},\ldots, p_{y,n})$. In these coordinates, $\Omega_\text{can}$ and $\omega_\text{can}$ can be expressed as:
 \begin{align*}
  \Omega_\text{can} &= \sum^n_{j = 1} dP_j\wedge dQ_j = \sum^n_{j=1} dP_{x,j}\wedge dQ_{x,j} - dP_{y,j}\wedge dQ_{y,j} + i\sum^n_{j=1}dP_{x,j}\wedge dQ_{y,j} + dP_{y,j}\wedge dQ_{x,j},\\
  \omega_\text{can} &= \sum^n_{j=1} dp_{x,j}\wedge dq_{x,j} + dp_{y,j}\wedge dq_{y,j}.
 \end{align*}
 Expressing $F$ in these coordinates gives:
 \begin{gather*}
  T^\ast\psi\circ F\circ \left(T^{\ast, (1,0)}\psi\right)^{-1} (\tilde{Q}_{x,j}, \tilde{Q}_{y,j}, P_{x,j}, P_{y,j}) = (\tilde{Q}_{x,j}, \tilde{Q}_{y,j}, P_{x,j}, -P_{y,j}).
 \end{gather*}
 In total, we obtain:
 \begin{gather*}
  F^\ast\omega_\text{can} = \sum^n_{j=1} dP_{x,j}\wedge dQ_{x,j} - dP_{y,j}\wedge dQ_{y,j} = \text{Re} (\Omega_\text{can}) = \Omega_R.
 \end{gather*}
\end{proof}

Now choose a semi-Riemannian metric $g_Y$ on $Y$. In \autoref{app:almost_complex_structures}, we show that any semi-Riemannian metric $g_Y$ on $Y$ induces an almost complex structure $J^\ast_{\nabla^{g_Y}}$ on $T^\ast Y$ which is $\omega_\text{can}$-compatible. Using $F$, we can transfer this almost complex structure from $T^\ast Y$ to $X$. We denote the result by $I_g\coloneqq dF^{-1}\circ J^\ast_{\nabla^{g_Y}}\circ dF$. By Proposition \autoref{prop:F}, $I_g$ on $X$ is also $\Omega_R$-compatible. Thus, $J_g\coloneqq I_g\circ J\circ I_g$ is $\Omega_R$-anticompatible and $(X,J_g;\Omega_R)$ is a PHSM. Note that $J_g$ depends on a choice of metric $g_Y$.\\
In general, this PHSM is proper, since we have not imposed any relation between $g_Y$ and the complex structure on $Y$. However, there is a special case in which $J_g$ is integrable, namely if $g_Y = h_R$ is the real part of a holomorphic metric $h = h_R + ih_I$ on $Y$:

\begin{Lem}[$g_Y = h_R\ \Rightarrow\ J$ and $I_g$ commute]
 Let $Y$ be a complex manifold with holomorphic metric $h = h_R + ih_I$. Then, $I_g$ obtained from the construction above for $g_Y = h_R$ commutes with the complex structure $J$ on $X\coloneqq T^{\ast, (1,0)}Y$. In particular, $J_g \coloneqq I_g\circ J\circ I_g = -J$.
\end{Lem}

\begin{proof}
 The idea of the proof is to choose coordinates in which $J$ and $I_g$ take a simple form. Let $p\in Y$ be any point. We start by choosing holomorphic coordinates\linebreak $(Q_1 = Q_{x,1} + iQ_{y,1},\ldots, Q_n = Q_{x,n} + iQ_{y,n}) = \psi$ of $Y$ near $p$ which, at the same time, are normal coordinates of $(Y,h_R)$ near $p$. In \autoref{app:holo_connection}, we show that such coordinates exist by considering normal coordinates of the holomorphic Levi-Civita connection $\nabla^h$ of the holomorphic metric $h$. The holomorphic normal coordinates $\psi$ then give rise to holomorphic coordinates $T^{\ast, (1,0)}\psi = (Q_1,\ldots, P_n = P_{x,n} + iP_{y,n})$ of $X$. As $J$ is the complex structure of $X$, $J$ takes the following form in $T^{\ast, (1,0)}\psi$ ($\alpha\in T^{\ast, (1,0)}_p Y$):
 \begin{alignat*}{2}
  J\left(\partial_{Q_{x,j}}\vert_\alpha\right) &= \partial_{Q_{y,j}}\vert_\alpha,\quad J\left(\partial_{Q_{y,j}}\vert_\alpha\right) &&= -\partial_{Q_{x,j}}\vert_\alpha,\\
  J\left(\partial_{P_{x,j}}\vert_\alpha\right) &= \partial_{P_{y,j}}\vert_\alpha,\quad J\left(\partial_{P_{y,j}}\vert_\alpha\right) &&= -\partial_{P_{x,j}}\vert_\alpha.
 \end{alignat*}
 In \autoref{app:almost_complex_structures}, we show that $I_g$ takes the following form\footnote{To be precise, we have determined the form of $J^\ast_{\nabla^{h_R}}$ in coordinates $T^\ast \psi$ in \autoref{app:almost_complex_structures}. We obtain the form of $I_g\coloneqq dF^{-1}\circ J^\ast_{\nabla^{h_R}}\circ dF$ by applying $F$. Mind the change of signs for $P_{y,j}$ due to $F$.} in $T^{\ast, (1,0)}\psi$:
 \begin{alignat*}{2}
  I_g\left(\partial_{Q_{x,j}}\vert_\alpha\right) &= -\partial_{P_{x,j}}\vert_\alpha,\quad I_g\left(\partial_{Q_{y,j}}\vert_\alpha\right) &&= -\partial_{P_{y,j}}\vert_\alpha,\\
  I_g\left(\partial_{P_{x,j}}\vert_\alpha\right) &= \partial_{Q_{x,j}}\vert_\alpha,\quad I_g\left(\partial_{P_{y,j}}\vert_\alpha\right) &&= \partial_{Q_{y,j}}\vert_\alpha.
 \end{alignat*}
 Now, one can check by direct computation that $J$ and $I_g$ commute at the point $\alpha$. Since the same argument can be repeated for every $\alpha\in T^{\ast, (1,0)}_pY$ and $p\in Y$, the result is true for all points of $X$ concluding the proof.
\end{proof}

Before we conclude this subsection and turn our attention to the deformation of HHSs, we should briefly illustrate one feature of proper PHHSs which is absent in the case of HHSs: the fact that the Hamiltonian vector fields $X^{\Omega_R}_{\mH_R}$ and $J(X^{\Omega_R}_{\mH_R})$ of a PHHS $(X,J;\Omega_R,\mH_R)$ do not need to be $J$-preserving (cf. Remark \autoref{rem:x_0-dependance}). Example \autoref{ex:constructing_PHHS} beautifully demonstrates that feature. To see this, let us first develop a criterium that tells us in which cases the vector fields $X^{\Omega_R}_{\mH_R}$ and $J(X^{\Omega_R}_{\mH_R})$ are $J$-preserving:

\begin{Prop}[When $X^{\Omega_R}_{\mH_R}$ and $J(X^{\Omega_R}_{\mH_R})$ are $J$-preserving]\label{prop:J-preserving_criteria}
 Let $(X,J;\Omega_R,\mH_R)$ be a PHHS with Hamiltonian vector field $X_\mH = 1/2 (X^{\Omega_R}_{\mH_R} - iJ(X^{\Omega_R}_{\mH_R}))$ and $2$-form\linebreak $\Omega_I = -\Omega_R (J\cdot,\cdot)$. Then, $X^{\Omega_R}_{\mH_R}$ is $J$-preserving if and only if $d\Omega_I (X^{\Omega_R}_{\mH_R},\cdot,\cdot) \equiv 0$. Similarly, $J(X^{\Omega_R}_{\mH_R})$ is $J$-preserving if and only if $d\Omega_I (J(X^{\Omega_R}_{\mH_R}),\cdot,\cdot) \equiv 0$.
\end{Prop}

\begin{proof}
 Let $(X,J;\Omega_R,\mH_R)$ be a PHHS with $2$-form $\Omega_I = -\Omega_R (J\cdot,\cdot)$ and Hamiltonian vector field $X_\mH = 1/2 (X^{\Omega_R}_{\mH_R} - iJ(X^{\Omega_R}_{\mH_R}))$. Take $V\in\{X^{\Omega_R}_{\mH_R}, J(X^{\Omega_R}_{\mH_R})\}$. We want to determine in which cases $V$ is $J$-preserving. Thereto, we use Proposition 2.10 in Chapter IX of \cite{kobayashi1969}:\linebreak A (real) vector field $V$ on a smooth manifold $X$ with almost complex structure $J$ is $J$-preserving if and only if the following equation is fulfilled for every vector field $W$:
 \begin{gather*}
  [V, J(W)] = J([V,W]).
 \end{gather*}
 Since $\Omega_R$ is non-degenerate, we find that $V$ is $J$-preserving if and only if
 \begin{gather*}
  \iota_{[V, J(W)]}\Omega_R - \iota_{J([V,W])}\Omega_R = 0
 \end{gather*}
 for every vector field $W$ on $X$. Let us now compute the left-hand side of this equation. First, we obtain by definition of $\Omega_I$:
 \begin{gather*}
  \iota_{J([V,W])}\Omega_R = -\iota_{[V,W]}\Omega_I.
 \end{gather*}
 Now remember the relation (cf. Proposition 3.10 in Chapter I of \cite{kobayashi1963}):
 \begin{gather*}
  \iota_{[V,W]} = [L_V,\iota_W].
 \end{gather*}
 Furthermore, recall that by Definition \autoref{def:PHHS_1} the $1$-forms $\Omega_R (V,\cdot)$ and\linebreak $\Omega_I (V,\cdot) = -\Omega_R (J(V),\cdot)$ are exact. Together with $d\Omega_R = 0$ and Cartan's magic formula, this implies:
 \begin{gather*}
  L_V\Omega_R = 0;\quad L_V\Omega_I = \iota_V d\Omega_I,
 \end{gather*}
 where $L_V$ denotes the Lie derivative of $V$. These relations allow us to compute:
 \begin{align*}
  \iota_{[V,J(W)]}\Omega_R &= [L_V,\iota_{J(W)}]\Omega_R = L_V\left(\iota_{J(W)}\Omega_R\right) - \iota_{J(W)}\left(L_V\Omega_R\right)\\
  &= -L_V\left(\iota_W\Omega_I\right) = -[L_V,\iota_W]\Omega_I - \iota_W\left(L_V\Omega_I\right)\\
  &= -\iota_{[V,W]}\Omega_I - \iota_W\iota_V d\Omega_I = \iota_{J([V,W])}\Omega_R - \iota_W\iota_V d\Omega_I.
 \end{align*}
 In total, we find that $V$ is $J$-preserving if and only if
 \begin{gather*}
  \iota_{[V, J(W)]}\Omega_R - \iota_{J([V,W])}\Omega_R = - \iota_W\iota_V d\Omega_I = 0
 \end{gather*}
 for every vector field $W$ on $X$ concluding the proof.
\end{proof}

Now let us consider Example \autoref{ex:constructing_PHHS} again with $f = 1$, $h = e^{x_1}$, and $H = \mH_R = -y_1$. Then, the Hamiltonian vector fields and the exterior derivative of $\Omega^g_I$ are given by:
\begin{gather*}
 X^{\Omega_R}_{\mH_R} = -\pa_{y_2};\quad J_g(X^{\Omega_R}_{\mH_R}) = e^{-x_1}\pa_{x_2};\quad d\Omega^g_I = e^{x_1} dx_1\wedge dx_2\wedge dy_1.
\end{gather*}
Using Proposition \autoref{prop:J-preserving_criteria}, we immediately see that $X^{\Omega_R}_{\mH_R}$ is $J_g$-preserving, while $J_g(X^{\Omega_R}_{\mH_R})$ is not. By choosing $H$ to be a linear combination of $\mH_R = -y_1$ and $\mH^g_I = -e^{-x_1}$, Example \autoref{ex:constructing_PHHS} even shows that neither $X^{\Omega_R}_{\mH_R}$ nor $J(X^{\Omega_R}_{\mH_R})$ of a PHHS $(X,J;\Omega_R,\mH_R)$ need to be $J$-preserving.

\subsection{Deforming HHSs by proper PHHSs}
\label{subsec:deforming_HHS}

In this subsection, we examine the question: ``How `large' is the set of proper PHHSs within the set of all PHHSs?'' The answer and the main result of this subsection is that proper PHHSs are generic if $\text{dim}_\mathbb{R}(X)>4$. To prove this, we first reduce the genericity of proper PHHSs to the claim that every HHS can be deformed by proper PHHSs. Afterwards, we show this claim in two steps. The first step is to locally bring every HHS into standard form. Secondly, we deform the HHS standard form by proper PHHSs within a small neighborhood.\\
We start by recalling the definition of a generic property. For a topological space $B$ and a subset $A\subset B$, we call the property that an element is contained in $A$ \textbf{generic} iff $B\backslash A$ is a meagre subset of $B$ in the sense of Baire. In particular, the property $a\in A$ is generic if $A$ is an open and dense subset of $B$. Now let $X$ be a smooth manifold. We introduce the following notations for the set of almost complex structures, of PHSMs, and of PHHSs on $X$:
\begin{align*}
 \mathcal{J}_\text{a.c.} (X)&\coloneqq \{J\in\Gamma (\text{End}(TX))\mid J^2 = -1\},\\
 \text{PHSM} (X)&\coloneqq \{(J,\Omega_R)\in\mathcal{J}_\text{a.c.}(X)\times\Omega^2 (X)\mid (X,J;\Omega_R)\text{ is a PHSM}\},\\
 \text{PHHS} (X)&\coloneqq \{(J,\Omega_R,\mH_R)\in\mathcal{J}_\text{a.c.}(X)\times\Omega^2 (X)\times C^\infty (X,\mathbb{R})\mid (X,J;\Omega_R,\mH_R)\text{ is a PHHS}\}.
\end{align*}
Similarly, we denote the set of complex structures, of HSMs, and of HHSs on $X$ by:
\begin{align*}
 \mathcal{J}_\text{c} (X)&\coloneqq \{J\in\mathcal{J}_\text{a.c.}(X)\mid J\text{ is integrable}\},\\
 \text{HSM} (X)&\coloneqq \{(J,\Omega_R)\in \text{PHSM} (X)\mid J\text{ is integrable}\},\\
 \text{HHS} (X)&\coloneqq \{(J,\Omega_R,\mH_R)\in \text{PHHS} (X)\mid J\text{ is integrable}\}.
\end{align*}
Lastly, we write for the set of proper, i.e., non-integrable almost complex structures, of proper PHSMs, and of proper PHHSs on $X$:
\begin{align*}
 \mathcal{J}_\text{p} (X)&\coloneqq \mathcal{J}_\text{a.c.} (X)\backslash \mathcal{J}_\text{c} (X),\\
 \text{PHSM}_\text{p} (X)&\coloneqq \text{PHSM} (X)\backslash \text{HSM} (X),\\
 \text{PHHS}_\text{p} (X)&\coloneqq \text{PHHS} (X)\backslash \text{HHS} (X).
\end{align*}
We equip every set with the topology induced by the compact-open topology. We wish to show that $\mathcal{J}_\text{p} (X)\subset \mathcal{J}_\text{a.c.} (X)$, $\text{PHSM}_\text{p} (X)\subset \text{PHSM} (X)$, and $\text{PHHS}_\text{p} (X)\subset \text{PHHS} (X)$ are open and dense subsets if $\text{dim}_\mathbb{R} (X)>4$. Clearly, the sets in question are open subsets, since all former mentioned subsets contain exactly those elements from their respective supersets for which the Nijenhuis tensor $N_J$ of $J$ does not vanish. The argument is completed by noting that $N_J\neq 0$ is an open condition.\\
It is left to show that the subsets are dense in their respective supersets. This can be done by showing that the ``integrable'' sets $\mathcal{J}_\text{c} (X)$, $\text{HSM} (X)$, and $\text{HHS}(X)$ are contained in the boundary of the ``proper'' sets $\mathcal{J}_\text{p}$, $\text{PHSM}_\text{p} (X)$, and $\text{PHHS}_\text{p} (X)$, respectively. The last statement is true if the ``integrable'' elements can be deformed by ``proper'' elements:

\begin{Def}[(Proper) deformation]\label{def:deformation}
 Let $X$ be a smooth manifold with almost complex structure $J$ on it. $(X,J^\varepsilon)$ is called a \textbf{deformation} of $(X,J)$ iff $J^\varepsilon$ describes a smooth\footnote{Here, smooth path means that the map $J^\varepsilon:X\times\mathbb{R}\to \text{\normalfont End}(TX)$ is smooth. Similar remarks apply to $\Omega^\varepsilon_R$ and $\mH^\varepsilon_R$.} path in $\mathcal{J}_\text{\normalfont a.c.} (X)$ with start point $J^0 = J$. Now let $\Omega_R$ be a $2$-form on $X$ such that $(X,J;\Omega_R)$ is a PHSM. Then, $(X,J^\varepsilon; \Omega^\varepsilon_R)$ is a deformation of $(X,J;\Omega_R)$ iff $(J^\varepsilon, \Omega^\varepsilon_R)$ describes a smooth path in $\text{\normalfont PHSM}(X)$ with start point $(J^0, \Omega^0_R) = (J,\Omega_R)$. If, additionally, $\mH_R$ is a function on $X$ such that $(X,J;\Omega_R,\mH_R)$ is a PHHS, we say $(X,J^\varepsilon; \Omega^\varepsilon,\mH^\varepsilon_R)$ is a deformation of $(X,J;\Omega_R,\mH_R)$ iff $(J^\varepsilon, \Omega^\varepsilon_R, \mH^\varepsilon_R)$ describes a smooth path in $\text{\normalfont PHHS}(X)$ with start point $(J^0, \Omega^0_R,\mH^0_R) = (J,\Omega_R,\mH_R)$. We call a deformation \textbf{proper} iff the corresponding $J^\varepsilon$ is not integrable for $\varepsilon\neq 0$.
\end{Def}

Before we continue with the proof of genericity, we shall quickly address one aspect concerning our definitions: in the definition of a deformation, we have neglected the imaginary parts $\Omega_I$ and $\mH_I$, even though they are crucial for the definition of HSMs and HHSs. One might wonder whether this is justified, i.e., whether every deformation of a PHSM or a PHHS automatically gives us a suitable $1$-parameter family $\Omega^\varepsilon_I$ and $\mH^\varepsilon_I$ of imaginary parts. Regarding the form $\Omega_I$, this is certainly true by simply setting $\Omega^\varepsilon_I\coloneqq -\Omega^\varepsilon_R (J^\varepsilon\cdot,\cdot)$. However, $\mH_I$ defined as a primitive of $\Omega_R (J(X^{\Omega_R}_{\mH_R}),\cdot)$ might be more problematic. It is not obvious why a smooth $1$-parameter family of exact $1$-forms $\alpha^\varepsilon$ should be the differential of a smooth $1$-parameter family of functions $f^\varepsilon$. Nevertheless, this is still true in the case of $X$ being contractible:

\begin{Prop}[Primitive of $\alpha^\varepsilon$]\label{prop:primitive}
 Let $X$ be a smooth contractible manifold, $\alpha^\varepsilon$ be a smooth $1$-parameter family of exact $1$-forms on $X$, and $f\in C^\infty (X,\mathbb{R})$ such that $\alpha^0 = df$. Then, there exists a smooth $1$-parameter family of functions $f^\varepsilon$ on $X$ such that $\alpha^\varepsilon = df^\varepsilon$ and $f^0 = f$.
\end{Prop}

\begin{proof}
 Let $X$, $\alpha^\varepsilon$, and $f$ be as above. $X$ is contractible, so there is a smooth map $F:[0,1]\times X\to X$ such that $F(0,x) = x_0$ and $F(1,x) = x$ for every $x\in X$ and some $x_0\in X$. For fixed $\varepsilon$, $F^\ast\alpha^\varepsilon$ is a closed $1$-form on $[0,1]\times X$ and can be expressed as (using $F_t:X\to X$, $F_t (x)\coloneqq F (t,x)$ for every $t\in [0,1]$):
 \begin{gather*}
  F^\ast\alpha^\varepsilon = F^\ast_t\alpha^\varepsilon + \beta^\varepsilon_t \cdot dt,
 \end{gather*}
 where $\beta^\varepsilon_t$ is a function on $X$ smoothly depending on $\varepsilon$ and $t$. For fixed $t$ and $\varepsilon$, $F^\ast_t\alpha^\varepsilon$ can be understood as a closed $1$-form on $X$. The closedness of $F^\ast\alpha^\varepsilon$ and $F^\ast_t\alpha^\varepsilon$ together with the expression for $F^\ast\alpha^\varepsilon$ implies:
 \begin{gather*}
  \frac{d}{dt}F^\ast_t\alpha^\varepsilon = d\beta^\varepsilon_t.
 \end{gather*}
 In total, we obtain:
 \begin{align*}
  \alpha^\varepsilon = \text{id}^\ast_X\alpha^\varepsilon - \text{const}^\ast_{x_0}\alpha^\varepsilon = F^\ast_1\alpha^\varepsilon - F^\ast_0\alpha^\varepsilon = \int\limits^1_0 \frac{d}{dt}(F^\ast_t\alpha^\varepsilon)\, dt = d\int\limits^1_0 \beta^\varepsilon_t dt\eqqcolon df^\varepsilon.
 \end{align*}
 Shifting $f^\varepsilon$ by a constant to match $f$ at $\varepsilon = 0$ concludes the proof.
\end{proof}

To show that the property ``proper'' is generic, we will only use local deformations. In particular, we can always shrink the neighborhood in which we deform an integrable structure such that it becomes contractible. Thus, we can rightfully neglect deformations of the imaginary parts $\Omega_I$ and $\mH_I$ here.\\
Let us return to the proof of genericity. As explained, it suffices to find a proper deformation of every ``integrable'' element to show genericity. We construct the desired proper deformations in two steps:
\begin{enumerate}
 \item We bring the ``integrable'' elements locally into standard form.
 \item We deform the standard form ``properly'' within a small neighborhood.
\end{enumerate}
Regarding the first step, it is clear what the standard forms of complex manifolds and HSMs are and how to achieve them. For complex manifolds, the standard form is just $X = \mathbb{C}^m$ with complex structure $J = i$ and every complex manifold can be brought into standard form by holomorphic charts. For HSMs, the standard form is $X = \mathbb{C}^{2n}$ together with $\Omega = \sum_j dP_j\wedge dQ_j$ and every HSM can be brought into standard form by holomorphic Darboux charts (cf. Theorem \autoref{thm:holo_Darboux} and \autoref{app:darboux}).\\
For a HHS $(X,\Omega,\mH)$, we observe that $\mH$ is either locally constant or regular at some point, i.e., there exists $x_0\in X$ such that $d\mH\vert_{x_0}\neq 0$. If the given HHS is locally constant near a point, then holomorphic Darboux coordinates describe the standard form of the HHS in question near that point. If the given HHS is regular near a point, then the following lemma brings it into standard form:

\begin{Lem}[Regular HHSs in standard form]\label{lem:HHS_in_standard_form}
 Let $(X,\Omega,\mH)$ be a HHS and\linebreak $x_0\in X$ be a point such that $d\mH\vert_{x_0}\neq 0$. Then, there exists a holomorphic chart\linebreak $\phi = (z_1,\ldots, z_{2n}):U\to V\subset\mathbb{C}^{2n}$ of $X$ near $x_0\in U$ such that $\Omega\vert_U = \sum^n_{j = 1} dz_{j+n}\wedge dz_j$ and $\mH\vert_U = z_{2n}$.
\end{Lem}

\begin{proof}
 Let $(X,\Omega,\mH)$ be a HHS and $x_0\in X$ be a point such that $d\mH\vert_{x_0}\neq 0$. Without loss of generality, we can assume $\mH (x_0) = 0$. The proof consists of three steps:
 \begin{enumerate}
  \item Construct a holomorphic function $G$ defined locally near $x_0$ satisfying\linebreak $\{\mH, G\}\coloneqq \Omega (X_\mH, X_G) = 1$.
  \item Find a holomorphic chart $\phi_3 = (z^\#_1,\ldots, z^\#_{2n}):U_3\to V_3$ of $X$ near $x_0$ such that $G\vert_{U_3} = z^\#_n$, $\mH\vert_{U_3} = z^\#_{2n}$, $X_\mH\vert_{U_3} = \pa_{z^\#_n}$, and $X_G\vert_{U_3} = -\pa_{z^\#_{2n}}$.
  \item We have $\Omega\vert_{U_3} = dz^\#_{2n}\wedge dz^\#_{n} + \Sigma$ where $\Sigma$ is a closed $2$-form only depending on $z^\#_1,\ldots, z^\#_{n-1}, z^\#_{n+1},\ldots, z^\#_{2n-1}$. Bring $\Sigma$ into standard form using Darboux's theorem for HSMs (cf. Theorem \autoref{thm:holo_Darboux} and \autoref{app:darboux}).
 \end{enumerate}
 \textbf{Step 1:} We pick a small open neighborhood $U_1$ of $x_0$ such that $U_1$ is the domain of a holomorphic chart $\phi_1 = (\hat z_1,\ldots, \hat z_{2n}):U_1\to V_1\subset\mathbb{C}^{2n}$ with $\phi_1 (x_0) = 0$ and $d\mH$ does not vanish on $U_1$. $\pa_{\hat z_1},\ldots, \pa_{\hat z_{2n}}$ form a holomorphic frame of the tangent bundle $T^{(1,0)}U_1$. Since the holomorphic Hamiltonian vector field $X_\mH$ of the HHS $(X,\Omega,\mH)$ does not vanish on $U_1$, we can replace one coordinate vector field, say $\pa_{\hat z_n}$, in the collection $\pa_{\hat z_1},\ldots, \pa_{\hat z_{2n}}$ with $X_\mH$ and still obtain a frame of $T^{(1,0)}U_1$ after shrinking $U_1$ if necessary.\\
 Now let $\varphi^{X_\mH}_z:U_1\to X$ be the complex flow of the holomorphic Hamiltonian vector field $X_\mH$ for suitable $z\in\mathbb{C}$. $\varphi^{X_\mH}_z$ is constructed as follows: for $x\in X$, we denote by $\gamma_x$ the holomorphic curve from $\mathbb{C}$ to $X$ solving the following initial value problem:
 \begin{gather*}
  \gamma^\prime_x (z) = X_\mH (\gamma_x (z));\quad \gamma_x (0) = x.
 \end{gather*}
 By Proposition \autoref{prop:holo_traj}, the curves $\gamma_x$ exist and are locally unique. With this, we can set the flow of $X_\mH$ to be $\varphi^{X_\mH}_z (x)\coloneqq \gamma_x (z)$.\\
 Let us now consider the map $\phi^{-1}_2$ from $\mathbb{C}^{2n}$ to $X$ defined by
 \begin{gather*}
  \phi^{-1}_2(\tilde z_1,\ldots, \tilde z_{2n})\coloneqq \varphi^{X_\mH}_{\tilde z_n}\circ \phi^{-1}_1 (\tilde z_1,\ldots, \tilde z_{n-1}, 0, \tilde z_{n+1},\ldots, \tilde z_{2n}).
 \end{gather*}
 The differential of $\phi^{-1}_2$ at $0\in\mathbb{C}^{2n}$ maps the standard complex basis of $\mathbb{C}^{2n}$ to the vector fields $\pa_{\hat z_1},\ldots, \pa_{\hat z_{n-1}}, X_{\mH}, \pa_{\hat z_{n+1}}, \ldots, \pa_{\hat z_{2n}}$ at $x_0\in X$. These vector fields form a complex basis at $x_0$, hence, $\phi^{-1}_2$ is a local biholomorphism near $0$ by the holomorphic\linebreak version of the inverse function theorem. This gives us the holomorphic chart\linebreak $\phi_2 = (\tilde z_1,\ldots, \tilde z_{2n}):U_2\to V_2$, where $U_2$ is a small open neighborhood of $x_0$.\\
 Next, we set $G:U_2\to\mathbb{C}$ to be the coordinate $\tilde z_{n}$. We find:
 \begin{gather*}
  dG\vert_x (X_\mH\vert_x) = \left.\frac{d}{dz}\right\vert_{z = 0} (G\circ \varphi^{X_\mH}_z (x)) = 1\quad\forall x\in U_2.
 \end{gather*}
 This implies for the Poisson bracket of $\mH$ and $G$:
 \begin{gather*}
  \{\mH, G\} \coloneqq \Omega (X_\mH, X_G) = dG (X_\mH) = 1,
 \end{gather*}
 where $X_G$ is the holomorphic Hamiltonian vector field of the HHS $(U_2,\Omega\vert_{U_2}, G)$.\\
 \textbf{Step 2:} Taking the Hamiltonian vector field of a holomorphic function is a Lie algebra homomorphism, hence, we get:
 \begin{gather*}
  [X_\mH, X_G] = X_{\{\mH, G\}} = 0.
 \end{gather*}
 As in the real case, the commutativity of the vector fields implies the\linebreak commutativity of their flows. This allows us to find a holomorphic chart\linebreak $\phi_3 = (z^\#_1,\ldots, z^\#_{2n}):U_3\to V_3$ of $X$ near $x_0$ such that $G\vert_{U_3} = z^\#_n$, $\mH\vert_{U_3} = z^\#_{2n}$, $X_\mH\vert_{U_3} = \pa_{z^\#_n}$, and $X_G\vert_{U_3} = -\pa_{z^\#_{2n}}$. The construction of $\phi_3$ makes use of the regular value theorem for complex manifolds. Consider the holomorphic map $f\coloneqq(H\vert_{U_2},G):U_2\to\mathbb{C}^2$. The map $f$ is a submersion due to $\Omega (X_\mH, X_G) = 1$. By the regular value theorem, the level sets of $f$ are complex submanifolds of $U_2$. The tangent space of a level set is given by the kernel of $df$. Now pick the level set $W\coloneqq f^{-1}(\mH (x_0), 0) = f^{-1}(0,0)$ containing $x_0$ and a holomorphic chart $\psi$ of $W$ near $x_0$ with $\psi (x_0) = 0$. Next, we define the map $\phi^{-1}_3$ via:
 \begin{gather*}
  \phi^{-1}_3 (z^\#_1,\ldots, z^\#_{2n})\coloneqq \varphi^{X_\mH}_{z^\#_n}\circ \varphi^{X_G}_{-z^\#_{2n}}\circ \psi^{-1} (z^\#_1,\ldots, z^\#_{n-1}, z^\#_{n+1}, \ldots, z^\#_{2n-1}).
 \end{gather*}
 The differential of $\phi^{-1}_3$ at $0\in\mathbb{C}^{2n}$ maps the standard complex basis of $\mathbb{C}^{2n}$ to the vector fields $X_\mH$, $-X_G$, and the coordinate vector fields of $\psi$ at $x_0$. The coordinate vector fields $v_j$ of $\psi$ satisfy $d\mH (v_j) = dG (v_j) = 0$, hence, are orthogonal to $X_\mH$ and $X_G$ with respect to the symplectic form $\Omega$. Thus, by $\Omega (X_\mH, X_G) = 1$ and standard arguments from linear symplectic algebra, the coordinate vector fields of $\psi$ form together with $X_\mH$ and $-X_G$ a complex basis of $T^{(1,0)}_{x_0}X$. We can again apply the inverse function theorem to show that $\phi^{-1}_3$ is locally biholomorphic near $0$. From $\phi^{-1}_3$, we obtain the holomorphic chart $\phi_3 = (z^\#_1,\ldots, z^\#_{2n}):U_3\to V_3$ of $X$ near $x_0$. Using the commutativity of $\varphi^{X_\mH}_{z^\#_n}$ and $\varphi^{X_G}_{-z^\#_{2n}}$,\linebreak we easily compute $\pa_{z^\#_n} = X_\mH\vert_{U_3}$ and $\pa_{z^\#_{2n}} = -X_G\vert_{U_3}$. By construction of $\phi_3$ via level sets of $\mH$ and $G$, by exploiting the fact that Hamiltonian flows preserve $\Omega$, and by using $d\mH (X_\mH) = dG (X_G) = 0$, we deduce that the functions $\mH$ and $G$ only change in $X_G$- and $X_\mH$-direction, respectively. Therefore, $\mH$ only depends on $z^\#_{2n}$ and $G$ only depends on $z^\#_{n}$. Integrating $dG(X_\mH) = d\mH (-X_G) = \Omega (X_\mH, X_G) = 1$ gives us $G = z^\#_{n}$ and $\mH = z^\#_{2n}$.\\
 \textbf{Step 3:} Let us inspect $\Omega$ in the chart $\phi_3$ more closely. The coordinates of $\phi_3$ satisfy $\iota_{\pa_{z^\#_{n}}}\Omega\vert_{U_3} = -dz^\#_{2n}$ and $\iota_{\pa_{z^\#_{2n}}}\Omega\vert_{U_3} = dz^\#_{n}$. Therefore, $\Omega\vert_{U_3}$ takes the following form:
 \begin{gather*}
  \Omega\vert_{U_3} = dz^\#_{2n}\wedge dz^\#_n + \sum\limits_{i,j\neq n,2n} \Omega_{ij} dz^\#_{i}\wedge dz^\#_{j}\eqqcolon dz^\#_{2n}\wedge dz^\#_n + \Sigma.
 \end{gather*}
 $\Omega\vert_{U_3}$ and $dz^\#_{2n}\wedge dz^\#_n$ are closed $2$-forms on $U_3$, thus, $\Sigma$ is also a closed $2$-form on $U_3$. $d\Sigma = 0$ implies that the partial derivatives of $\Omega_{ij}$ with respect to $z^\#_{n}$ and $z^\#_{2n}$ have to vanish for every $i,j\neq n,2n$. Thus, the restriction of $\Sigma$ to a hyperplane $z^\#_n = c_1$ and $z^\#_{2n} = c_2$ does not depend on the values $c_1$ and $c_2$. Furthermore, the restriction of $\Sigma$ is a holomorphic symplectic $2$-form on such a hyperplane. This is a direct consequence of $\Omega\vert_{U_3}$ being a holomorphic symplectic $2$-form on $U_3$. Therefore, we can apply Theorem \autoref{thm:holo_Darboux} to $\Sigma$ giving us a holomorphic chart $\Psi$ of the hyperplane $z^\#_n = 0$ and $z^\#_{2n} = 0$ in which $\Sigma$ assumes the standard form. Replacing the coordinates $z^\#_1,\ldots, z^\#_{n-1}, z^\#_{n+1}, \ldots, z^\#_{2n-1}$ of $\phi_3$ with the coordinates of $\Psi$ yields the holomorphic chart $\phi = (z_1,\ldots, z_{2n}):U\to V$ of $X$ near $x_0$ (choose $\Psi (x_0) = 0$):
 \begin{gather*}
  \phi^{-1} (z_1,\ldots, z_{2n})\coloneqq \varphi^{X_\mH}_{z_n}\circ \varphi^{X_G}_{-z_{2n}}\circ \Psi^{-1} (z_1,\ldots, z_{n-1}, z_{n+1}, \ldots, z_{2n-1}).
 \end{gather*}
 In this chart, $\Omega$ and $\mH$ take the form $\Omega\vert_U = \sum^n_{j = 1} dz_{j+n}\wedge dz_j$ and $\mH\vert_U = z_{2n}$ concluding the proof.
\end{proof}

\begin{Cor}[Every regular HHS is locally integrable as a Hamiltonian system]\label{cor:integrable}
 Let $(X,\Omega,\mH)$ be a HHS and $x_0\in X$ be a point such that $d\mH\vert_{x_0}\neq 0$. Then, there exists an open neighborhood $U$ of $x_0$ and holomorphic functions $F_i,G_i:U\to\mathbb{C}$ for $i\in\{1,\ldots, n\}$, $2n\coloneqq \text{\normalfont dim}_\mathbb{C}(X)$, such that:
 \begin{gather*}
  F_n = \mH\vert_U,\quad \{F_i, G_j\} = \delta_{ij},\quad \{F_i, F_j\} = \{G_i,G_j\} = 0\quad\forall i,j\in\{1,\ldots, n\}.
 \end{gather*}
 In particular, every regular HHS is locally integrable as a Hamiltonian system.
\end{Cor}

\begin{proof}
 Take the chart $\phi$ from Lemma \autoref{lem:HHS_in_standard_form} and set $F_j\coloneqq z_{j + n}$ as well as $G_j\coloneqq z_j$.
\end{proof}

\begin{Rem}[Every regular RHS is locally integrable]\label{rem:integrable}
 Lemma \autoref{lem:HHS_in_standard_form} and Corollary \autoref{cor:integrable} are also true for regular RHSs with almost exactly the same proofs. In particular, every regular RHS is locally integrable.
\end{Rem}

With the ``integrable'' elements in standard form, let us now turn our attention to the second step. We need to construct local proper deformations of the standard form. Here, we only show how to deform HHSs explicitly. The deformations of HSMs and complex manifolds can be obtained in a similar way.

\begin{Prop}[Deformation of standard HHSs]\label{prop:deformation_of_standard_HHS}
 Let $(X,\Omega,\mH)$ be a HHS with complex structure $J$ and decompositions $\Omega = \Omega_R + i\Omega_I$ and $\mH = \mH_R + i\mH_I$, where $X = \mathbb{C}^{2n}\ (n>1)$, $J = i$, $\Omega = \sum^n_{j = 1} z_{j+n}\wedge z_j$, and $\mH \equiv c$ for some constant $c\in\mathbb{C}$ or $\mH = z_{2n}$. Furthermore, let $U\subset \mathbb{C}^{2n}$ be any non-empty open subset. Then, there exists a proper deformation $(X,J^\varepsilon;\Omega^\varepsilon_R,\mH^\varepsilon_R)$ of the HHS $(X,\Omega,\mH)$ such that $J^\varepsilon\vert_{X\backslash U} = J\vert_{X\backslash U}$, $\Omega^\varepsilon_R = \Omega_R$, and $\mH^\varepsilon_R = \mH_R$ for every $\varepsilon\in\mathbb{R}$.
\end{Prop}

\begin{proof}
 The idea for the deformation is based on Example \autoref{ex:constructing_PHHS}. Let $(X,\Omega,\mH)$ be a HHS as above with non-empty open subset $U\subset X$. Now pick a non-constant smooth function $f:X\cong\mathbb{R}^{4n}\to\mathbb{R}$ satisfying:
 \begin{gather*}
  f(x)\geq 0\ \forall x\in X,\quad f(x) = 0\ \forall x\in X\backslash U.
 \end{gather*}
 We define the $1$-parameter family of smooth functions $r^\varepsilon$ on $X$ as follows:
 \begin{gather*}
  r^\varepsilon (x)\coloneqq 1 + \varepsilon^2 f(x)\quad \forall x\in X\ \forall \varepsilon\in\mathbb{R}.
 \end{gather*}
 Using $r^\varepsilon$, we define the $1$-parameter family of almost complex structures $J^\varepsilon$ on $X$:
 \begin{align*}
  J^\varepsilon (\pa_{x_1}) \coloneqq r^\varepsilon\pa_{y_1},\ J^\varepsilon (\pa_{x_{n+1}}) \coloneqq \frac{1}{r^\varepsilon}\pa_{y_{n+1}}&;\quad J^\varepsilon (\pa_{y_1}) \coloneqq \frac{-1}{r^\varepsilon}\pa_{x_1},\ J^\varepsilon (\pa_{y_{n+1}}) \coloneqq -r^\varepsilon\pa_{x_{n+1}};\\
  J^\varepsilon (\pa_{x_j}) \coloneqq \pa_{y_j}&;\quad J^\varepsilon (\pa_{y_j}) \coloneqq -\pa_{x_j}\quad \forall j\in\{2,\ldots, n, n+2,\ldots, 2n\},
 \end{align*}
 where $\pa_{x_j}$ and $\pa_{y_j}$ are the vector fields coming from the coordinates\linebreak $(z_1 = x_1 + iy_1,\ldots, z_{2n} = x_{2n} + iy_{2n})\in\mathbb{C}^{2n}$. Clearly, $J^\varepsilon$ coincides with $J = i$ for $\varepsilon = 0$. Moreover, one easily checks that $J^\varepsilon$ satisfies $J^\varepsilon\vert_{X\backslash U} = J\vert_{X\backslash U}$ and is $\Omega_R$-anticompatible for every $\varepsilon\in\mathbb{R}$.\\
 Next, we set $\Omega^\varepsilon_R\equiv \Omega_R$ and $\Omega^\varepsilon_I\coloneqq -\Omega^\varepsilon_R (J^\varepsilon\cdot,\cdot)$ for every $\varepsilon\in\mathbb{R}$. As in Example \autoref{ex:constructing_PHHS}, the exterior derivative of $\Omega^\varepsilon_I$ can be expressed as
 \begin{gather*}
  d\Omega^\varepsilon_I = \varepsilon^2 \cdot df\wedge\left( dy_{n+1}\wedge dx_1 - \left(\frac{1}{r^\varepsilon}\right)^2dx_{n+1}\wedge dy_1\right).
 \end{gather*}
 We see that for every $\varepsilon\neq 0$ there exists a point $x_0\in U$ satisfying $d\Omega^\varepsilon_I\vert_{x_0}\neq 0$. Thus, $(X,J^\varepsilon;\Omega^\varepsilon_R)$ becomes a proper PHSM for every $\varepsilon\neq 0$ due to Theorem \autoref{thm:rel_HSM_PHSM}.\\
 The only thing left to check is that $\mH^\varepsilon_R\equiv \mH_R$ is compatible with the PHSM $(X,J^\varepsilon;\Omega^\varepsilon_R)$, i.e., $d[\Omega^\varepsilon_R (J^\varepsilon (X^{\Omega^\varepsilon_R}_{\mH^\varepsilon_R}),\cdot)] = 0$ for every $\varepsilon$. In the case of $\mH\equiv c$, this is trivially true, because then the Hamiltonian vector field $X^{\Omega^\varepsilon_R}_{\mH^\varepsilon_R}$ vanishes. For $\mH = z_{2n}$, we first realize that the equation above is equivalent to the condition that $\mH^\varepsilon_R$ is the real part of some pseudo-holomorphic function $\mH^\varepsilon$. We already know that $\mH^\varepsilon\equiv \mH$ is pseudo-holomorphic with respect to $J = i$. Since $\mH^\varepsilon$ only depends on the last component $z_{2n}$ and both $J^\varepsilon$ and $J = i$ act in the same way on the last components of $X$ for every $\varepsilon$ (as $n>1$), $\mH^\varepsilon\equiv\mH$ is also pseudo-holomorphic with respect to $J^\varepsilon$ for every $\varepsilon$. This turns $(X,J^\varepsilon;\Omega^\varepsilon_R,\mH^\varepsilon_R)$ into the desired deformation of $(X,\Omega,\mH)$ concluding the proof.
\end{proof}

\begin{Rem}[The case $\mH\equiv c$]\label{rem:constant_Ham}
 Upon closer inspection, we note that the proof also works for $n = 1$ if $\mH\equiv c$. Indeed, we only need the condition $n>1$ to ensure that $\mH^\varepsilon_R$ is the real part of some pseudo-holomorphic function $\mH^\varepsilon$. However, the constant function is pseudo-holomorphic with respect to any almost complex structure.
\end{Rem}

In order to find local proper deformations of the standard HSM, we simply use the deformation from the proof above and forget the Hamiltonian $\mH^\varepsilon_R$. Since we are not constrained by the compatibility with the Hamiltonian $\mH^\varepsilon_R$ in this case, this deformation even works for $n=1$. Similarly, we obtain a local proper deformation of the complex manifold $\mathbb{C}^m$ if we take the almost complex structure $J^\varepsilon$ from the proof. Here, we need to enforce $m>1$ to make sure that we even have four real dimensions to rescale.\\
Combining Lemma \autoref{lem:HHS_in_standard_form} with Proposition \autoref{prop:deformation_of_standard_HHS} proves the following theorem:
\newpage

\begin{Thm}[Proper PHHSs are generic]\label{thm:generic}
 Let $X$ be a smooth manifold, then the following statements apply depending on the real dimension of $X$:
 \begin{enumerate}
  \item If $\text{\normalfont dim}_\mathbb{R}(X) = 2$: Every almost complex structure on $X$ is integrable and automatically a complex structure\footnote{It is easy to check that the Nijenhuis tensor always vanishes in two dimensions.}.
  \item If $\text{\normalfont dim}_\mathbb{R}(X) > 2$: Every complex manifold $(X,J)$ and HSM $(X,\Omega)$ admits a proper deformation. In particular, the non-integrable almost complex structures and proper PHSMs on $X$ are generic within the set of all almost complex structures and PHSMs on $X$, respectively.
  \item If $\text{\normalfont dim}_\mathbb{R}(X) > 4$: Every HHS $(X,\Omega, \mH)$ admits a proper deformation. In particular, the proper PHHSs on $X$ are generic within the set of all PHHSs on $X$.
 \end{enumerate}
\end{Thm}

Before we conclude this section, let us quickly comment on PHHSs $(X,J;\Omega_R,\mH_R)$ in dimension $\text{dim}_\mathbb{R}(X) = 4$. These systems are the only geometrical object studied in this subsection to which Theorem \autoref{thm:generic} does \underline{not} apply. The reason is that Proposition \autoref{prop:deformation_of_standard_HHS} fails for $\text{dim}_\mathbb{R}(X) = 4$, since $\mH_R = x_2$ cannot be the real part of a complex function which is pseudo-holomorphic with respect to deformations $J^\varepsilon$ as chosen in the proof of Proposition \autoref{prop:deformation_of_standard_HHS}. Nevertheless, Statement 3 of Theorem \autoref{thm:generic} might still be true for $\text{dim}_\mathbb{R}(X) = 4$. One possible way to prove this could be to modify Proposition \autoref{prop:deformation_of_standard_HHS}. Instead of choosing $\mH^\varepsilon_R$ to be independent of $\varepsilon$, we could allow for general deformations $\mH^\varepsilon_R$ of $\mH_R$. Finding such deformations $\mH^\varepsilon_R$ involves solving a second order PDE with boundary conditions. However, the existence of non-trivial solutions to the given problem might be forbidden by the Nijenhuis tensor. We elaborate on this thought: let $X$ be a smooth manifold with almost complex structure $J$, $N_J$ be the Nijenhuis tensor of $J$, and $f:X\to\mathbb{C}$ be a pseudo-holomorphic map, i.e., $df\circ J = i\cdot df$. A straightforward calculation reveals that $df\left(N_J (V,W)\right) = 0$ for all vector fields $V$ and $W$ on $X$. Thus, the image of $N_J\vert_x$ is contained within the kernel of $df\vert_x$ for any pseudo-holomorphic function $f$ and any point $x\in X$. This implies that there are at most $1/2(\text{dim}_\mathbb{R}(X)-r_x)$ pseudo-holomorphic functions on $X$ whose differentials at a given point $x\in X$ are $\mathbb{C}$-linearly independent, where $r_x$ is the rank\footnote{The Nijenhuis tensor satisfies the relation $N_J (J(V),W) = -JN_J (V,W)$, hence, its rank is even.} of $N_J\vert_x$. In four dimensions, the rank of $N_J$ alone does not exclude the existence of non-trivial pseudo-holomorphic functions: let $V_1$, $V_2$, $J(V_1)$, and $J(V_2)$ be a local frame of $X$. One easily sees that, because of the symmetries of $N_J$, i.e., $N_J (V,W) = -N_J(W,V)$ and $N_J(J(V), W) = -JN_J (V,W)$, $N_J (V_1, V_2)$ and $N_J (V_1, J(V_2)) = -JN_J (V_1,V_2)$ are the only two components of $N_J$ in the local frame $V_1$, $V_2$, $J(V_1)$, and $J(V_2)$ which are not redundant. Thus, the rank of the Nijenhuis tensor is at most $2$ in four dimensions. For instance, the deformation $J^\varepsilon$ from above for $n=1$,
 \begin{gather*}
  J^\varepsilon (\pa_{x_1}) \coloneqq r^\varepsilon\pa_{y_1},\ J^\varepsilon (\pa_{x_{2}}) \coloneqq \frac{1}{r^\varepsilon}\pa_{y_{2}};\quad J^\varepsilon (\pa_{y_1}) \coloneqq \frac{-1}{r^\varepsilon}\pa_{x_1},\ J^\varepsilon (\pa_{y_{2}}) \coloneqq -r^\varepsilon\pa_{x_{2}},
 \end{gather*}
 yields for $r^\varepsilon\in C^\infty(\mathbb{R}^4,\mathbb{R}_{+})$:
 \begin{gather*}
  N_{J^\varepsilon} (\pa_{x_1},\pa_{x_2}) = \frac{1}{r^\varepsilon}J^\varepsilon\left(N_{J^\varepsilon} (\pa_{x_1},\pa_{y_2})\right) = \sum\limits_{a\in\{x,y\}}\sum^2_{i,j = 1;\ i\neq j}\left(\pa_{a_i}\ln (r^\varepsilon)\right)\cdot \pa_{a_j}.
 \end{gather*}
 Hence, the (real) rank of $N_{J^\varepsilon}\vert_x$ is $2$ for $dr^\varepsilon\vert_x \neq 0$ and $0$ for $dr^\varepsilon\vert_x = 0$.\\
 Nevertheless, the rank of $N_J$ is a rather weak bound for the number of independent pseudo-holomorphic functions. The exact number is given by $1/2(\text{dim}_\mathbb{R}(X) - k)$, where $k$ is the rank of the IJ-bundle\footnote{Confer \cite{muskarov1986} for the definition of the IJ-bundle and a detailed investigation of the relation between the Nijenhuis tensor and pseudo-holomorphic functions.} on $X$ containing the image of $N_J$. Even though there are non-integrable almost complex structures $J$ in four dimensions whose IJ-bundle does not have full rank, e.g. $J^\varepsilon$ with $r^\varepsilon = e^{-x_1}$ as in Example \autoref{ex:constructing_PHHS}, it is not clear at all why this should also apply to the almost complex structures $J^\varepsilon$ as above for general functions $r^\varepsilon$.

\newpage
\section*{Appendix}
\addcontentsline{toc}{section}{Appendix}
\markboth{}{Appendix}

\begin{appendix}
 \section{Proof of Darboux's Theorem for HSMs}
\label{app:darboux}

In this part of the appendix, we want to show that there is a holomorphic counterpart to Darboux's theorem.

\begin{Thm*}[Darboux's theorem for HSMs]
 Let $(X,\Omega)$ be a holomorphic symplectic manifold (HSM) of complex dimension $\text{\normalfont dim}_\mathbb{C}(X) = 2n$ $(n\in\mathbb{N})$. Then, for every point $x\in X$, there is a holomorphic chart $\psi = (Q_1,\ldots, Q_n, P_1,\ldots, P_n):U\to V\subset\mathbb{C}^{2n}$ of $X$ near $x$ such that
 \begin{gather*}
  \Omega\vert_{U} = \sum\limits^n_{j = 1} dP_j\wedge dQ_j.
 \end{gather*}
\end{Thm*}

\begin{proof}
 The proof of Darboux's theorem for HSMs works almost completely analogously to the proof of Darboux's theorem in the real setup given by Weinstein. In particular, we will apply Moser's trick to HSMs. A detailed transcription of Moser's trick to complex manifolds can be found in \cite{soldatenkov2021} by Soldatenkov and Verbitsky.\\
 Let $(X,\Omega)$ be a HSM of complex dimension $\text{\normalfont dim}_\mathbb{C}(X) = 2n$ ($n\in\mathbb{N}$), $x\in X$ be any point of $X$, and $(\hat U, \hat\psi = (\hat z_1,\ldots, \hat z_{2n}))$ be a holomorphic chart of $X$ near $x$. First, we observe that every complex symplectic form $\omega$ on a complex vector space $V$ of dimension $\text{\normalfont dim}_\mathbb{C}(V) = 2n$ can be brought into standard form, i.e., can be written as
 \begin{gather*}
  \omega = \sum\limits^n_{j=1} \theta_{j+n}\wedge\theta_{j}
 \end{gather*}
 for a basis $(\theta_1,\ldots,\theta_{2n})$ dual to some complex basis $(e_1,\ldots, e_{2n})$ of $V$. Thus, we can assume that $\Omega$ at $x$ in the chart $(\hat U,\hat\psi)$ takes the form
 \begin{gather*}
  \Omega\vert_x = \sum\limits^n_{j = 1} d\hat z_{j+n}\vert_x\wedge d\hat z_{j}\vert_x
 \end{gather*}
 by applying a $\mathbb{C}$-linear transformation to $(\hat U,\hat\psi)$ if necessary. Next, we define the following $2$-form on $\hat U$:
 \begin{gather*}
  \Omega_1\coloneqq\sum\limits^n_{j=1} d\hat z_{j+n}\wedge d\hat z_{j} \in\Omega^{(2,0)}(\hat U).
 \end{gather*}
 Clearly, both $\Omega_0\coloneqq \Omega\vert_{\hat U}$ and $\Omega_1$ are holomorphic symplectic $2$-forms on $\hat U$. Now, we define the $2$-form $\Omega_t$ on $\hat U$ as the interpolation of $\Omega_0$ and $\Omega_1$:
 \begin{gather*}
  \Omega_t\coloneqq \Omega_0 + t(\Omega_1 - \Omega_0)\quad\forall t\in\mathbb{R}.
 \end{gather*}
 For every $t\in\mathbb{R}$, the form $\Omega_t$ is holomorphic and closed, as $\Omega_0$ and $\Omega_1$ are holomorphic and closed. Further, we find that $\Omega_t\vert_x = \Omega_0\vert_x$ for every $t\in\mathbb{R}$, as $\Omega_0$ and $\Omega_1$ coincide at $x$ by construction. This means that $\Omega_t\vert_x$ is non-degenerate for every $t$. As non-degeneracy is an open property, we can construct an open subset $U^\prime\subset\hat U$ containing $x$ such that $\Omega_t\vert_{U^\prime}$ is a non-degenerate $2$-form for every $t\in [0,1]$, where we have also used the fact that $\Omega_t$ depends continuously on $t$ and $[0,1]$ is a compact interval. This turns ($U^\prime$, $\Omega_t\vert_{U^\prime}$) into a HSM for every $t\in [0,1]$. Moreover, we can assume that $U^\prime$ is contractible by shrinking $U^\prime$ if necessary. For the sake of simplicity and ease of notation, we assume from now on that $U^\prime = \hat U$. As $\hat U$ is contractible, its cohomology is trivial. This allows us to apply Theorem 2.5 (Moser's isotopy, version I) from \cite{soldatenkov2021}. Hence, we can find a $1$-form $\alpha$\footnote{In the paper by Soldatenkov and Verbitsky, the $1$-form $\alpha \equiv \alpha_t$ has, in general, a non-trivial $t$-dependence. In the present case, however, we can omit the $t$-dependence, as $d/dt\,\Omega_t$ does not depend on $t$ anymore.} on $\hat U$ of type $(1,0)$ such that
 \begin{gather*}
  \Omega_1 - \Omega_0 = \frac{d}{dt}\Omega_t = d\alpha.
 \end{gather*}
 Since the $2$-form $\Omega_1 - \Omega_0$ is of type $(2,0)$, the $1$-form $\alpha$ satisfies $\bar{\partial}\alpha = 0$. Hence, the form $\alpha$ is even holomorphic. We can assume without loss of generality that $\alpha$ satisfies $\alpha\vert_x = 0$ by replacing $\alpha$ with
 \begin{gather*}
  \alpha^\prime = \alpha - \alpha^x
 \end{gather*}
 if necessary, where $\alpha^x$ is a holomorphic $1$-form with $\alpha^x\vert_x = \alpha\vert_x$ and $d\alpha^x = 0$. For every $t\in [0,1]$, define the vector field $V_t$ on $\hat U$ with values in $T^{(1,0)}\hat U$ by
 \begin{gather*}
  \iota_{V_t}\Omega_t = -\alpha.
 \end{gather*}
 Note that $V_t$ is well-defined, as $\alpha$ is a $1$-form of type $(1,0)$ and $\Omega_t$ is non-degenerate on $T^{(1,0)}\hat U$, thus, $\Omega_t$ gives rise to an isomorphism from $T^{(1,0)}\hat U$ to $T^{\ast, (1,0)}\hat U$. Since $\alpha$ and $\Omega_t$ are holomorphic, $V_t$ is a holomorphic vector field on $\hat U$ for every $t\in [0,1]$. Using the closedness of $\Omega_t$ and Cartan's magic formula, we calculate the Lie derivative $L_{V_t}\Omega_t$:
 \begin{gather*}
  L_{V_t}\Omega_t = d\iota_{V_t}\Omega_t + \iota_{V_t}d\Omega_t = d\iota_{V_t}\Omega_t = -d\alpha = -\frac{d}{dt}\Omega_t\quad\forall t\in [0,1].
 \end{gather*}
 Recall\footnote{Confer Proposition \autoref{prop:holo_vec_field_equiv_J_pre_vec_field} and \cite{kobayashi1969} for details.} that every holomorphic vector field $V$ can be uniquely written as\linebreak $1/2(V^R - i\cdot J(V^R))$, where $J$ is the complex structure of the underlying complex manifold and $V^R$ is a real $J$-preserving vector field. Now let $V^R_t$ be the real $J$-preserving vector field corresponding to $V_t$ for every $t\in [0,1]$. Next, we want to show the following equation:
 \begin{gather*}
  L_{V^R_t}\Omega_t = L_{V_t}\Omega_t = -\frac{d}{dt}\Omega_t.
 \end{gather*}
 We do this by proving a more general auxiliary lemma:
 
 \begin{Aux*}
  Let $X$ be a complex manifold with complex structure $J\in\Gamma (\text{\normalfont End}(TX))$. Further, let $V^R\in\Gamma_J (TX)$ be a $J$-preserving vector field on $X$ with corresponding holomorphic vector field $V\coloneqq 1/2(V^R -iJ(V^R))$ and $T$ be a holomorphic tensor field on $X$, then the Lie derivatives of $T$ with respect to $V^R$ and $V$ coincide, i.e.:
  \begin{gather*}
   L_{V^R} T = L_{V}T.
  \end{gather*}
 \end{Aux*}

 \begin{proof}
  Let $X$ and $J$ be as above and let $V^R\in\Gamma_J (TX)$ be a $J$-preserving vector field on $X$ with corresponding holomorphic vector field $V\coloneqq 1/2(V^R -iJ(V^R))$. Further, let $T$ be a holomorphic $(k,l)$-tensor field on $X$. The Lie derivative is complex linear, thus, we have by definition of $V$:
  \begin{gather*}
   L_V T = \frac{1}{2}\left(L_{V^R}T - i\cdot L_{J(V^R)}T\right).
  \end{gather*}
  Hence, it suffices to show:
  \begin{gather}\label{eq:lie_i}
   L_{J(V^R)}T = i\cdot L_{V^R}T.
  \end{gather}
  In a holomorphic chart $\phi = (z_1,\ldots, z_n):U\to V\subset\mathbb{C}^n$ of $X$, $T$ can be expressed as
  \begin{gather*}
   T\vert_U = \sum^n_{i_1\ldots i_k, j_1\ldots j_l = 1} T^{i_1\ldots i_k}_{j_1\ldots j_l}\cdot dz_{j_1}\otimes\ldots\otimes dz_{j_l}\otimes \partial_{z_{i_1}}\otimes\ldots\otimes \partial_{z_{i_k}},
  \end{gather*}
  where $T^{i_1\ldots i_k}_{j_1\ldots j_l}:U\to\mathbb{C}$ are holomorphic functions on $U$. Since the Lie derivative can be computed locally and satisfies the Leibniz rule, it suffices to show Equation \eqref{eq:lie_i} for $T$ being $T^{i_1\ldots i_k}_{j_1\ldots j_l}$, $dz_{i}$, and $\partial_{z_j}$. As the Lie derivative also commutes with the exterior differential $d$ and $\partial_{z_j}$ is a (local) holomorphic vector field, it is sufficient to prove Equation \eqref{eq:lie_i} for $T$ being a holomorphic function $h$ and holomorphic vector field $W$. For $T = h$, we find:
  \begin{gather*}
   L_{J(V^R)} h = dh\left( J (V^R)\right) = i\cdot dh (V^R) = i\cdot L_{V^R}h,
  \end{gather*}
  where we used that $h$ is holomorphic, i.e. $dh\circ J = i\cdot dh$. For $T = W$, we can use Proposition \autoref{prop:holo_vec_field_equiv_J_pre_vec_field} to obtain:
  \begin{align*}
   L_{J(V^R)}W &= \left[J(V^R), W\right] = \frac{1}{2}\left([J(V^R), W^R] - i[J(V^R), J(W^R)]\right)\\
   &= \frac{1}{2}\left([V^R, J(W^R)] - i[V^R, J^2(W^R)]\right) = \frac{i}{2}\left([V^R, W^R] - i[V^R, J(W^R)]\right)\\
   &= i\left[V^R, W\right] = i\cdot L_{V^R}W,
  \end{align*}
  concluding the proof.
 \end{proof}
 
 Now return to the proof of Darboux's theorem for HSMs. Let $\varphi_t$ be the (possibly local) flow of the real time-dependent vector field $V^R_t$. The pull-back $\varphi^\ast_t\Omega_t$ is a solution of the initial value problem:
 \begin{gather*}
  \frac{d}{dt}(\varphi^\ast_t \Omega_t) = \varphi^\ast_t (L_{V^R_t}\Omega_t + \frac{d}{dt}\Omega_t) = 0,\quad \varphi^\ast_0\Omega_0 = \Omega_0,
 \end{gather*}
 where we suppressed the fact that $\varphi_t$ might not be defined on all of $\hat U$ for every $t\in [0,1]$ in our notation. Clearly, $\Omega_0$ is also a solution to the same initial value problem. As the solution to the given initial value problem is unique, we obtain:
 \begin{gather*}
  \varphi^\ast_t\Omega_t = \Omega_0.
 \end{gather*}
 We have to show that the last equation holds true for every $t\in [0,1]$ in some open neighborhood $U\subset\hat U$ of $x$. For this, we recall that $\alpha\vert_x = 0$ by construction. Thus, we have $V^R_t (x) = V_t (x) = 0$ for every $t\in [0,1]$. This implies that the flow $\varphi_t$ is stationary at $x$, i.e., $\varphi_t (x) = x$ for every $t\in [0,1]$. We can deduce from this that there exists an open neighborhood $U\subset\hat U$ of $x$ such that the flow $\varphi_t: U\to \varphi_t(U)\subset\hat U$ is a well-defined diffeomorphism for every $t\in [0,1]$. In particular, the time-$1$-map $\varphi_1:U\to\varphi_1 (U)$ satisfies:
 \begin{gather*}
  \varphi^\ast_1\Omega_1 = \Omega_0.
 \end{gather*}
 Hence, $(U,\psi\coloneqq\hat \psi\circ\varphi_1)$ is a smooth chart of $X$ near $x$ which satisfies:
 \begin{gather*}
  \psi^{-1\,\ast}\Omega\vert_U = {\hat\psi}^{-1\,\ast}\left(\varphi^{-1\,\ast}_1 \Omega_0\right) = {\hat\psi}^{-1\,\ast}\Omega_1 = \sum^n_{j=1} \theta_{j+n}\wedge \theta_{j},
 \end{gather*}
 where $\sum \theta_{j+n}\wedge\theta_{j}$ is the standard symplectic form on $\mathbb{C}^{2n}$. Hence, the holomorphic symplectic form $\Omega$ takes the following form on $(U,\psi \equiv (z_1,\ldots, z_{2n}))$:
 \begin{gather*}
  \Omega\vert_{U} = \sum\limits^n_{j = 1} dz_{j+n}\wedge dz_{j}.
 \end{gather*}
 We see that $(U,\psi)$ is a good candidate for the desired Darboux chart. To conclude the proof, we need to show that $(U,\psi)$ is also a holomorphic chart of $X$. For this, it suffices to prove that the map $\varphi_1: U\to \varphi_1 (U)$ is locally biholomorphic. The idea behind this proof is simple:
%  We deduce from Lemma (holo. int. curves\todo{Insert proper name here!}) that the flow $\varphi^{V^R}_t$ of a \underline{time-independent} real $J$-preserving vector field $V^R$ is locally biholomorphic.
 In Chapter IX of \cite{kobayashi1969}, it is shown that the \underline{time-independent} $J$-preserving\footnote{$J$-preserving vector fields are called infinitesimal automorphisms in \cite{kobayashi1969}.} vector fields $V^R$ on a complex manifold $X$ are exactly those real vector fields whose flow $\varphi^{V^R}_t$ is locally biholomorphic. However, we cannot directly apply this statement to $V^R_t$, as, in general, $V^R_t$ carries a non-trivial time-dependence. To account for this, we relate $V_t$ to a time-independent holomorphic vector field $V$ on $\hat U \times O \ni (x,t)$, where $O\subset\mathbb{C}$.\\
 First, we generalize the definition of $\Omega_t$ and allow for complex times $t\equiv\tau\in\mathbb{C}$. By the same arguments as before and by shrinking $\hat U$ if necessary, we find an open neighborhood $O\subset\mathbb{C}$ of $[0,1]$ such that $\Omega_\tau$ is non-degenerate for every $\tau\in O$. This allows us to generalize the definition of $V_\tau$ to all $\tau\in O$. Now observe that we cannot only view $\Omega_\tau$ as a time-dependent $(2,0)$-form on $\hat U$, but also as a time-independent $(2,0)$-form on $\hat U\times O$. As the time-dependence of $\Omega_\tau$ is clearly holomorphic, $\Omega_\tau$ as a form on $\hat U\times O$ is also holomorphic. Thus, $V_\tau$ understood as a vector field on $\hat U\times O$ is also holomorphic. Now consider the time-independent vector field $V$ on $\hat U\times O$:
 \begin{gather*}
  V(y,\tau^\prime)\coloneqq V_{\tau^\prime}(y) + \partial_\tau\vert_{(y,\tau^\prime)}\quad\forall (y,\tau^\prime)\in\hat U\times O.
 \end{gather*}
 As $V_\tau$ and $\partial_\tau$ are holomorphic, $V$ is also holomorphic. Now let $V^R$ be the real $J$-preserving vector field corresponding to $V$ and let $\Gamma:[0,1]\to\hat U\times O$, $\Gamma (r)\equiv (\gamma (r), \rho (r))$ be a smooth curve in $\hat U\times O$ with $\rho (0) = 0$. Then, we have the following auxiliary lemma:
 
 \begin{Aux*}
  $\Gamma$ is an integral curve of $V^R$ if and only if $\gamma:[0,1]\to\hat U$ is an integral curve of $V^R_t$ and $\rho (r) = r$ for every $r\in [0,1]$.
 \end{Aux*}
 
 \begin{proof}[Proof of auxiliary lemma]
  This follows from a quick computation: let $\tau = t + is$ be the decomposition of $\tau$ into real and imaginary part, then the vector field $V^R$ is given by:
  \begin{gather*}
   V^R(y,\tau^\prime) = V^R_{\tau^\prime}(y) + \partial_t\vert_{(y,\tau^\prime)}\quad\forall (y,\tau^\prime)\in\hat U\times O.
  \end{gather*}
  Further, let $\rho = \rho_R + i\rho_I$ be the decomposition of $\rho$ into real and imaginary part, then the integral curve equation of $V^R$ for $\Gamma$ can be written as:
  \begin{gather*}
   \dot\gamma (r) = V^R_{\rho (r)} (\gamma (r));\quad \dot \rho_R (r)\cdot\partial_t\vert_{\Gamma (r)} + \dot \rho_I (r)\cdot\partial_s\vert_{\Gamma (r)} = \partial_t\vert_{\Gamma (r)}\quad r\in [0,1].
  \end{gather*}
  As $\rho$ needs to satisfy the initial condition $\rho (0) = 0$, $\rho$ is given by $\text{id}_{[0,1]}$ if $\Gamma$ is an integral curve of $V^R$. In this case, $\gamma$ has to satisfy the following differential equation:
  \begin{gather*}
   \dot\gamma (r) = V^R_{r} (\gamma (r))\quad r\in [0,1].
  \end{gather*}
  Thus, if $\Gamma$ is an integral curve of $V^R$, then $\gamma$ is an integral curve of the time-dependent vector field $V^R_t$ on $\hat U$. The converse direction follows similarly.
 \end{proof}
 
 From the auxiliary lemma, it follows that the flow $\varphi^{V^R}_t$ of $V^R$ and the flow $\varphi_t$ of the time-dependent vector field $V^R_t$ are related in the following way:
 \begin{gather*}
  \varphi^{V^R}_t (y,0) = \left(\varphi_t (y), t\right)
 \end{gather*}
 for every suitable $y\in\hat U$. As discussed earlier, $\varphi^{V^R}_t$ is the flow of a holomorphic vector field and, hence, locally biholomorphic. This implies that $\varphi_t$ is also locally biholomorphic concluding the proof.
\end{proof}

 \newpage
 \section{Proof of Morse Darboux Lemmata}
\label{app:morse_darboux_lem}

Our goal in this part is to prove Lemma \autoref{lem:morse_darboux_lem_I} and \autoref{lem:morse_darboux_lem_II}. We start with Lemma \autoref{lem:morse_darboux_lem_I}.

\begin{Lem*}[Morse Darboux lemma I]
 Let $(M^2,\omega)$ be a symplectic $2$-manifold, $L^1$ be a smooth $1$-manifold, $f\in C^\infty (M,L)$, and $p\in M$ be a non-degenerate critical point of $f$ with Morse index $\mu_f (p) = 0$. Further, let $T>0$ be a positive real number. Then, there exists a $C^1$-chart $\psi_L:U_L\to V_L\subset\R$ of $L$ near $f(p)$ which is smooth on $U_L\backslash\{f(p)\}$ such that all non-constant trajectories near $p$ of the RHS $(U_M, \omega\vert_{U_M}, H)$ with $U_M\coloneqq f^{-1}(U_L)$ and $H\coloneqq \psi_L\circ f\vert_{U_M}$ are $T$-periodic.
\end{Lem*}

\begin{proof}
 The proof consists of three steps:
 \begin{enumerate}
  \item First, we convince ourselves that the non-constant trajectories $\gamma$ near $p$ are indeed periodic.
  \item Afterwards, we compute the period $\hat T(r)$ of a trajectory $\gamma$ near $p$ with $f\circ \gamma = r^2$ to show that $\hat T(r)$ is defined for $r\in (-\varepsilon,\varepsilon)$ ($\varepsilon>0$), depends smoothly on $r$, and is bounded from below by a positive constant.
  \item Lastly, we use these properties of $\hat T(r)$ to define a $C^1$-diffeomorphism\linebreak $\psi_L:U_L\to V_L\subset\R$ such that the trajectories of the rescaled RHS $(U_M,\omega\vert_{U_M}, H)$ with $U_M$ and $H$ as above have fixed period $T>0$.
 \end{enumerate}
 \textbf{Step 1}\\\\
 Without loss of generality, we can assume, after choosing appropriate charts, that $L = \R$, $f(p) = 0\in\R$, and that the (usual) Morse index $\mu_f (p)$ of $f:M\to\R$ is $0$. Now, we apply the Morse lemma to find a chart $\hat\psi_M = (\hat x, \hat y):\hat U_M\to \hat V_M$ of $M$ near $p$ with $\hat \psi_M (p) = 0$ such that $f\vert_{\hat U_M} = {\hat x}^2 + {\hat y}^2$. In this chart, we have $\omega\vert_{\hat U_M} = \hat v\, d\hat x\wedge d\hat y$, where $\hat v\in C^\infty (\hat U_M,\R)$. Since $\omega$ is non-degenerate, we can assume $\hat v >0$ (after permuting $\hat x$ and $\hat y$ if necessary). Now consider the RHS $(M,\omega, f)$ and its trajectories $\gamma$ near $p$. $f$ is constant along $\gamma$, so for $r>0$ small enough the trajectories $\gamma$ near $p$ with $f\circ \gamma = r^2$ move along the circle
 \begin{gather*}
  f^{-1}(r^2) = \hat \psi^{-1}_M (\{(\hat x,\hat y)\in\R^2\mid \hat x^2 + \hat y^2 = r^2\}) \cong S^1
 \end{gather*}
 with velocity $\dot\gamma\neq 0$. Hence, the trajectories near $p$ are periodic.\\\\
 \textbf{Step 2}\\\\
 Denote the period of $\gamma$ with $f\circ \gamma = r^2$ and $r>0$ by $\hat T(r)$. We calculate $\hat T(r)$ by going into polar coordinates:
 \begin{gather*}
  (\hat x,\hat y) = (r\cos (\varphi), r\sin (\varphi)).
 \end{gather*}
 Define $v \coloneqq \hat v\circ \hat\psi^{-1}_M$, then the Hamiltonian vector field $X_f$ is given by:
 \begin{gather*}
  d\hat\psi_M (X_f) = \frac{2}{v(\hat x, \hat y)}
  \begin{pmatrix}
   -\hat y\\ \hat x
  \end{pmatrix} = \frac{2}{v(r\cos (\varphi), r\sin (\varphi))}
  \begin{pmatrix}
   -r\sin (\varphi)\\ r\cos (\varphi)
  \end{pmatrix}.
 \end{gather*}
 Now parameterize an integral curve $\gamma:\R\to M$ of $X_f$ by $r_\gamma, \varphi_\gamma:\R\to\R$ in the following way:
 \begin{gather*}
  \hat \psi_M\circ \gamma (t) =
  \begin{pmatrix}
   r_\gamma (t)\cos (\varphi_\gamma (t))\\
   r_\gamma (t)\sin (\varphi_\gamma (t))
  \end{pmatrix}.
 \end{gather*}
 The integral curve equation $\dot\gamma = X_f (\gamma)$ now yields:
 \begin{gather*}
  \dot r_\gamma = 0;\quad \dot \varphi_\gamma = \frac{2}{v(r_\gamma\cos (\varphi_\gamma), r_\gamma\sin (\varphi_\gamma))}.
 \end{gather*}
 Thus, $r_\gamma$ is constant and, since $f\circ\gamma = r^2$, given by $r$. This allows us to define $\Phi:\R\to \R$ by
 \begin{align*}
  \Phi (t)&\coloneqq \varphi_\gamma (t) - \varphi_\gamma (0) = \int\limits^t_0 \dot\varphi_\gamma (t^\prime) dt^\prime\\
  &= \int\limits^t_0 \frac{2}{v(r\cos (\varphi_\gamma (t^\prime)), r\sin (\varphi_\gamma(t^\prime)))} dt^\prime.
 \end{align*}
 $\Phi$ is an orientation preserving diffeomorphism, since $\dot\Phi = \frac{2}{v}>0$. Hence, $\Phi^{-1}$ exists and is given by:
 \begin{gather*}
  \Phi^{-1} (\alpha) = \frac{1}{2}\int\limits^{\varphi_0+\alpha}_{\varphi_0} v(r\cos (\varphi), r\sin (\varphi))\, d\varphi
 \end{gather*}
 with $\varphi_0\coloneqq \varphi_\gamma (0)$, as
 \begin{align*}
  \frac{d\Phi^{-1}}{d\alpha} (\alpha) &\stackrel{\phantom{\varphi_\gamma = \varphi_0 + \Phi}}{=} \frac{1}{\dot\Phi (\Phi^{-1}(\alpha))} = \frac{1}{2}v (r\cos (\varphi_\gamma (\Phi^{-1}(\alpha))), r\sin (\varphi_\gamma (\Phi^{-1}(\alpha))))\\
  &\stackrel{\varphi_\gamma = \varphi_0 + \Phi}{=} \frac{1}{2} v(r\cos (\varphi_0 + \alpha), r\sin (\varphi_0 + \alpha)).
 \end{align*}
 We can now use $\Phi$ and $\Phi^{-1}$ to compute $\hat T(r)$:
 \begin{align*}
  &\varphi_\gamma (t + \hat T(r)) = \varphi_\gamma (t) + 2\pi\quad\forall t\in\R\quad \Rightarrow \Phi (\hat T(r)) = 2\pi\\
  \Rightarrow &\hat T(r) = \Phi^{-1} (2\pi) = \frac{1}{2}\int\limits^{\varphi_0+2\pi}_{\varphi_0} v(r\cos (\varphi), r\sin (\varphi))\, d\varphi = \frac{1}{2}\int\limits^{2\pi}_{0} v(r\cos (\varphi), r\sin (\varphi))\, d\varphi
 \end{align*}
 As we can see, $\hat T(r)$ depends smoothly on $r>0$. In fact, this formula allows us to define $\hat T(r)$ smoothly for $r\leq 0$ as well. It turns out that $\hat T(r)$ is even:
 \begin{align*}
  \hat T(-r) &= \frac{1}{2}\int\limits^{2\pi}_{0} v(-r\cos (\varphi), -r\sin (\varphi))\, d\varphi = \frac{1}{2}\int\limits^{2\pi}_{0} v(r\cos (\varphi+\pi), r\sin (\varphi+\pi))\, d\varphi\\
  &= \frac{1}{2}\int\limits^{2\pi}_{0} v(r\cos (\varphi), r\sin (\varphi))\, d\varphi = \hat T(r).
 \end{align*}
 After shrinking $\hat U_M$ if necessary, we can assume that $\hat U_M$ has a compact neighborhood. Thus, $\hat v\in C^\infty (\hat U_M,\R)$ is bounded from below by a positive constant $v_{\min}>0$. Therefore, $v\coloneqq \hat v\circ\hat\psi^{-1}_M$ is also bounded from below by $v_{\min}$. This implies:
 \begin{gather*}
  \hat T(r) = \frac{1}{2}\int\limits^{2\pi}_{0} v(r\cos (\varphi), r\sin (\varphi)) \geq \pi v_{\min} > 0\quad\forall r.
 \end{gather*}
 \textbf{Step 3}\\\\
 Lastly, we use these properties of $\hat T(r)$ to define a $C^1$-diffeomorphism $\psi_L:U_L\to V_L\subset\R$ on a neighborhood $U_L\subset L$ of $0\in\R = L$ which rescales the periods of the trajectories $\gamma$ to a fixed period $T > 0$. Consider a trajectory $\gamma$ near $p$ with $f\circ \gamma = r^2$ again. We want to define $\psi_L (s)$ with $s = r^2$ in such a way that the rescaled trajectory $\Gamma (t)\coloneqq \gamma (\hat T(r)\cdot t/T)$ is a trajectory of the rescaled RHS $(U_M,\omega\vert_{U_M},H)$ ($U_M$ and $H$ as above). Hence, we want $\Gamma$ to be an integral curve of $X_H$:
 \begin{align*}
  \dot\Gamma (t) &= \frac{\hat T(r)}{T}\dot\gamma \left(\frac{\hat T(r)}{T}t\right) = \frac{\hat T(r)}{T} X_f (\Gamma (t))\\
  &\stackrel{!}{=} X_H (\Gamma (t)) = \frac{d\psi_L}{ds} (f\circ\Gamma (t)) X_f (\Gamma (t)) = \frac{d\psi_L}{ds} (r^2) X_f (\Gamma (t))
 \end{align*}
 Thus, we obtain the following condition for $\psi_L$:
 \begin{gather}
  \frac{d\psi_L}{ds} (r^2) = \frac{\hat T(r)}{T}.\label{eq:psi_L}
 \end{gather}
 To solve Equation \eqref{eq:psi_L}, we define the function $g:\R\to\R$ by $g(s)\coloneqq \sqrt{|s|}$. $g$ is continuous on $\R$ and smooth on $\R\backslash\{0\}$. Thus, $\hat T\circ g$ is also continuous on a neighborhood $U_L$ of $0$ and smooth on $U_L\backslash\{0\}$. Now we define $\psi_L$ as
 \begin{gather*}
  \psi_L (s)\coloneqq \frac{1}{T}\int\limits^s_0 \hat T\circ g (s^\prime) ds^\prime = \frac{1}{T} \int\limits^s_0 \hat T (\sqrt{|s^\prime|}) ds^\prime.
 \end{gather*}
 Since $\hat T\circ g$ is continuous, $\psi_L$ exists, is a $C^1$-function on $U_L$, and smooth on $U_L\backslash\{0\}$. Furthermore:
 \begin{gather*}
  \frac{d\psi_L}{ds} (s) = \frac{\hat T (\sqrt{|s|})}{T}\geq \frac{\pi v_{\min}}{T} >0.
 \end{gather*}
 Therefore, $\psi_L$ is also a $C^1$-diffeomorphism. The last equation together with the fact that $\hat T (r)$ is even shows Equation \eqref{eq:psi_L} concluding the proof:
 \begin{gather*}
  \frac{d\psi_L}{ds} (r^2) = \frac{\hat T (|r|)}{T} = \frac{\hat T(r)}{T}.
 \end{gather*}
\end{proof}

\begin{Rem}[No regularity issues in a real analytic setup]\label{rem:no_reg_issue}
 If all objects in Lemma \autoref{lem:morse_darboux_lem_I} are real analytic instead of smooth, then the regularity issues do not occur, i.e., the chart $\psi_L$ can be chosen to be real analytic. We can see this as follows: By similar arguments as before, the map $\hat T:\R\to\R$ assigning the period $\hat T(r)$ to each radius $r$ is given by
 \begin{gather*}
  \hat T(r) = \frac{1}{2}\int\limits^{2\pi}_{0} v(r\cos (\varphi), r\sin (\varphi))\, d\varphi
 \end{gather*}
 and, hence, real analytic, as $v$ is real analytic. Thus, $\hat T$ can be written as a power series near $r = 0$:
 \begin{gather*}
  \hat T (r) = \sum^\infty_{k=0} a_k r^k.
 \end{gather*}
 Now recall that $\hat T$ is even, therefore, only even powers occur in the power series of $\hat T$:
 \begin{gather*}
  \hat T (r) = \sum^\infty_{k=0} a_{2k} r^{2k}.
 \end{gather*}
 This allows us to define the real analytic function $\hat t$ by
 \begin{gather*}
  \hat t (s) \coloneqq \sum^\infty_{k=0} a_{2k} s^k.
 \end{gather*}
 Obviously, $\hat T$ and $\hat t$ satisfy the relation: $\hat t (r^2) = \hat T (r)$. Now we define the real-analytic chart $\psi_L$ by
 \begin{gather*}
  \psi_L (s)\coloneqq \frac{1}{T}\int\limits^s_0 \hat t (s^\prime) ds^\prime.
 \end{gather*}
 As in the proof of Lemma \autoref{lem:morse_darboux_lem_I}, all non-constant trajectories of the Hamiltonian $\psi_L\circ f$ are $T$-periodic, since $\psi_L$ fulfills the equation
 \begin{gather*}
  \frac{d\psi_L}{ds} (r^2) = \frac{\hat t (r^2)}{T} = \frac{\hat T (r)}{T}.
 \end{gather*}
\end{Rem}

\noindent Next, we prove Lemma \autoref{lem:morse_darboux_lem_II}.

\newpage

\begin{Lem*}[Morse Darboux lemma II]
 Let $(M^2,\omega)$ be a symplectic $2$-manifold and let\linebreak $H\in C^\infty (M,\R)$ be a smooth function on $M$ with non-degenerate critical point $p\in M$ of Morse index $\mu_H (p)\neq 1$. Further, let $T>0$ be a positive real number. Then, the following statements are equivalent:
 \begin{enumerate}
  \item There exists a topological chart $\psi_M = (x,y):U_M\to V_M\subset\R^2$ of $M$ near $p$ which is smooth on $U_M\backslash\{p\}$ such that ($\psi_M (p) = 0$):
  \begin{enumerate}[label = (\alph*)]
   \item $H\vert_{U_M} = H(p) \pm \frac{\pi}{T}(x^2 + y^2)$,
   \item $\omega\vert_{U_M} = dx\wedge dy$.
  \end{enumerate}
  \item There exists an open neighborhood $U_M\subset M$ of $p$ such that all non-constant trajectories of the RHS $(U_M, \omega\vert_{U_M}, H\vert_{U_M})$ are $T$-periodic.
  \item There exists a number $E_0 > 0$ such that $\int_{U(E)} \omega = T\cdot E$ for every $E\in [0, E_0]$,\linebreak where $U(E)$ is the connected component containing $p$ of the set\linebreak $\{q\in M\mid |H(q)-H(p)|\leq E\}$.
 \end{enumerate}
\end{Lem*}

\begin{proof}
 The idea of the proof is simple: The implications ``$1.\Rightarrow 2.$'' and ``$2.\Rightarrow 3.$'' follow from straightforward computations. To show the remaining implication ``$3.\Rightarrow 1.$'', we first choose a Morse chart $\hat \psi_M = (\hat x, \hat y):\hat U_M\to \hat V_M$ such that $H\vert_{\hat U_M} = H(p) + \varepsilon\frac{\pi}{T} (\hat x^2 + \hat y^2)$, where $\varepsilon\in\{-1,+1\}$. In general, $\hat \psi_M$ is not a Darboux chart for $\omega$. Still, the trajectories of the RHS $(\hat U_M, \omega\vert_{\hat U_M}, H\vert_{\hat U_M})$ are circles, in particular periodic orbits. Solely the angular velocity of these circles might not be constant. To rectify this, we go into polar coordinates $(r,\varphi)$ and apply an appropriately chosen diffeomorphism to $\varphi$. This operation results in a new chart $\psi_M$. Since we have not changed the radius $r$, $\psi_M$ is still a Morse chart. However, the change of the angle coordinate turns $\psi_M$ into a Darboux chart.\\\\
 $\boxed{1.\Rightarrow 2.}$\\\\
 In the chart $\psi_M$, we find for the Hamiltonian vector field $X_H$:
 \begin{gather*}
  d\psi_M(X_H) = \pm\frac{2\pi}{T}\begin{pmatrix}-y\\ x\end{pmatrix}.
 \end{gather*}
 Hence, the integral curves $\gamma$ of $X_H$ are given by:
 \begin{gather*}
  \gamma (t) = \begin{pmatrix}r_0 \cos (\varphi_0 \pm \frac{2\pi}{T}t)\\ r_0 \sin (\varphi_0 \pm \frac{2\pi}{T}t)\end{pmatrix}.
 \end{gather*}
 Thus, the trajectories $\gamma$ near $p$ are $T$-periodic.\\\\
 $\boxed{2.\Rightarrow 3.}$\\\\
 $p$ is a non-degenerate critical point of $H$ with Morse index $\mu_H (p) \neq 1$, hence, we can find a Morse chart $\hat \psi_M = (\hat x,\hat y):\hat U_M\to \hat V_M$ of $M$ near $p$ such that ($\hat\psi_M (p) = 0$)
 \begin{gather*}
  H\vert_{\hat U_M} = H(p) + \varepsilon\frac{\pi}{T}(\hat x^2 + \hat y^2)
 \end{gather*}
 for $\varepsilon\in\{-1,+1\}$. In this chart, we have $\omega\vert_{\hat U_M} = \hat v\cdot d\hat x\wedge d\hat y$ for $\hat v\in C^\infty (\hat U_M,\R)$. After permuting $\hat x$ and $\hat y$ if necessary, we can assume that $\hat v >0$. Let $\gamma$ be an integral curve of $X_H$. In polar coordinates $(\hat x, \hat y) = (r\cos (\varphi), r\sin (\varphi))$, we can parameterize $\gamma$ via $r_\gamma, \varphi_\gamma:\R\to\R$ as follows:
 \begin{gather*}
  \gamma (t) = \begin{pmatrix}
                r_\gamma (t) \cos (\varphi_\gamma (t))\\
                r_\gamma (t) \sin (\varphi_\gamma (t))
               \end{pmatrix}.
 \end{gather*}
 With $v\coloneqq \hat v\circ \hat\psi^{-1}_M$, the integral curve equation becomes:
 \begin{gather*}
  \dot r_\gamma = 0;\quad \dot \varphi_\gamma = \frac{2\pi\varepsilon}{Tv(r_\gamma\cos (\varphi_\gamma), r_\gamma\sin (\varphi_\gamma))}.
 \end{gather*}
 Thus, $r_\gamma\equiv r$ is constant. This allows us to define $\Phi:\R\to \R$ by
 \begin{align*}
  \Phi (t)&\coloneqq \varphi_\gamma (t) - \varphi_\gamma (0) = \int\limits^t_0 \dot\varphi_\gamma (t^\prime) dt^\prime\\
  &= \int\limits^t_0 \frac{2\pi\varepsilon}{T v(r\cos (\varphi_\gamma (t^\prime)), r\sin (\varphi_\gamma(t^\prime)))} dt^\prime.
 \end{align*}
 As in the proof of Lemma \autoref{lem:morse_darboux_lem_I}, $\Phi^{-1}$ exists and is given by:
 \begin{gather*}
  \Phi^{-1} (\alpha) = \frac{T}{2\pi\varepsilon}\int\limits^{\varphi_0+\alpha}_{\varphi_0} v(r\cos (\varphi), r\sin (\varphi))\, d\varphi
 \end{gather*}
 with $\varphi_0\coloneqq \varphi_\gamma (0)$. We now use the fact that, by assumption, $\gamma$ is $T$-periodic, so $\Phi^{-1}$ satisfies:
 \begin{gather*}
  \Phi^{-1} (2\pi) = \varepsilon T\quad\Rightarrow \int\limits^{2\pi}_0 v(r\cos (\varphi), r\sin (\varphi)) d\varphi = 2\pi.
 \end{gather*}
 Observe that the last equation holds for all $r>0$ small enough. 
 %By continuity and the fact that polar coordinates are ``invariant'' under the transformation $r\mapsto -r$ and $\varphi\mapsto \varphi + \pi$, it even holds for all $r\in (-\varepsilon_0,\varepsilon_0)$ with $\varepsilon_0>0$ small enough.
 It allows us to compute the symplectic area $\int_{U(E)}\omega$:
 \begin{align*}
  \int\limits_{U(E)} \omega &= \int\limits^{\sqrt{\frac{TE}{\pi}}}_0 \int\limits^{2\pi}_0 v(r\cos(\varphi), r\sin(\varphi))\, rdr\, d\varphi\\
  &= \int\limits^{\sqrt{\frac{TE}{\pi}}}_0 \left(\int\limits^{2\pi}_0 v(r\cos(\varphi), r\sin(\varphi))\, d\varphi\right) rdr\\
  &= \int\limits^{\sqrt{\frac{TE}{\pi}}}_0 2\pi r\, dr = T\cdot E.
 \end{align*}
 $\boxed{3.\Rightarrow 1.}$\\\\
 As in ``$2.\Rightarrow 3.$'', we can find a Morse chart $\hat \psi_M = (\hat x,\hat y):\hat U_M\to \hat V_M$ of $M$ near $p$ such that ($\hat\psi_M (p) = 0$)
 \begin{gather*}
  H\vert_{\hat U_M} = H(p) + \varepsilon\frac{\pi}{T}(\hat x^2 + \hat y^2)
 \end{gather*}
 for $\varepsilon\in\{-1,+1\}$ and $\omega\vert_{\hat U_M} = \hat v\cdot d\hat x\wedge d\hat y$ for $\hat v\in C^\infty (\hat U_M,\R)$ with $\hat v >0$. By taking the derivative of $\int_{U(E)}\omega = T\cdot E$ with respect to $E$, we deduce that
 \begin{gather*}
  \int\limits^{2\pi}_0 v(r\cos (\varphi), r\sin (\varphi))\, d\varphi = 2\pi
 \end{gather*}
 for $r>0$ small enough and $v\coloneqq \hat v\circ \hat\psi^{-1}_M$. The last equation implies that the map\linebreak $P:(0,\varepsilon_0)\times S^1\to (0,\varepsilon_0)\times S^1$ given by
 \begin{gather*}
  P(r, [\varphi])\coloneqq \left(r, \left[\int\limits^{\varphi}_0 v(r\cos (\varphi^\prime), r\sin (\varphi^\prime))\, d\varphi^\prime\right]\right)
 \end{gather*}
 is well-defined for $\varepsilon_0 >0$ small enough. $P$ is a smooth diffeomorphism, since the Jacobian of $P$ is $v>0$. Denote the map associated with polar coordinates by $S:\R_+ \times S^1\to \R^2\backslash\{0\}$, i.e., $S(r, [\varphi])\coloneqq (r\cos (\varphi), r\sin (\varphi))$, and consider the map $S\circ P\circ S^{-1}:\dot D_{\varepsilon_0}\to \dot D_{\varepsilon_0}$, where $\dot D_{\varepsilon_0}\coloneqq \{x\in\R^2\backslash\{0\}\mid ||x||\leq \varepsilon_0\}$. $S\circ P\circ S^{-1}$ is a smooth diffeomorphism, since both $S$ and $P$ are smooth diffeomorphisms. Furthermore, $S\circ P\circ S^{-1}$ maps circles of radius $r$ to circles of radius $r$, hence, we can extend $S\circ P\circ S^{-1}$ to a homeomorphism on $D_{\varepsilon_0}\coloneqq\{x\in\R^2\mid ||x||\leq \varepsilon_0\}$ by setting $S\circ P\circ S^{-1} (0)\coloneqq 0$.\\
 Now consider the map $(x,y)\equiv \psi_M\coloneqq S\circ P\circ S^{-1}\circ\hat \psi_M:U_M\to V_M$. $\psi_M$ is a topological chart of $M$ near $p$ and smooth on $U_M\backslash\{p\}$. Recall that $\hat \psi_M$ is a Morse chart for $H$, thus, the level sets of $H$ are circles in the chart $\hat \psi_M$. Since the charts $\psi_M$ and $\hat \psi_M$ only differ by postcomposition with $S\circ P\circ S^{-1}$ which preserves circles, $\psi_M$ is also a Morse chart for $H$:
 \begin{gather*}
  H\vert_{U_M} = H(p) + \varepsilon\frac{\pi}{T}(x^2 + y^2).
 \end{gather*}
 Furthermore, the fact that the Jacobian of $P$ is $v$ implies that
 \begin{gather*}
  \omega\vert_{U_M} = dx\wedge dy
 \end{gather*}
 concluding the proof.
\end{proof}

 \newpage
 \section{Various Action Functionals for HHSs and PHHSs}
\label{app:action_functionals}

In \autoref{subsec:holo_action_fun_and_prin} and \autoref{subsec:def_PHHS}, we have defined and studied action functionals for HHSs and PHHSs. The ``critical points'' of these action functionals gave us (pseudo-)holomorphic trajectories of the system under consideration whose domains are parallelograms in the complex plane. However, these action functionals are not the only functionals whose ``critical points'' can be linked to (pseudo-)holomorphic trajectories. There is, in fact, an abundance of action functionals that differ in the domain of their trajectories and the way the ``one-dimensional'' action functionals of their underlying RHSs are integrated. In this part of the appendix, we present and examine a large selection of such action functionals. First, we only formulate and explore action functionals for HHSs. Afterwards, we explain how these action functionals need to be modified in order to give action functionals for PHHSs. Hereby, we realize that the presented action functionals for PHHSs are all real-valued. From this point of view, a Floer-like theory for PHHSs revolving around these real-valued functionals might be possible.\\
We begin by defining an action functional for holomorphic trajectories whose domains are disks $D^{z_0}_R\subset\mathbb{C}$ of radius $R>0$ centered at $z_0\in\mathbb{C}$. To do that, we first need to partition the disk $D^{z_0}_R$ into lines. We choose the partition consisting of lines starting at the center $z_0$ and ending at any boundary point $z\in\partial D^{z_0}_R$. For every such radial line, we consider the action functional $\mathcal{A}^{\Lambda}_{e^{i\alpha}\mH}$ from Remark \autoref{rem:tilted_traj}. We now obtain an action functional for HHSs by integrating the action $\mathcal{A}^{\Lambda}_{e^{i\alpha}\mH}$ over all radial lines, i.e., $\alpha\in [0,2\pi]$:

\begin{Prop}[Action functional $\actdiski$]\label{prop:actdiski}
 Let $(X,\Omega = d\Lambda, \mH)$ be an exact\linebreak HHS, $D^{z_0}_R\coloneqq\{z\in\mathbb{C}\mid |z-z_0|\leq R\}$ be a disk of radius $R>0$ centered at $z_0\in\mathbb{C}$,\linebreak $\mathcal{P}_{D^{z_0}_R}\coloneqq C^\infty (D^{z_0}_R, X)$ be the set of smooth maps from $D^{z_0}_R$ to $X$, and $\actdiski:\mathcal{P}_{D^{z_0}_R}\to\mathbb{C}$ be the action functional defined by
 \begin{gather*}
  \actdiski [\gamma]\coloneqq \frac{1}{2\pi}\int\limits^{2\pi}_0\int\limits^R_0\left[\Lambda\vert_{\gamma_\alpha (r)}\left(\frac{d\gamma_\alpha}{dr}(r)\right) - e^{i\alpha}\cdot \mH\circ\gamma_\alpha (r)\right]dr\, d\alpha\quad\forall \gamma \in\mathcal{P}_{D^{z_0}_R},
 \end{gather*}
 where $\gamma_\alpha:[0,R]\to X$ is defined by $\gamma_\alpha (r)\coloneqq \gamma (z_0 + re^{i\alpha})$. Now let $\gamma\in\mathcal{P}_{D^{z_0}_R}$. Then, $\gamma$ is a holomorphic trajectory of the HHS $(X,\Omega,\mH)$ iff $\gamma$ is a ``critical point'' of $\actdiski$. Here, ``critical points'' means that we only allow for those variations of $\gamma$ which keep $\gamma$ fixed at the boundary $\partial D^{z_0}_R$ and the \underline{center} $z_0$.
\end{Prop}

\begin{proof}
 Take the notations from above. Using Remark \autoref{rem:tilted_traj} and writing $\actdiski$ as
 \begin{gather*}
  \actdiski [\gamma] = \frac{1}{2\pi}\int\limits^{2\pi}_0\mathcal{A}^{\Lambda}_{e^{i\alpha}\mH} [\gamma_\alpha] d\alpha,
 \end{gather*}
 we can show as in the proof of Lemma \autoref{lem:holo_action_prin} that $\gamma$ is a ``critical point'' of $\actdiski$ iff\linebreak $\gamma_\alpha:[0,R]\to X$ is a (real) integral curve of $\cos (\alpha)\cdot X^R_\mH + \sin (\alpha)\cdot J(X^R_\mH)$ for every $\alpha\in [0,2\pi]$, where $X_\mH = 1/2 (X^R_\mH - iJ(X^R_\mH))$ is the Hamiltonian vector field of $(X,\Omega,\mH)$. Thus, a ``critical point'' $\gamma$ is uniquely determined, given an initial value $x_0\coloneqq\gamma (z_0)$, by:
 \begin{gather*}
   \gamma(z_0 + re^{i\alpha}) = \varphi^{\cos (\alpha)\cdot X^R_\mH + \sin (\alpha)\cdot J(X^R_\mH)}_r (x_0) = \varphi^{r\cos (\alpha)\cdot X^R_\mH + r\sin (\alpha)\cdot J(X^R_\mH)}_1 (x_0),
 \end{gather*}
 where $\varphi^V_t$ is the time-$t$-flow of a real vector field $V$ on $X$. Comparing the last equation with the formula for the holomorphic trajectory $\gamma^{z_0,x_0}$ satisfying $\gamma (z_0) = x_0$ given in the proof of Proposition \autoref{prop:holo_traj} shows that $\gamma$ is a holomorphic trajectory of the HHS $(X,\Omega,\mH)$ iff $\gamma$ is a ``critical point'' of $\actdiski$.\\
 Lastly, we have to explain why a ``critical point'' $\gamma$ of $\actdiski$ needs to fix the variations of $\gamma$ at the boundary $\partial D^{z_0}_R$ and the center $z_0$. Recall the variation of $\mathcal{A}^{\Lambda}_{e^{i\alpha}\mH}$ at $\gamma_\alpha$. In general, the variation of this functional also includes terms associated with the boundary of the image of $\gamma_\alpha$. This boundary consists of two points, namely the center $z_0$ and one boundary point $z\in\partial D^{z_0}_R$. To get rid of these boundary terms in the variation of $\actdiski$, we have to keep $\gamma$ fixed at $z_0$ and $\partial D^{z_0}_R$.
\end{proof}

In \autoref{subsec:holo_action_fun_and_prin}, we have given two reasons why we need to vary over all smooth curves $\gamma$ and cannot simple restrict the variational problem to the space of holomorphic curves $\gamma$. The new-found action functional offers an additional perspective on that matter. It maps every holomorphic curve $\gamma$ to zero, hence, only varying it over the space of holomorphic curves is meaningless:

\begin{Prop}\label{prop:actdiski_value}
 Take the assumptions and notations from Proposition \autoref{prop:actdiski}. Further, let $\gamma:D^{z_0}_R\to X$ be any holomorphic map from $D^{z_0}_R$ to $X$. Then:
 \begin{gather*}
  \actdiski [\gamma] = 0.
 \end{gather*}
\end{Prop}

\begin{proof}
 Take the assumptions and notations from Proposition \autoref{prop:actdiski} and let $\gamma:D^{z_0}_R\to X$ be holomorphic. Using the relation $\Lambda (J\cdot) = i\cdot \Lambda$ for holomorphic $1$-forms, we find:
 \begin{gather*}
  \Lambda\vert_{\gamma_\alpha(r)}\left(\frac{d\gamma_\alpha}{dr}(r)\right) = e^{i\alpha}\cdot\Lambda\vert_{\gamma_\alpha (r)}\left( \gamma^{\prime}(z_0 + re^{i\alpha})\right),
 \end{gather*}
 where $\gamma^{\prime}$ is the complex derivative of $\gamma$. With this, we obtain:
 \begin{align*}
  \actdiski [\gamma] &= \frac{1}{2\pi}\int\limits^{2\pi}_0\int\limits^R_0\left[\Lambda\vert_{\gamma_\alpha (r)}\left(\frac{d\gamma_\alpha}{dr}(r)\right) - e^{i\alpha}\cdot \mH\circ\gamma_\alpha (r)\right]dr\, d\alpha\\
  &= \frac{1}{2\pi}\int\limits^{2\pi}_0\int\limits^R_0\left[\Lambda\vert_{\gamma (z_0 + re^{i\alpha})}\left(\gamma^\prime (z_0 + re^{i\alpha})\right) - \mH\circ\gamma (z_0 + re^{i\alpha})\right]dr\, e^{i\alpha}\, d\alpha\\
  &= \int\limits^R_0 \frac{1}{2\pi i}\oint\limits_{|z| = 1}\left[\Lambda\vert_{\gamma (z_0 + rz)}\left(\gamma^\prime (z_0 + rz)\right) - \mH\circ\gamma (z_0 + rz)\right]dz\, dr\\
  &= \int\limits^R_0 \text{\normalfont Res}_{z = 0}\left[\Lambda\vert_{\gamma (z_0 + rz)}\left(\gamma^\prime (z_0 + rz)\right) - \mH\circ\gamma (z_0 + rz)\right] dr,
 \end{align*}
 where $\text{\normalfont Res}_{z = 0} [f(z)]$ denotes the residue of a meromorphic function $f(z)$ at $z = 0$. In the last line of the computation, we have used Cauchy's theorem. The function for which we need to determine the residue is clearly holomorphic at $z = 0$. Thus, the residue is $0$ and we see that the action vanishes for holomorphic curves concluding the proof.
\end{proof}

Even though the ``critical values'' of $\actdiski$ are nice and easy to understand, the action functional itself does not appear to be particularly useful. Often, we want to modify action functionals such that trajectories become actual critical points. The standard ways to achieve this are to either put the boundary of the trajectory on Lagrangian submanifolds or to impose periodicity. Both ways do not appear to be meaningful here. For the presented action functional, periodicity means periodicity of the radial lines. Thus, a ``periodic'' curve $\gamma:D^{z_0}_R\to X$ needs to attain the same value on its boundary as on its center. However, the only holomorphic maps $\gamma:D^{z_0}_R\to X$ exhibiting such a behavior are constant curves by the identity theorem.\\
The other method, mapping the ``boundary'' to Lagrangian submanifolds, takes an unnatural and downright ugly form here, namely mapping $z_0$ and $\partial D^{z_0}_R$ to Lagrangian submanifolds. The action functional $\actdiskii$ improves on $\actdiski$ in that regard. To avoid boundary terms associated with $z_0$, which are at the center\footnote{Cum grano salis.} of our problem, we now partition the disk $D^{z_0}_R$ into lines starting at $z_0 - z$ and ending at $z_0 + z$ ($|z| = R$). In order to account for the doubled length of the radial lines, we only integrate over the angles $\alpha\in [0,\pi]$ this time:

\begin{Prop}[Action functional $\actdiskii$]\label{prop:actdiskii}
 Let $(X,\Omega = d\Lambda, \mH)$ be an exact\linebreak HHS, $D^{z_0}_R\coloneqq\{z\in\mathbb{C}\mid |z-z_0|\leq R\}$ be a disk of radius $R>0$ centered at $z_0\in\mathbb{C}$,\linebreak $\mathcal{P}_{D^{z_0}_R}\coloneqq C^\infty (D^{z_0}_R, X)$ be the set of smooth maps from $D^{z_0}_R$ to $X$, and $\actdiskii:\mathcal{P}_{D^{z_0}_R}\to\mathbb{C}$ be the action functional defined by
 \begin{gather*}
  \actdiskii [\gamma]\coloneqq \frac{i}{4R}\int\limits^{\pi}_0\int\limits^R_{-R}\left[\Lambda\vert_{\gamma_\alpha (r)}\left(\frac{d\gamma_\alpha}{dr}(r)\right) - e^{i\alpha}\cdot \mH\circ\gamma_\alpha (r)\right]dr\, d\alpha\quad\forall \gamma \in\mathcal{P}_{D^{z_0}_R},
 \end{gather*}
 where $\gamma_\alpha:[-R,R]\to X$ is defined by $\gamma_\alpha (r)\coloneqq \gamma (z_0 + re^{i\alpha})$. Now let $\gamma\in\mathcal{P}_{D^{z_0}_R}$. Then, $\gamma$ is a holomorphic trajectory of the HHS $(X,\Omega,\mH)$ iff $\gamma$ is a ``critical point'' of $\actdiskii$. Here, ``critical points'' means that we only allow for those variations of $\gamma$ which keep $\gamma$ fixed at the boundary $\partial D^{z_0}_R$.
\end{Prop}

\begin{proof}
 The proof works as the proof of Proposition \autoref{prop:actdiski} by writing $\actdiskii$ as
 \begin{gather*}
  \actdiskii [\gamma] = \frac{i}{4R}\int\limits^{\pi}_0\mathcal{A}^{\Lambda}_{e^{i\alpha}\mH} [\gamma_\alpha] d\alpha.
 \end{gather*}
 Here, the variations of $\gamma$ only need to keep $\gamma$ fixed at the boundary $\partial D^{z_0}_R$, since the radial lines start and end at $\partial D^{z_0}_R$.
\end{proof}

\begin{Rem}[No Proposition \autoref{prop:actdiski_value} for $\actdiskii$]\label{rem:actdiskii_value}
 Proposition \autoref{prop:actdiski_value} does not apply to $\actdiskii$. In fact, the normalization in Proposition \autoref{prop:actdiskii} is chosen such that the action of constant curves is given by the Hamilton function:
 \begin{gather*}
  \actdiskii [\gamma_{x_0}] = \mH (x_0),
 \end{gather*}
 where $\gamma_{x_0} (z)\coloneqq x_0\in X$ for every $z\in D^{z_0}_{R}$. Thus, any singular point $x_0$ of $\mH$ with $\mH (x_0)\neq 0$ provides a counterexample to Proposition \autoref{prop:actdiski_value} for $\actdiskii$.
\end{Rem}

If we modify $\actdiskii$ such that the holomorphic trajectories become actual critical points, we see that this action is a bit more reasonable. In the Lagrangian case, we now restrict the space of smooth curves $\gamma:D^{z_0}_R\to X$ to the space of those curves which only map the boundary $\partial D^{z_0}_R$ to Lagrangian submanifolds, as one would expect. However, the modification via periodicity still only gives trivial results. One can see this as follows: now, periodicity means periodicity of radial lines starting and ending at $\partial D^{z_0}_R$. In this sense, we say $\gamma:D^{z_0}_R\to X$ is ``periodic'' if it assigns the same value to opposite points on the boundary $\partial D^{z_0}_R$. For the sake of simplicity, let us now assume $z_0 = 0$. For such a ``periodic'' $\gamma$, define $\gamma_-$ by $\gamma_- (z)\coloneqq \gamma (-z)$. If $\gamma$ is holomorphic, then $\gamma_-$ is also holomorphic and, by assumption, attains the same values on $\partial D^{0}_R$ as $\gamma$. Hence, by the identity theorem, $\gamma$ and $\gamma_-$ denote the same map. However, if $\gamma$ is even a holomorphic trajectory, then $\gamma$ is an integral curve of the Hamiltonian vector field $X_\mH$ and we have:
\begin{gather*}
 X_\mH (\gamma (z)) = \frac{\pa}{\pa z}(\gamma (z)) = \frac{\pa}{\pa z} (\gamma (-z)) = - X_\mH (\gamma (-z)) = -X_\mH (\gamma (z)).
\end{gather*}
Thus, the Hamiltonian vector field vanishes in this case and $\gamma$ is a constant curve.\\
We cannot only formulate $\actdiski$ and $\actdiskii$ for disks $D^{z_0}_R$, but for any bounded star-shaped domain\footnote{Here, a domain $D\subset\mathbb{C}$ is a path-connected subset of $\mathbb{C}$ with non-empty interior $D^\circ$ dense in $D$.} $D\subset\mathbb{C}$ with smooth boundary\footnote{The boundary $b$ is parameterized by the polar angle $\alpha$ in the decomposition $z = z_0 + re^{i\alpha}\in D$.} $b:\mathbb{R}/2\pi\mathbb{Z}\to \partial D$:

\begin{Prop}[Action functionals $\acti$ and $\actii$ for bounded star-shaped domains] \label{prop:action_starshaped}
 Let $(X,\Omega = d\Lambda, \mH)$ be an exact HHS, let $D\subset\mathbb{C}$ be a bounded domain in $\mathbb{C}$ which is star-shaped with respect to $z_0$ and has smooth boundary $b:\mathbb{R}/2\pi\mathbb{Z}\to\partial D$, and let $\mathcal{P}_{D}\coloneqq C^\infty (D, X)$ be the set of smooth maps from $D$ to $X$. Then, we can define the action functionals $\acti:\mathcal{P}_{D}\to\mathbb{C}$ and $\actii:\mathcal{P}_D\to \mathbb{C}$ by
 \begin{align*}
  \acti [\gamma]&\coloneqq \frac{1}{2\pi}\int\limits^{2\pi}_0\int\limits^{R(\alpha)}_0\left[\Lambda\vert_{\gamma_\alpha (r)}\left(\frac{d\gamma_\alpha}{dr}(r)\right) - e^{i\alpha}\cdot \mH\circ\gamma_\alpha (r)\right]dr\, d\alpha,\\
  \actii [\gamma]&\coloneqq \frac{i}{4\hat{R}}\int\limits^{\pi}_0\int\limits^{R(\alpha)}_{-R(\alpha - \pi)}\left[\Lambda\vert_{\gamma_\alpha (r)}\left(\frac{d\gamma_\alpha}{dr}(r)\right) - e^{i\alpha}\cdot \mH\circ\gamma_\alpha (r)\right]dr\, d\alpha,
 \end{align*}
 where $\gamma\in\mathcal{P}_D$, $R:\mathbb{R}/2\pi\mathbb{Z}\to\mathbb{R}$ is defined by $R(\alpha)\coloneqq |b(\alpha)-z_0|$, $\gamma_\alpha:[-R(\alpha - \pi), R(\alpha)]\to X$ is given by $\gamma_\alpha (r)\coloneqq \gamma(z_0 + re^{i\alpha})$, and $\hat{R}$ is defined by
 \begin{gather*}
  \hat{R}\coloneqq \frac{i}{4}\left[\int\limits^{2\pi}_{\pi} R(\alpha)e^{i\alpha} d\alpha - \int\limits^{\pi}_{0} R(\alpha)e^{i\alpha} d\alpha\right].
 \end{gather*}
 Now let $\gamma\in\mathcal{P}_{D}$. Then, $\gamma$ is a holomorphic trajectory of the HHS $(X,\Omega,\mH)$ iff $\gamma$ is a ``critical point''\footnote{In the sense of Proposition \autoref{prop:actdiski}.} of $\acti$ iff $\gamma$ is a ``critical point''\footnote{In the sense of Proposition \autoref{prop:actdiskii}.} of $\actii$.
\end{Prop}

\begin{proof}
 Confer the proofs of Proposition \autoref{prop:actdiski} and \autoref{prop:actdiskii}.
\end{proof}

\begin{Rem}[Normalization of $\acti$ and $\actii$]\label{rem:normalization}
 The normalization of $\acti$ and $\actii$ are chosen such that they coincide with our previous definitions for $D = D^{z_0}_R$ being a disk. In particular, $\actii$ agrees with the Hamilton function $\mH$ for constant curves $\gamma$.
\end{Rem}

One might wonder how the action functionals $\acti$ and $\actii$ are related to the action functional $\mathcal{A}^{P_\alpha}_\mH$ for parallelograms $P_\alpha$ from \autoref{subsec:holo_action_fun_and_prin} and \autoref{subsec:def_PHHS}, especially because a parallelogram $P_\alpha$ is also a bounded star-shaped domain in $\mathbb{C}$. When we modify the functionals $\acti$ and $\actii$ to describe general PHHSs, we will see that $\acti$ and $\actii$ differ a lot from $\mathcal{A}^{P_\alpha}_\mH$. To compare $\mathcal{A}^{P_\alpha}_\mH$ directly with $\acti$ and $\actii$, let us express $\acti$ and $\actii$ in the same coordinates as $\mathcal{A}^{P_\alpha}_\mH$, namely Cartesian coordinates $z = t + is$:

\begin{Prop}[$\acti$ and $\actii$ in Cartesian coordinates]\label{prop:cartesian_coordinates}
 Employ the assumptions and notations from Proposition \autoref{prop:action_starshaped}. For $\gamma\in\mathcal{P}_D$, we define the following derivatives in Cartesian coordinates $z = t + is \in D$:
 \begin{gather*}
  \frac{\partial\gamma}{\partial z} (z)\coloneqq \frac{1}{2}\left(\frac{\partial\gamma}{\partial t} (z) - i\frac{\partial\gamma}{\partial s} (z)\right);\quad \frac{\partial\gamma}{\partial \bar{z}} (z)\coloneqq \frac{1}{2}\left(\frac{\partial\gamma}{\partial t} (z) + i\frac{\partial\gamma}{\partial s} (z)\right).
 \end{gather*}
 Furthermore, define the complex functions $f,g:D\to\mathbb{C}$ by:
 \begin{gather*}
  f(z)\coloneqq \Lambda\vert_{\gamma (z)}\left(\frac{\partial\gamma}{\partial z}(z)\right) - \mH\circ\gamma (z);\quad g(z)\coloneqq \Lambda\vert_{\gamma (z)}\left(\frac{\partial\gamma}{\partial \bar{z}}(z)\right).
 \end{gather*}
 Then, the action functionals $\acti$ and $\actii$ in Cartesian coordinates are given by:
 \begin{align*}
  \acti [\gamma] &= \frac{1}{2\pi}\iint\limits_{D}\left[\frac{f(z)}{\bar{z} - \bar{z_0}} + \frac{g(z)}{z - z_0}\right]dt\wedge ds,\\
  \actii [\gamma] &= \frac{i}{4\hat{R}}\iint\limits_{D^+}\left[\frac{f(z)}{\bar{z} - \bar{z_0}} + \frac{g(z)}{z - z_0}\right]dt\wedge ds - \frac{i}{4\hat{R}}\iint\limits_{D^-}\left[\frac{f(z)}{\bar{z} - \bar{z_0}} + \frac{g(z)}{z - z_0}\right]dt\wedge ds,
 \end{align*}
 where $z = t+is\in D$, $\bar{\cdot}$ denotes the complex conjugation, $D^+\coloneqq \{z\in D\mid \text{\normalfont Im}(z-z_0)\geq 0\}$, and\linebreak $D^-\coloneqq \{z\in D\mid \text{\normalfont Im}(z-z_0)\leq 0\}$.
\end{Prop}

\begin{proof}
 Take the assumptions and notations from above. We only show Proposition \autoref{prop:cartesian_coordinates} for $D = D^0_R\equiv D_R$ being a disk of radius $R>0$ centered at the origin. The general case can be shown similarly. Using the derivatives defined above, we can write:
 \begin{gather*}
  \frac{\partial\gamma}{\partial t} (z) = \frac{\partial\gamma}{\partial z} (z) + \frac{\partial\gamma}{\partial \bar{z}} (z);\quad \frac{\partial\gamma}{\partial s} (z) = i\left(\frac{\partial\gamma}{\partial z} (z) - \frac{\partial\gamma}{\partial \bar{z}} (z)\right).
 \end{gather*}
 Now consider the map $\gamma_\alpha:[-R,R]\to X$ defined in polar coordinates $z = re^{i\alpha}$ by\linebreak $\gamma_\alpha (r) = \gamma (re^{i\alpha})$ for every $\alpha\in [0,2\pi]$. $\Lambda$ applied to the derivative of $\gamma_\alpha$ gives:
 \begin{alignat*}{2}
  \Lambda\vert_{\gamma_\alpha (r)}\left(\frac{d\gamma_\alpha}{dr} (r)\right) &= \cos (\alpha)\cdot \Lambda\vert_{\gamma (re^{i\alpha})}\left(\frac{\partial\gamma}{\partial t} (re^{i\alpha})\right) &&+ \sin (\alpha)\cdot \Lambda\vert_{\gamma (re^{i\alpha})}\left(\frac{\partial\gamma}{\partial s} (re^{i\alpha})\right)\\
  &= e^{i\alpha}\cdot \Lambda\vert_{\gamma (re^{i\alpha})}\left(\frac{\partial\gamma}{\partial z} (re^{i\alpha})\right) &&+ e^{-i\alpha}\cdot \Lambda\vert_{\gamma (re^{i\alpha})}\left(\frac{\partial\gamma}{\partial \bar{z}} (re^{i\alpha})\right).
 \end{alignat*}
 Recalling the definition of $f$ and $g$, this allows us to write:
 \begin{align*}
  \Lambda\vert_{\gamma_\alpha (r)}\left(\frac{d\gamma_\alpha}{dr}(r)\right) - e^{i\alpha}\cdot \mH\circ\gamma_\alpha (r) &= e^{i\alpha}\cdot f(re^{i\alpha}) + e^{-i\alpha}\cdot g(re^{i\alpha})\\
  &= r\cdot\left[\frac{f(re^{i\alpha})}{re^{-i\alpha}} + \frac{g(re^{i\alpha})}{re^{i\alpha}}\right] = r\cdot\left[\frac{f(z)}{\bar{z}} + \frac{g(z)}{z}\right].
 \end{align*}
 The expression for $\acti$ is now obtained by inserting the last equation into the defining formula for $\acti$ and using $r\cdot dr\wedge d\alpha = dt\wedge ds$ for $re^{i\alpha} = z = t + is$. Observing that the integrands of $\acti$ and $\actii$ agree on $D^+_R$ and differ by a sign on $D^-_R$ concludes the proof.
\end{proof}

%\footnote{We have renamed the angle $\alpha$ of the parallelogram $P_\alpha$ to $\beta$ here to avoid confusion with the polar coordinates $z = re^{i\alpha}$.}

The form of $\acti$ and $\actii$ in Cartesian coordinates is rather remarkable. In fact, we can express the action functional $\mathcal{A}^{P_\alpha}_\mH$ from \autoref{subsec:holo_action_fun_and_prin} in the same form as $\acti$ for $D = P_\alpha$, just with different complex functions $f$ and $g$, namely by setting:
\begin{gather*}
 f(z) = 2\pi\overline{(z - z_0)}\cdot\left[\Lambda_R\vert_{\gamma (z)}\left(2\frac{\partial\gamma}{\partial z}(z)\right) - \mH\circ\gamma (z)\right];\quad g(z) = 0.
\end{gather*}
The similarities between $\acti$ and $\mathcal{A}^{P_\alpha}_\mH$ become even more apparent if we evaluate $\mathcal{A}^{P_\alpha}_\mH$ at holomorphic curves $\gamma:P_\alpha\to X$. As in Remark \autoref{rem:several_remarks}, Point 3, we find in this case:
\begin{gather*}
 \mathcal{A}^{P_\alpha}_\mH[\gamma] =  \iint\limits_{P_\alpha}\left[\Lambda\left(\frac{\partial\gamma}{\partial z}(z)\right) - \mH\circ\gamma (z)\right] dt\wedge ds.
\end{gather*}
Because the function $g$ as defined in Proposition \autoref{prop:cartesian_coordinates} vanishes for holomorphic $\gamma$, the only difference between $\acti$ and $\mathcal{A}^{P_\alpha}_\mH$ is now the factor $2\pi\overline{(z - z_0)}$ in the integrand. However, this seemingly small difference is rather impactful, as we will shortly see.\\
Lastly, we want to modify the actions $\acti$ and $\actii$ in such a way that they also apply to PHHSs. Recall that for an exact PHHS $(X,J;\Omega_R = d\Lambda_R,\mH_R)$ the induced $2$-form $\Omega_I$ is, in general, not closed. Hence, only the parts of $\acti$ and $\actii$ that do not include $\Lambda_I$ are well-defined for PHHSs. Precisely speaking, these are the real part of $\acti$ and the real part of $-4i\hat R\cdot \actii$. Still, these \underline{real-valued} functionals satisfy an action principle with respect to the pseudo-holomorphic trajectories of a PHHS. In fact, this can be shown in the same way as Proposition \autoref{prop:actdiski} by simply observing that Remark \autoref{rem:tilted_traj}, which is crucial for the proof of Proposition \autoref{prop:actdiski}, is also valid for the real part of the functional $\mathcal{A}^\Lambda_{e^{i\alpha}\mH}$.\\
The generalization of $\acti$ and $\actii$ to PHHSs has now revealed the most striking difference between $\acti$ and $\mathcal{A}^{P_\alpha}_\mH$: while the action principle related to $\acti$ still applies if we only consider the real part of $\acti$, both the real \underline{and} imaginary part of $\mathcal{A}^{P_\alpha}_\mH$ are crucial for the validity of the action principle related to $\mathcal{A}^{P_\alpha}_\mH$. In particular, there might exist a Floer-like theory related to the real part of $\acti$ or $-4i\hat R\cdot \actii$. For $\mathcal{A}^{P_\alpha}_\mH$, we have no intuition on how such a theory should look like, since we do not know how to interpret a complex function as a Morse function in the sense of Morse homology. Since we can turn $\acti$ and $\actii$ into real-valued action functionals, the same objections do not apply to them. Nevertheless, the question remains whether the real part of $\acti$ or $-4i\hat R\cdot \actii$ are indeed suitable Morse functions and whether the resulting Floer theories, if they exist, give any non-trivial result. At least for the (conjectured) Hamiltonian\footnote{In Hamiltonian Floer theory, one only considers periodic orbits.} Floer theory related to $\acti$ and $\actii$, it is most likely that it only gives trivial results, since all trajectories, which are ``periodic'' in a sense suitable for $\acti$ and $\actii$ as explained above, are automatically constant.

 \newpage
 \section{Almost Complex Structures on $\mathbf{TM}$ and $\mathbf{T^\ast M}$}
\label{app:almost_complex_structures}

In this part, we explain how a connection $\nabla$ on a manifold $M$ induces an almost complex structure $J_\nabla$ on its tangent bundle $TM$ (cf. \cite{Cieliebak1994}). If $\nabla = \nabla^g$ is the Levi-Civita connection of a semi-Riemannian metric $g$ on $M$, then $\nabla^g$ also defines an almost complex structure $J^\ast_{\nabla^g}$ on the cotangent bundle $T^\ast M$ via the bundle isomorphism $G:TM\to T^\ast M, v\mapsto \iota_v g$. In this case, $J^\ast_{\nabla^g}$ is compatible with the canonical symplectic form $\omega_\can$ on $T^\ast M$ in the sense that $\omega_\can (\cdot, J^\ast_{\nabla^g}\cdot)$ is a semi-Riemannian metric on $T^\ast M$ of signature $(2s, 2t)$, where $(s,t)$ is the signature of $g$. Furthermore, we will see that $J_{\nabla^g}$ or, equivalently, $J^\ast_{\nabla^g}$ is integrable if and only if $\nabla^g$ or, equivalently, $g$ is flat.\\
Let $M$ a smooth manifold of dimension $n$, $\pi_E:E\to M$ be a (smooth) vector bundle, and $\nabla$ be a linear connection\footnote{Sometimes, the term ``affine connection'' is used instead of ``linear connection''.} on the vector bundle $E\to M$. The (fiberwise) kernel of the differential $d\pi_E: TE\to TM$ yields the vertical subbundle $VE$ of $TE$, while $\nabla$ defines the horizontal subbundle $HE$ of $TE$. In fact, the notion of a horizontal subbundle $HE$ is equivalent to the notion of a linear connection $\nabla$ on $E\to M$. To see this, let $K:TE = VE\oplus HE\to E$ be the vertical projection\footnote{
 Technically speaking, the map $\hat K:TE = VE\oplus HE\to VE$ is the vertical projection. To obtain $K$ from $\hat K$, we have already exploited the fact that $E\to M$ is a vector bundle allowing us to identify the fibers of $VE$ with the fibers of $E$ via the linear isomorphism
 \begin{gather*}
  E_p\to V_wE, v\mapsto \left.\frac{d}{dt}\right\vert_{t = 0} (w + vt)
 \end{gather*}
 for $p\in M$ and $w\in E_p = \pi_E^{-1}(p)$.
} (cf. \cite{Eliasson1967} for the construction of $K$). The data $HE$ and $K$ are equivalent, since, given the horizontal subbundle $HE$, we can always define the vertical projection $K$ and, given the map $K$, we can always define the horizontal bundle $HE$ to be the (fiberwise) kernel of $K$. Likewise, the data $\nabla$ and $K$ are equivalent. Their relation is encoded in the following formula:
\begin{gather*}
 \nabla_X Y = K\circ dY (X)\quad\forall X\in TM,
\end{gather*}
where the section $Y\in\Gamma (E)$ is viewed as a smooth map $Y:M\to E$.\\
Now observe that the map $f^\nabla \coloneqq (d\pi_E, K):TE\to TM\oplus E$ is a bundle map over\linebreak $\pi_E:E\to M$, i.e., the diagram
\begin{center}
 \begin{tikzcd}
  TE \arrow[r, "f^\nabla"] \arrow[d]
  & TM\oplus E \arrow[d] \\
  E \arrow[r, "\pi_E"]
  & M
 \end{tikzcd}
\end{center}
commutes, where the vertical arrows are the base point projections of the vector bundles $TE$ and $TM\oplus E$. Fiberwise, the bundle map $f^\nabla$ is a linear isomorphism. Thus, $TE$ is isomorphic to the pullback bundle $\pi_E^\ast (TM\oplus E)$.\\
Now take $E\to M$ to be the tangent bundle $TM\to M$ of $M$ with base point projection $\pi\equiv \pi_{TM}:TM\to M$. Then, $f^\nabla$ allows us to define the almost complex structure $J_\nabla$ on the manifold $TM$ by ``pulling back'' the almost complex structure\linebreak $J_{TM\oplus TM}:TM\oplus TM \to TM\oplus TM, (w_1, w_2)\mapsto (w_2, -w_1)$ of the vector bundle\linebreak $TM\oplus TM\to M$ to the vector bundle $T(TM)\to TM$. Explicitly speaking, the almost complex structure $J_\nabla$ is completely determined by the following equations:
\begin{gather*}
 d\pi\circ J_\nabla = K;\quad K\circ J_\nabla = -d\pi.
\end{gather*}
Next, we wish to express $J_\nabla$ in local coordinates. For this, we need to parametrize the vertical and horizontal subspaces first. Choose a point $p\in M$ and normal coordinates $\psi = (x_1,\ldots, x_n):U\to V\subset\R^n$ of $(M,\nabla)$ near $p$, i.e., a chart $\psi$ in which all lines through the origin $\psi (p) = 0$ are geodesics. We denote by $T\psi = (\hat x_1,\ldots, \hat x_n, v_1,\ldots, v_n)$ the coordinates of $TM$ near any point $w\in TM$ with $\pi (w) = p$, which are defined by:
\begin{gather*}
 (T\psi)^{-1} (\hat x_1,\ldots, \hat x_n, v_1,\ldots, v_n)\coloneqq \sum^n_{k = 1}v_k\partial_{x_k}\vert_{\psi^{-1}(\hat x_1,\ldots, \hat x_n)}.
\end{gather*}
We then find:
\begin{gather*}
 d\pi (\partial_{v_k}) = \left.\frac{d}{dt}\right\vert_{t = 0} \pi\left(\sum^n_{l\neq k}v_l\partial_{x_l} + (v_k + t)\partial_{x_k}\right) = \left.\frac{d}{dt}\right\vert_{t = 0} \pi\left(\sum^n_{l = 1}v_l\partial_{x_l}\right) = 0.
\end{gather*}
Thus, the vector fields $\partial_{v_1},\ldots, \partial_{v_n}$ span the vertical subspaces.\\
To parametrize the horizontal subspaces, we consider the local vector field $X_c\coloneqq \sum_k c_k\partial_{x_k}$ with constants $c_k\in\R$. We find:
\begin{gather*}
 \nabla_{\partial_{x_k}} X_c (p) = K\circ dX_c\vert_p (\partial_{x_k}\vert_p) = K\left(\left.\frac{d}{dt}\right\vert_{t = 0}\left(X_c\circ \psi^{-1}(\psi (p) + t\hat e_k)\right)\right) = K(\partial_{\hat x_k}\vert_{X_c (p)})
\end{gather*}
Hence, the vectors $\partial_{\hat x_1}\vert_w,\ldots, \partial_{\hat x_n}\vert_w$ span the horizontal subspaces $H(TM) = \ker (K)$ for any $w = X_c (p)\in T_pM$ if and only if the equation
\begin{gather*}
 \nabla_{\partial_{x_i}}\partial_{x_j} (p) = 0
\end{gather*}
holds for all $i,j\in\{1,\ldots, n\}$. The last equation is satisfied for all normal coordinates $(x_1,\ldots, x_n)$ near any point $p\in M$ if and only if $\nabla$ is symmetric, i.e., satisfies:
\begin{gather*}
 \nabla_X Y - \nabla_Y X = [X,Y]\quad \forall X,Y\in\Gamma (TM).
\end{gather*}
Thus, we shall henceforth assume that the connection $\nabla$ is symmetric.\\
In total, we have found that the vertical subspaces at $w\in T_pM$ are spanned by the vectors $\partial_{v_1}\vert_w,\ldots, \partial_{v_n}\vert_w$, while the horizontal subspaces at $w\in T_pM$ are spanned by the vectors $\partial_{\hat x_1}\vert_w,\ldots, \partial_{\hat x_n}\vert_w$ for normal coordinates $\psi = (x_1,\ldots, x_n)$ of $(M,\nabla)$ near $p\in M$ with $T\psi = (\hat x_1,\ldots, \hat x_n, v_1,\ldots, v_n)$. If we express an arbitrary vector $u\in T_w (TM)$ and its image $J_\nabla(u)\in T_w (TM)$ as
\begin{gather*}
 u = \sum^n_{k = 1}a_k\partial_{v_k}\vert_w + b_k\partial_{\hat x_k}\vert_w\quad\text{and}\quad J_\nabla(u) = \sum^n_{k = 1}c_k\partial_{v_k}\vert_w + d_k\partial_{\hat x_k}\vert_w,
\end{gather*}
we can compute the coefficients $c_k$ and $d_k$ in terms of $a_k$ and $b_k$:
\begin{alignat*}{2}
 -\sum^n_{k = 1}b_k\partial_{x_k}\vert_p &= -d\pi (u) = K\circ J_\nabla (u) = \sum^n_{k = 1}c_k\partial_{x_k}\vert_p\quad &\Rightarrow c_k = -b_k\\
 \sum^n_{k = 1} a_k\partial_{x_k}\vert_p &= K(u) = d\pi\circ J_\nabla (u) = \sum^n_{k =1}d_k\partial_{x_k}\vert_p\quad &\Rightarrow d_k = a_k,
\end{alignat*}
where we used $V(TM) = \ker (d\pi)$, $H(TM) = \ker (K)$, and $d\pi (\partial_{\hat x_k}\vert_w) = \partial_{x_k}\vert_p = K (\partial_{v_k}\vert_w)$. This gives us:
\begin{gather*}
 J_\nabla (\partial_{v_k}\vert_w) = \partial_{\hat x_k}\vert_w;\quad J_\nabla (\partial_{\hat x_k}\vert_w) = -\partial_{v_k}\vert_w.
\end{gather*}
We now see that the almost complex structure $J_\nabla$ assumes the standard form in normal coordinates near $p\in M$ for points $w\in T_pM$. This does not mean, however, that $J_\nabla$ is integrable, since the last equation is not necessarily true for all points $w$ within the chart domain $TU$. This is the case if $\nabla_{\partial_{x_i}}\partial_{x_j} \equiv 0$, i.e., if $\nabla$ is flat\footnote{If, additionally, $\nabla = \nabla^g$ is the Levi-Civita connection of some semi-Riemannian metric $g$, this becomes an ``if and only if''-statement, i.e., $J_{\nabla^g}$ is integrable if and only if $\nabla^g$ is flat if and only if $g$ is flat (this statement was proven for Riemannian metrics by Michel Grueneberg \cite{Grueneberg2001}, but the same proof also works for semi-Riemannian metrics).}.\\
Next, we want to transfer the almost complex structure $J_\nabla$ from $TM$ to $T^\ast M$. In general, we can pick any bundle isomorphism $TM\to T^\ast M$, which is then also a diffeomorphism between the manifolds $TM$ and $T^\ast M$, and translate $J_\nabla$ using this diffeomorphism. However, there is no canonical choice of bundle isomorphism for generic manifolds $M$ with connection $\nabla$. The situation is different if $M$ is equipped with a semi-Riemannian metric $g$ and $\nabla$ is the Levi-Civita connection $\nabla^g$. In this case, we can choose $G:TM\to T^\ast M$, $v\mapsto \iota_v g$ as our bundle isomorphism and define the almost complex structure $J^\ast_{\nabla^g}$ on $T^\ast M$ via
\begin{gather*}
 J^\ast_{\nabla^g}\coloneqq dG\circ J_{\nabla^g}\circ dG^{-1}.
\end{gather*}
Again, we aim to express $J^\ast_{\nabla^g}$ in normal coordinates $\psi = (x_1,\ldots, x_n)$ of $(M, g)$ near $p\in M$. Note that the chosen normal coordinates satisfy
\begin{gather*}
 g(p) = \sum^s_{k = 1} dx_k\vert^2_p - \sum^{n = s+t}_{k = s+1} dx_k\vert^2_p\quad\text{and}\quad \partial_{x_k} g_{lm} (p) = 0,
\end{gather*}
where $(s,t)$ is the signature of $g$. As before, we introduce the notation\linebreak $T\psi = (\hat x_1,\ldots, \hat x_n, v_1,\ldots, v_n)$ for the induced coordinates on $TM$. Similarly, we employ the notation $T^\ast \psi = (q_1,\ldots, q_n, p_1,\ldots, p_n)$ for the induced coordinates on $T^\ast M$:
\begin{gather*}
 (T^\ast\psi)^{-1} (q_1,\ldots, q_n, p_1,\ldots, p_n) \coloneqq \sum^n_{k = 1} p_k dx_k\vert_{\psi^{-1}(q_1,\ldots, q_n)}.
\end{gather*}
In these coordinates, $G$ is given by:
\begin{align*}
 T^\ast\psi\circ G\circ (T\psi)^{-1} (0,\ldots, 0, v_1,\ldots, v_n) = (0,\ldots, 0, v_1,\ldots, v_s, -v_{s+1},\ldots, -v_n).
\end{align*}
Together with $\partial_{x_k} g_{lm} (p) = 0$, this implies:
\begin{gather*}
 dG\vert_w (\partial_{\hat x_k}\vert_w) = \partial_{q_k}\vert_{G(w)},\quad dG\vert_w (\partial_{v_k}\vert_w) = \begin{cases}\partial_{p_k}\vert_{G(w)}\text{ for }1\leq k\leq s\\ -\partial_{p_k}\vert_{G(w)}\text{ for } k>s\end{cases}\quad \forall w\in T_pM.
\end{gather*}
This allows us to compute $J^\ast_{\nabla^g}$ in coordinates for points $\alpha\in T^\ast_p M$:
\begin{alignat*}{3}
 &J^\ast_{\nabla^g} (\partial_{q_k}\vert_\alpha) = -\partial_{p_k}\vert_\alpha;\quad &&J^\ast_{\nabla^g} (\partial_{p_k}\vert_\alpha) = \phantom{-} \partial_{q_k}\vert_\alpha\quad &&(1\leq k\leq s),\\
 &J^\ast_{\nabla^g} (\partial_{q_k}\vert_\alpha) = \phantom{-} \partial_{p_k}\vert_\alpha;\quad &&J^\ast_{\nabla^g} (\partial_{p_k}\vert_\alpha) = -\partial_{q_k}\vert_\alpha\quad &&(k>s).
\end{alignat*}
Again, this does not imply that $J^\ast_{\nabla^g}$ is integrable, as the equations above are only true for points $\alpha\in T^\ast_p M$ and not necessarily the entire chart domain $T^\ast U$. However, $J^\ast_{\nabla^g}$ is integrable if and only if $\nabla^g$ is flat, i.e., if and only if $g$ is flat (cf. \cite{Grueneberg2001}).\\
The reason why we are interested in the almost complex structure $J^\ast_{\nabla^g}$ is the curious fact that $J^\ast_{\nabla^g}$ is naturally compatible with the canonical symplectic form $\omega_\can$ on $T^\ast M$, as one easily checks: In the coordinates $(q_1,\ldots, q_n, p_1,\ldots, p_n)$ of $T^\ast M$ from above, $\omega_\can$ is given by:
\begin{gather*}
 \omega_\can\vert_{T^\ast U} = \sum^n_{k = 1} dp_k\wedge dq_k.
\end{gather*}
Thus, $\omega_\can (\cdot, J^\ast_{\nabla^g}\cdot)$ is a semi-Riemannian metric on $T^\ast M$ of signature $(2s, 2t)$:
\begin{gather*}
 \omega_\can\vert_\alpha (\cdot, J^\ast_{\nabla^g}\cdot) = \sum^s_{k = 1} dq_k\vert^2_\alpha + dp_k\vert^2_\alpha - \sum^{n}_{k = s+1} dq_k\vert^2_\alpha + dp_k\vert^2_\alpha\quad\forall \alpha\in T_pM.
\end{gather*}
As we can find normal coordinates near any $p\in M$, we have proven the following theorem:
\begin{Thm}[Almost complex structures on $TM$ and $T^\ast M$]
 Let $M$ be a smooth,\linebreak $n$-dimensional manifold together with a connection $\nabla$ on it. Then, there exists a unique almost complex structure $J_\nabla$ on the tangent bundle $TM$ such that
 \begin{gather*}
  d\pi\circ J_\nabla = K;\quad K\circ J_\nabla = -d\pi,
 \end{gather*}
 where $\pi:TM\to M$ is the base point projection of $TM$ and $K:T(TM)\to TM$ is the vertical projection corresponding to $\nabla$. If $\nabla$ is symmetric, then $J_\nabla$ can be expressed as
 \begin{gather*}
  J_\nabla (\partial_{v_k}\vert_w) = \partial_{\hat x_k}\vert_w;\quad J_\nabla (\partial_{\hat x_k}\vert_w) = -\partial_{v_k}\vert_w,
 \end{gather*}
 where $w\in TM$ is a point, $\psi = (x_1,\ldots, x_n)$ are normal coordinates of $(M,\nabla)$ near $p = \pi (w)$, and $T\psi = (\hat x_1,\ldots, \hat x_n, v_1,\ldots, v_n)$ are the induced coordinates on $TM$. If, additionally, $\nabla$ is flat, then $J_{\nabla}$ is integrable.\\
 If $\nabla = \nabla^g$ is the Levi-Civita connection of a semi-Riemannian metric $g$ on $M$ of signature $(s,t)$, then there exists a (canonical) almost complex structure $J^\ast_{\nabla^g}$ on $T^\ast M$ such that $\omega_\can (\cdot, J^\ast_{\nabla^g}\cdot)$ is a semi-Riemannian metric on $T^\ast M$ of signature $(2s, 2t)$, where $\omega_\can$ is the canonical symplectic form on $T^\ast M$. In coordinates, $J^\ast_{\nabla^g}$ is given by:
 \begin{alignat*}{3}
  &J^\ast_{\nabla^g} (\partial_{q_k}\vert_\alpha) = -\partial_{p_k}\vert_\alpha;\quad &&J^\ast_{\nabla^g} (\partial_{p_k}\vert_\alpha) = \phantom{-} \partial_{q_k}\vert_\alpha\quad &&(1\leq k\leq s),\\
  &J^\ast_{\nabla^g} (\partial_{q_k}\vert_\alpha) = \phantom{-} \partial_{p_k}\vert_\alpha;\quad &&J^\ast_{\nabla^g} (\partial_{p_k}\vert_\alpha) = -\partial_{q_k}\vert_\alpha\quad &&(k>s),
 \end{alignat*}
 where $\alpha\in T^\ast_p M$ is a point, $\psi = (x_1,\ldots, x_n)$ are normal coordinates of $(M,g)$ near $p\in M$, and $T^\ast \psi = (q_1,\ldots, q_n, p_1,\ldots, p_n)$ are the induced coordinates of $T^\ast M$. Furthermore, $J^\ast_{\nabla^g}$ is integrable if and only if $g$ is flat.
\end{Thm}

 \newpage
 \section{Holomorphic Levi-Civita Connection}
\label{app:holo_connection}

In this part, we introduce the notion of a (linear/affine) holomorphic connection on a complex manifold $(X,J)$ (cf. the end of Section 4.2 in \cite{huybrechts2005}). In particular, we define the holomorphic Levi-Civita connection $\nabla^h$ induced by a holomorphic metric $h = h_R + ih_I$ on $(X,J)$. We show that the standard Levi-Civita connections $\nabla^{h_R}$ and $\nabla^{h_I}$ of the real and imaginary part $h_R$ and $h_I$ agree with each other, that the holomorphic Levi-Civita connection $\nabla^h$ is just the complexification of $\nabla^{h_R} = \nabla^{h_I}$, and that holomorphic normal coordinates of $h$ are also normal coordinates of $h_R$.\\
We start by defining a holomorphic connection $\nabla$ on a complex manifold $(X,J)$ (cf. \cite{huybrechts2005}):

\begin{Def}[Holomorphic connection]\label{def:holo_connection}
 Let $X$ be a complex manifold with complex structure $J$ and let $\Gamma (T^{(1,0)}U)$ be the complex vector space of holomorphic vector fields on any open subset $U\subset X$. We call $\nabla$ a \textbf{holomorphic connection} on $(X,J)$ iff $\nabla$ is a collection
 \begin{gather*}
  \{\nab{U}:\hvec{U}\times\hvec{U}\to\hvec{U}\mid U\subset X\text{\normalfont\ open}\}
 \end{gather*}
 of $\Cx$-bilinear maps satisfying for all open subsets $U^\prime\subset U\subset X$, all holomorphic functions $f:U\to\Cx$, and all holomorphic vector fields $X,Y\in\hvec{U}$: 
 \begin{enumerate}[label = (\arabic*)]
  \item Tensorial in first component:
  \begin{gather*}
   \nab{U}_{fX} Y = f\nab{U}_X Y.
  \end{gather*}
  \item Leibniz rule in second component:
  \begin{gather*}
   \nab{U}_X fY = X(f) Y + f\nab{U}_X Y.
  \end{gather*}
  \item Presheaf property:
  \begin{gather*}
   (\nab{U}_X Y)\vert_{U^\prime} = \nab{U^\prime}_{X\vert_{U^\prime}} Y\vert_{U^\prime}.
  \end{gather*}
 \end{enumerate}
\end{Def}

\begin{Rem}[Presheaf property]\label{rem:presheaf}
 The standard definition of a (smooth) connection $\nabla$ on a (smooth) manifold $M$\footnote{To be precise, a linear/affine connection $\nabla$ on the vector bundle $TM\to M$.} is formulated globally and usually does not include the presheaf property, since the existence of (smooth) partitions of unity together with Property (1) and (2) implies the presheaf property in those cases. However, there are no holomorphic partitions of unity for complex manifolds, which is why we need to enforce Property (3) manually. Otherwise, the zero map would be a holomorphic connection on the complex torus $\Cx^n/\Gamma$ ($\Gamma$: lattice), since the (global) Leibniz rule reduces to linearity in this case.
\end{Rem}

As in the real case, a holomorphic connection can be computed locally:

\begin{Prop}[$\nabla$ is local]
 Let $(X,J)$ be a complex manifold with holomorphic connection $\nabla$. For all points $p\in X$, all open neighborhoods $U^\prime\subset U\subset X$ of $p$, and all holomorphic vector fields $X,Y\in \hvec{U}$ on $U$, the expression $\nab{U}_X Y (p)$ is completely determined by $X(p)$ and $Y\vert_{U^\prime}$.
\end{Prop}

\begin{proof}
 Take the notations from above, then:
 \begin{gather*}
  \nab{U}_X Y (p) = (\nab{U}_X Y)\vert_{U^\prime} (p) \stackrel{(3)}{=} \nab{U^\prime}_{X\vert_{U^\prime}} Y\vert_{U^\prime} (p)
 \end{gather*}
 Thus, $\nab{U}_X Y(p)$ only depends on $X\vert_{U^\prime}$ and $Y\vert_{U^\prime}$. By going into holomorphic charts of $X$ near $p$ and using Property (1), we see that $\nab{U}_X Y (p)$ only depends on $X(p)$ and $Y\vert_{U^\prime}$.
\end{proof}

Next, we want to define the holomorphic Levi-Civita connection $\nabla^h$ of a holomorphic metric $h$ on $(X,J)$. As in the real case, $\nabla^h$ should be the unique holomorphic connection on $X$ which is torsion-free and compatible with $h$. Thus, we first need to define the torsion $T$ of a holomorphic connection $\nabla$:

\begin{Def}[Torsion $T$]
 Let $(X,J)$ be a complex manifold with holomorphic connection $\nabla$. The \textbf{torsion}\footnote{Sometimes also called torsion tensor or torsion (tensor) field} $T$ of $\nabla$ is the collection of maps
 \begin{gather*}
  \{\tor{U}:\hvec{U}\times\hvec{U}\to\hvec{U}\mid U\subset X\text{\normalfont\ open}\}
 \end{gather*}
 defined by
 \begin{gather*}
  \tor{U}(X,Y)\coloneqq \nab{U}_X Y - \nab{U}_Y X - [X,Y]\quad\forall X,Y\in\hvec{U}\ \forall U\subset X\text{\normalfont\ open}.
 \end{gather*}
 $\nabla$ is said to be \textbf{torsion-free} iff $\tor{U}\equiv 0$ for all open subsets $U\subset X$.
\end{Def}

\begin{Rem*}[$T$ is a tensor]
 $T$ is bitensorial and satisfies the presheaf property.
\end{Rem*}

We also need to say what it means for a holomorphic connection $\nabla$ to be compatible with a holomorphic metric $h$:

\begin{Def}[Metric compatibility]
 Let $(X,J)$ be a complex manifold with holomorphic connection $\nabla$. Furthermore, let $h$ be a holomorphic metric on $X$. The \textbf{metric compatibility tensor} $\nabla h$ of $\nabla$ and $h$ is the collection of maps
 \begin{gather*}
  \{\nab{U}h:\hvec{U}^3\to\mathcal{O}_U\coloneqq\{f:U\to\Cx\mid f\text{\normalfont\ holomorphic}\}\mid U\subset X\text{\normalfont\ open}\}
 \end{gather*}
 defined by
 \begin{gather*}
  \nab{U}h(X,Y,Z)\coloneqq X(h\vert_U (Y,Z)) - h\vert_U (\nab{U}_X Y, Z) - h\vert_U (Y, \nab{U}_X Z)
 \end{gather*}
 for all holomorphic vector fields $X,Y,Z\in\hvec{U}$ and all open subsets $U\subset X$. We say that $\nabla$ is \textbf{compatible} with $h$ iff $\nab{U}h\equiv 0$ for every open subset $U\subset X$.
\end{Def}

\begin{Rem*}[$\nabla h$ is a tensor.]
 $\nabla h$ is tritensorial and satisfies the presheaf property.
\end{Rem*}

We are now ready to define the holomorphic Levi-Civita connection $\nabla^h$:

\begin{Lem}[Levi-Civita connection $\nabla^h$]\label{lem:holo_levi-civita}
 Let $(X,J)$ be a complex manifold together with a holomorphic metric $h$. Then, there exists exactly one holomorphic connection $\nabla^h$ on $X$, the holomorphic \textbf{Levi-Civita connection}, which is torsion-free and compatible with $h$.
\end{Lem}

\begin{proof}
 Take the notations from above. The proof works very similarly to the real case by exploiting the Koszul formula.\\\\
 \noindent \textbf{Uniqueness:} Let $\nabla^{h,1}$ and $\nabla^{h,2}$ be two holomorphic connections which are torsion-free and compatible with $h$. As in the real case, one can show that $\nabla^{h,1}$ and $\nabla^{h,2}$ satisfy the Koszul formula ($U\subset X$ open; $X,Y,Z\in\hvec{U}$):
 \begin{align*}
  2h\vert_U (\nab{U}^{h,1}_X Y, Z) &= X(h\vert_U (Y,Z)) + Y(h\vert_U (Z,X)) - Z (h\vert_U (X,Y))\\
  &+ h\vert_U ([X,Y], Z) - h\vert_U ([Y,Z], X) - h\vert_U([X,Z],Y)\\
  &= 2h\vert_U (\nab{U}^{h,2} Y, Z).
 \end{align*}
 Now let $U\subset X$ be any open subset and $X,Y\in\hvec{U}$ be two holomorphic vector fields on $U$. Pick any point $p\in U$ and a holomorphic chart $(z_1,\ldots, z_n):U^\prime\to V^\prime\subset\Cx^n$ of $X$ near $p$ such that $U^\prime\subset U$. Due to the presheaf property of $\nabla^{h,i}$, we have ($i\in\{1,2\}$):
 \begin{gather*}
  \nab{U}^{h,i}_X Y (p) = \nab{U^\prime}^{h,i}_{X\vert_{U^\prime}} Y\vert_{U^\prime} (p).
 \end{gather*}
 We now set $Z\vert_{U^\prime} = \partial_{z_j}$ ($j\in\{1,\ldots, n\}$) in the formula above (evaluated at $p$) and obtain using the previous equation:
 \begin{gather*}
  h\vert_{U^\prime} (\nab{U^\prime}^{h,1}_{X\vert_{U^\prime}} Y\vert_{U^\prime}, \partial_{z_j}) (p) = h\vert_{U^\prime} (\nab{U^\prime}^{h,2}_{X\vert_{U^\prime}} Y\vert_{U^\prime}, \partial_{z_j}) (p).
 \end{gather*}
 Since the last equation holds for every $j\in\{1,\ldots, n\}$, we can combine the previous two equations to find:
 \begin{gather*}
  \nab{U}^{h,1}_X Y (p) = \nab{U^\prime}^{h,1}_{X\vert_{U^\prime}} Y\vert_{U^\prime} (p) = \nab{U^\prime}^{h,2}_{X\vert_{U^\prime}} Y\vert_{U^\prime} (p) = \nab{U}^{h,2}_X Y (p).
 \end{gather*}
 As this argument can be repeated for every $p\in U$, we find:
 \begin{gather*}
  \nab{U}^{h,1}_X Y = \nab{U}^{h,2}_X Y.
 \end{gather*}
 Again, we can repeat the argument for all $X,Y\in\hvec{U}$ and any open subset $U\subset X$ implying $\nabla^{h,1} = \nabla^{h,2}$.\\\\
 \noindent \textbf{Existence:} The Koszul formula gives us an expression for
 \begin{gather*}
  h\vert_U (\nab{U}^{h}_X Y, Z).
 \end{gather*}
 A priori, it is not cleared whether this expression completely determines $\nab{U}^h_X Y$, as $U$ might be ``too large'' to admit enough linearly independent, holomorphic vector fields $Z$. However, we can always shrink $U$ by the presheaf property. Especially, if we shrink $U$ to be a chart domain, then $U$ admits enough holomorphic vector fields. Thus, we can define $\nab{U}^h_X Y (p)$ for any point $p\in U$ via the Koszul formula in a small chart near $p$. Since the Koszul formula is independent of the choice of charts, the resulting holomorphic connection $\nabla^h$ is well-defined. One easily checks that $\nabla^h$ defined this way is torsion-free and compatible with $h$.
\end{proof}

Next, we want to consider normal coordinates of a holomorphic connection. To do so, we need to define geodesics first:

\begin{Def}[Complex derivative $\nabla/dz$ along $\gamma$ and geodesics]
 Let $(X,J)$ be a complex manifold with holomorphic connection $\nabla$. Further, let $D\subset\Cx$ be an open and connected subset, $\gamma:D\to X$ be a holomorphic curve, and $Y$ be a holomorphic vector field along $\gamma$, i.e., a holomorphic section of the pullback bundle $\gamma^\ast T^{(1,0)}X$. Analogously to the real case, we can define the \textbf{complex derivative} $\nabla/dz$ along $\gamma$. In holomorphic charts $\phi = (z_1,\ldots, z_n):U^\prime\to V^\prime$ of $X$, $\nabla/dz$ is defined by:
 \begin{gather*}
  \left(\frac{\nabla}{dz} Y\right)_j (z)\coloneqq Y_j^\prime (z) + \sum^n_{k,l = 1}\gamma_k^\prime (z)\cdot Y_l (z)\cdot \Gamma^j_{kl} (\gamma (z)),
 \end{gather*}
 where ``$\prime$'' denotes the usual complex derivative and $\nab{U^\prime}_{\partial_{z_k}} \partial_{z_l}\eqqcolon \sum_{j} \Gamma^j_{kl}\cdot \partial_{z_j}$ are the holomorphic \textbf{Christoffel symbols} of $\nabla$. We call $\gamma$ a \textbf{geodesic} of $\nabla$ iff $\nabla/dz\ \gamma^\prime \equiv 0$.
\end{Def}

The following existence and uniqueness result regarding geodesics of $\nabla$ is reminiscent of the real case:

\begin{Prop}[Existence and uniqueness of geodesics]
 Let $(X,J)$ be a complex manifold with holomorphic connection $\nabla$. Further, let $p\in X$ and $v\in T^{(1,0)}_p X$. Then, for every $z_0\in\Cx$, there exists an open and connected neighborhood $D\subset\Cx$ of $z_0$ and a geodesic $\gamma:D\to X$ of $\nabla$ such that $\gamma (z_0) = p$ and $\gamma^\prime (z_0) = v$. Moreover, if $\gamma_1, \gamma_2:D\to X$ are two geodesics of $\nabla$ with $\gamma_1 (z_0) = \gamma_2 (z_0) = p$ and $\gamma^\prime_1 (z_0) = \gamma^\prime_2 (z_0) = v$ for any open and connected neighborhood $D\subset\Cx$ of $z_0$, then $\gamma_1 \equiv \gamma_2$.
\end{Prop}

\begin{proof}
 We observe that the geodesic equation is locally a second order complex differential equation. By introducing new variables $v_i\coloneqq \gamma^\prime_i$, thus, doubling the number of equations, we can rewrite the second order differential equation into a first order one. Since the geodesic equation has no explicit time-dependence ($z$-dependence), we can interpret the first order differential equation obtained this way as the integral curve equation of a holomorphic vector field. The existence and uniqueness results for geodesics now follow from the existence and uniqueness results for integral curves of holomorphic vector fields (cf. Proposition \autoref{prop:holo_traj}).\\
 There is also an alternative way to prove uniqueness: Write the geodesic equation in holomorphic charts near $p$ and expand the coordinates of $\gamma_1$, $\gamma_2$ in a power series around $z_0$. This way, the geodesic equation becomes a recursive formula for the coefficients of the power series. Furthermore, we realize that this recursive formula completely determines all coefficients if we fix the first two terms in each power series. Hence, fixing $\gamma_1 (z_0) = \gamma_2 (z_0) = p$ and $\gamma^\prime_1 (z_0) = \gamma^\prime_2 (z_0) = v$ uniquely determines the coefficients of the power series. The rest now follows from the identity theorem.
\end{proof}

\begin{Rem}[Geodesics depend holomorphically on initial values]
 Proposition \autoref{prop:holo_traj} also shows that a geodesic $\gamma$ of a holomorphic connection $\nabla$ with $\gamma (z_0) = p$ and $\gamma^\prime (z_0) = v$ depends holomorphically on $z_0$, $v$, and $p$.
\end{Rem}

We can now use the geodesics of $\nabla$ to define the exponential map of $\nabla$:

\begin{Def}[Exponential map]
 Let $(X,J)$ be a complex manifold with holomorphic connection $\nabla$  and $p\in X$ be any point. The \textbf{exponential map} of $\nabla$, $\exp_p:V\to X$ with a suitable open neighborhood $V\subset T^{(1,0)}_p X$ of $0$, is defined by:
 \begin{gather*}
  \exp_p (v) \coloneqq \gamma^v_p (1),
 \end{gather*}
 where $\gamma^v_p$ is a geodesic of $\nabla$ satisfying $\gamma^v_p (0) = p$ and $\gamma^{v\prime}_p (0) = v$. Here, $V\subset T^{(1,0)}_p X$ is chosen small enough such that $\gamma^v_p (1)$ is uniquely defined and unaffected by monodromy effects (cf. \autoref{subsec:holo_traj}). For instance, choose $V\subset T^{(1,0)}_pX$ such that the domain of the geodesic $\gamma^v_p:D\to X$ can be chosen to be $D = \{z\in\Cx\mid |z|<2\}$ for every $v\in V$.
\end{Def}

\begin{Rem}[Exponential map is holomorphic]
 By the previous remark, the exponential map $\exp_p$ of a holomorphic connection $\nabla$ is holomorphic.
\end{Rem}

We can now shrink the domain of the exponential map to obtain a biholomorphism:

\begin{Prop}
 Let $(X,J)$ be a complex manifold with holomorphic connection $\nabla$  and let $p\in X$ be any point. Then, there exists an open neighborhood $V\subset T^{(1,0)}_p X$ of $0$ such that $\exp_p: V\to \exp_p (V)$ is a biholomorphism.
\end{Prop}

\begin{proof}
 Apply the holomorphic inverse function theorem to $d\exp_p\vert_0 = \text{id}_{T^{1,0}_p X}$.
\end{proof}

Lastly, we use this biholomorphism to define normal coordinates:

\begin{Def}[Normal coordinates]
 Let $(X,J)$ be a complex manifold with holomorphic connection $\nabla$  and let $p\in X$ be any point. Further, choose a $\Cx$-linear isomorphism $l:\Cx^n\to T^{(1,0)}_p X$. We call the coordinates $(z_1,\ldots, z_n)$ of the holomorphic chart $l^{-1}\circ \exp_p^{-1}$ \textbf{holomorphic normal coordinates} of $(X,\nabla)$ near $p$. If $\nabla = \nabla^h$ is the Levi-Civita connection of a holomorphic metric $h$ on $X$, we additionally require that the vectors $v_1\coloneqq l (\hat e_1),\ldots, v_n\coloneqq l (\hat e_n)$ ($\hat e_1,\ldots, \hat e_n$: standard basis of $\Cx^n$) satisfy:
 \begin{gather*}
  h (v_i, v_j) = \delta_{ij}.
 \end{gather*}
\end{Def}

\begin{Rem*}
 Note that in the complex case all non-degenerate, symmetric bilinear forms are isomorphic, while in the real case two non-degenerate, symmetric bilinear forms are isomorphic if and only if they have the same signature.
\end{Rem*}

Holomorphic normal coordinates satisfy properties similar to their real counterparts:

\newpage

\begin{Prop}[Properties of holomorphic normal coordinates]
 Let $(X,J)$ be a complex manifold with holomorphic metric $h$ and corresponding Levi-Civita connection $\nabla^h$. Further, let $p\in X$ be any point and $\phi = (z_1,\ldots, z_n):U\to V$ be holomorphic normal coordinates of $(X,\nabla^h)$ near $p$. Then, we have in coordinates $(z_1,\ldots, z_n)$:
 \begin{enumerate}[label = (\arabic*)]
  \item $h_{ij} (p) = \delta_{ij}$
  \item $\partial_{z_k}h_{ij} (p) = 0$
  \item $\Gamma^k_{ij} (p) = 0$
 \end{enumerate}
\end{Prop}

\begin{proof}
 We only show (3), as (1) is obvious and (2) follows from (3) by exploiting the vanishing of the metric compatibility tensor. Define for $x = (x_1,\ldots, x_n)\in\Cx^n$ the curve $\gamma (z)\coloneqq \phi^{-1}(z\cdot x)$. By definition of the holomorphic normal coordinates, $\gamma$ is a geodesic of $\nabla^h$ with $\gamma^\prime (0) = d\phi^{-1}\vert_0 (x)$. In coordinates $(z_1,\ldots, z_n)$, the geodesic equation reads:
 \begin{gather*}
  \gamma^{\prime\prime}_{j}(z) + \sum^n_{k,l = 1} \Gamma^j_{kl} (\gamma (z))\gamma^\prime_k (z)\gamma^\prime_l (z) = 0
 \end{gather*}
 For $z = 0$, this gives:
 \begin{gather*}
  \sum^n_{k,l = 1} \Gamma^j_{kl} (p)x_kx_l = 0
 \end{gather*}
 The last equation is true for any $x\in\Cx^n$, thus, it enforces $\Gamma^j_{kl} (p) + \Gamma^j_{lk}(p) = 0$. Since $\nabla^h$ is symmetric, $\Gamma^j_{kl} (p)$ is symmetric in $k$ and $l$. This implies $2\Gamma^j_{kl} (p) = 0$ concluding the proof.
\end{proof}

Our next goal is to relate the holomorphic Levi-Civita connection $\nabla^{h}$ to the standard Levi-Civita connections $\nabla^{h_R}$ and $\nabla^{h_I}$. In particular, we aim to relate the holomorphic normal coordinates of $(X,\nabla^h)$ to the standard normal coordinates of $(X,\nabla^{h_R})$ and $(X,\nabla^{h_I})$ for holomorphic metrics $h = h_R + ih_I$. To achieve that, we first prove the following proposition:

\begin{Prop}\label{prop:h_R_in_coordinates}
 Let $(X,J)$ be a complex manifold with holomorphic metric $h = h_R + ih_I$ and corresponding Levi-Civita connections $\nabla^h$, $\nabla^{h_R}$, and $\nabla^{h_I}$. Further, let $p\in X$ be any point and $\phi = (z_1 = x_1 + iy_1,\ldots, z_n = x_n + iy_n):U\to V$ be holomorphic normal coordinates of $(X,\nabla^h)$ near $p$. Then:
 \begin{enumerate}[label = (\arabic*)]
  \item $h_R\vert_p (\partial_{x_i}\vert_p, \partial_{x_j}\vert_p) = -h_R\vert_p (\partial_{y_i}\vert_p, \partial_{y_j}\vert_p) = \delta_{ij},\quad h_I\vert_p (\partial_{x_i}\vert_p, \partial_{y_j}\vert_p) = \delta_{ij}$\\
  $h_R\vert_p (\partial_{x_i}\vert_p, \partial_{y_j}\vert_p) = h_I\vert_p (\partial_{x_i}\vert_p, \partial_{x_j}\vert_p)\ \ = h_I\vert_p (\partial_{y_i}\vert_p, \partial_{y_j}\vert_p)\qquad\ \ = 0$
  \item All derivatives of $h_R$ and $h_I$ vanish at $p$ in coordinates $(x_1,\ldots, x_n, y_1,\ldots, y_n)$.
  \item All Christoffel symbols of $\nabla^{h_R}$ and $\nabla^{h_I}$ vanish at $p$ in coordinates\linebreak $(x_1,\ldots, x_n, y_1,\ldots, y_n)$.
 \end{enumerate}
\end{Prop}

\begin{proof}
 (1) follows from (1) of the previous proposition:
 \begin{gather*}
  h_R\vert_p + i h_I\vert_p = h\vert_p = \sum^n_{j = 1} dz^2_j\vert_p = \sum^n_{j = 1}\left(dx^2_j\vert_p - dy^2_j\vert_p\right) + i\sum^n_{j = 1}\left(dx_j\vert_p\otimes dy_j\vert_p + dy_j\vert_p\otimes dx_j\vert_p\right).
 \end{gather*}
 (2) follows from (2) of the previous proposition and the fact that the components $h_{ij}$ are holomorphic, i.e., $\partial_{\bar{z}_k} h_{ij} (p) = 0$.\\
 (3) follows from (2) of the proposition at hand and the Koszul formula for Christoffel symbols of standard Levi-Civita connections.
\end{proof}

The last proposition implies that the Levi-Civita connections $\nabla^{h_R}$ and $\nabla^{h_I}$ describe the same connection.

\begin{Cor}[$\nabla^{h_R} = \nabla^{h_I}$]
 Let $(X,J)$ be a complex manifold with holomorphic metric $h = h_R + ih_I$. Then, the (standard) Levi-Civita connections $\nabla^{h_R}$ and $\nabla^{h_I}$ coincide, i.e., $\nabla^{h_R} = \nabla^{h_I}$.
\end{Cor}

\begin{proof}
 This is an immediate consequence of the previous proposition: For every point $p\in X$, there are coordinates in which the Christoffel symbols of $\nabla^{h_R}$ and $\nabla^{h_I}$ agree at $p$, hence, $\nabla^{h_R} = \nabla^{h_I}$ by definition of the Christoffel symbols.
\end{proof}

We will see later on that the holomorphic Levi-Civita connection $\nabla^h$ can be regarded as the complexification of the connection $\nabla^{h_R} = \nabla^{h_I}$. Before we can make this statement precise, we first need to investigate how the complex structure $J$ interacts with the connections $\nabla^{h}$ and $\nabla^{h_R} = \nabla^{h_I}$. For this, we introduce the following definition:

\begin{Def}[Complex connection]
 Let $X$ be a smooth manifold with almost complex structure $J$. We call a connection $\nabla:\Gamma (TX)\times \Gamma (TX)\to \Gamma (TX)$ on $X$ \textbf{complex} with respect to $J$ iff $\nabla J \equiv 0$.
\end{Def}

\begin{Rem*}[Definition of $\nabla J$]
 Recall that for all $X,Y\in\Gamma (TX)$, one has:
 \begin{gather*}
  (\nabla_X J) (Y) = \nabla_X (J(Y)) - J(\nabla_X Y).
 \end{gather*}
\end{Rem*}

We find that $\nabla^{h_R} = \nabla^{h_I}$ is complex:

\begin{Prop}[$\nabla^{h_R} = \nabla^{h_I}$ is complex]
 Let $(X,J)$ be a complex manifold with holomorphic metric $h = h_R + ih_I$. Then, the connection $\nabla^{h_R} = \nabla^{h_I}$ is complex with respect to $J$.
\end{Prop}

\begin{proof}
 We need to show:
 \begin{gather*}
  (\nabla^{h_R}_X J) (Y) = \nabla^{h_R}_X (J(Y)) - J(\nabla^{h_R}_X Y) = 0\quad\forall X,Y\in\Gamma (TX).
 \end{gather*}
 The object $(\nabla^{h_R}_X J) (Y)$ is tensorial in $X$ and $Y$, hence, it suffices to compute $\nabla^{h_R} J$ at any point $p\in X$ for the tangent vector fields associated to the holomorphic normal coordinates $(z_1 = x_1 + iy_1,\ldots, z_n = x_n + iy_n)$ near $p$. By Proposition \autoref{prop:h_R_in_coordinates}, we find:
 \begin{align*}
  \nabla^{h_R}_{\partial_{x_i}}\left( J (\partial_{x_j})\right) (p) &= \nabla^{h_R}_{\partial_{x_i}} \partial_{y_j} (p) = 0,\quad \nabla^{h_R}_{\partial_{x_i}}\left( J (\partial_{y_j})\right) = \ldots\\
  J(\nabla^{h_R}_{\partial_{x_i}} \partial_{x_j}) (p) &= 0,\quad J(\nabla^{h_R}_{\partial_{x_i}} \partial_{y_j}) (p) =\ldots
 \end{align*}
 Hence, $\nabla^{h_R} J = 0$.
\end{proof}

Before we move on with the investigation of the relation between $\nabla^h$ and $\nabla^{h_R} = \nabla^{h_I}$, let us take a moment to understand what it means for a connection to be complex. To do that, we first realize that every connection $\nabla:\Gamma (TX)\times \Gamma (TX)\to \Gamma (TX)$ on a smooth manifold $X$ can be complexified by $\Cx$-linearity, i.e., can be turned into a connection $\nabla:\Gamma (T_\Cx X)\times \Gamma (T_\Cx X)\to \Gamma (T_\Cx X)$ acting on smooth complex vector fields. Now we can rephrase the property ``complex'' as follows:

\begin{Prop}
 Let $X$ be a smooth manifold with almost complex structure $J$. A connection $\nabla:\Gamma (TX)\times \Gamma (TX)\to \Gamma (TX)$ is complex if and only if its complexification $\nabla:\Gamma (T_\Cx X)\times \Gamma (T_\Cx X)\to \Gamma (T_\Cx X)$ satisfies:
 \begin{gather*}
  \nabla_X Y\in \Gamma_{C^\infty} (T^{(1,0)}X)\quad \forall X\in \Gamma (T_\Cx X)\ \forall Y\in \Gamma_{C^\infty} (T^{(1,0)}X),
 \end{gather*}
 where $\Gamma_{C^\infty} (T^{(1,0)}X)$ is the space of smooth sections of $T^{(1,0)}X$.
\end{Prop}

\begin{proof}
 ``$\Rightarrow$'': Let $X\in \Gamma (T_\Cx X)$ and $Y\in \Gamma_{C^\infty} (T^{(1,0)}X)$, then $Y = 1/2 (\hat Y - iJ (\hat Y))$ for a unique real vector field $\hat Y$ on $X$. Thus, we obtain:
 \begin{gather*}
  \nabla_X Y = \frac{1}{2}\left(\nabla_X\hat Y - i\nabla_X (J(\hat Y))\right) \stackrel{\nabla J = 0}{=} \frac{1}{2}\left(\nabla_X\hat Y - iJ (\nabla_X \hat Y)\right)\in\Gamma_{C^\infty} (T^{(1,0)}X).
 \end{gather*}
 ``$\Leftarrow$'': Let $X\in\Gamma (TX)$ and $\hat Y\in \Gamma (TX)$, then $Y\coloneqq 1/2 (\hat Y - iJ(\hat Y))\in\Gamma_{C^\infty} (T^{(1,0)}X)$. Thus, we find $\nabla_X Y \in\Gamma_{C^\infty} (T^{(1,0)}X)$. Since
 \begin{gather*}
  \text{Re} \left(\nabla_X Y\right) = \frac{1}{2} \nabla_X \hat Y,
 \end{gather*}
 we must have:
 \begin{gather*}
  \text{Im} \left(\nabla_X Y\right) = -\frac{1}{2} J(\nabla_X \hat Y).
 \end{gather*}
 In total, this gives:
 \begin{gather*}
  \frac{1}{2}\nabla_X (\hat Y - i J(\hat Y)) = \frac{1}{2}\left(\nabla_X \hat Y - i J(\nabla_X \hat Y)\right)
 \end{gather*}
 Hence:
 \begin{gather*}
  \nabla_X J(\hat Y) = J(\nabla_X\hat Y)\ \Rightarrow\ \nabla J = 0
 \end{gather*}
\end{proof}

We now see that, at first glance, the definition of a complex connection and a holomorphic connection look very similar. A complex connection $\nabla:\Gamma (TX)\times \Gamma (TX)\to \Gamma (TX)$ can be regarded as a map $\nabla:\Gamma_{C^\infty} (T^{(1,0)}X)\times \Gamma_{C^\infty} (T^{(1,0)}X)\to \Gamma_{C^\infty} (T^{(1,0)}X)$ (via complexification) satisfying Property (1), (2), and (3)\footnote{We can convince ourselves that a complex connection satisfies Property (3) by exploiting smooth partitions of unity (cf. Remark \autoref{rem:presheaf}).} from Definition \autoref{def:holo_connection}, where all vector fields $X$ and $Y$ are taken to be smooth complex vector fields and all functions $f$ are taken to be smooth $\Cx$-valued functions. The main difference between complex and holomorphic connections is that holomorphic connections are only defined for integrable $J$ and send holomorphic vector fields to holomorphic vector fields, while complex connections are defined on general almost complex manifolds $(X,J)$. In particular, if $J$ is non-integrable, there is no notion of holomorphic vector fields. Even if $J$ is integrable, complex connections are not required to send holomorphic vector fields to holomorphic vector fields. In this sense, the notion of a holomorphic connection is more restrictive than the notion of a complex connection.\\
We now have all tools at hand to show that $\nabla^h$ is the complexification of $\nabla^{h_R} = \nabla^{h_I}$:

\begin{Lem}[$\nabla^h = \nabla^{h_R} = \nabla^{h_I}$]
 Let $(X,J)$ be a complex manifold with holomorphic metric $h = h_R + ih_I$. Then, the holomorphic Levi-Civita connection $\nabla^h$ is given by the complexification\footnote{We also denote the complexification of $\nabla^{h_R} = \nabla^{h_I}$ by $\nabla^{h_R} = \nabla^{h_I}$.} of the standard Levi-Civita connections $\nabla^{h_R} = \nabla^{h_I}$, i.e.:
 \begin{gather*}
  \nab{U}^h_X Y = \nabla^{h_R\vert_U}_X Y = \nabla^{h_I\vert_U}_X Y\quad\forall X,Y\in\Gamma (T^{(1,0)}X)\ \forall U\subset X\text{ open}.
 \end{gather*}
\end{Lem}

\begin{proof}
 For smooth real vector fields $X,Y,Z\in\Gamma (TX)$, we can express the terms
 \begin{gather*}
  2h_R \left(\nabla^{h_R}_X Y, Z\right)\quad\text{and}\quad 2h_I \left(\nabla^{h_I}_X Y, Z\right)
 \end{gather*}
 with the help of the Koszul formula as in the proof of Lemma \autoref{lem:holo_levi-civita}. We realize that the Koszul formula also holds for complex smooth vector fields $X,Y,Z\in\Gamma (T_\Cx X)$ if we complexify the expressions above and the Koszul formula by $\Cx$-linearity in $X,Y,Z$. In particular, for holomorphic vector fields $X,Y,Z\in\Gamma (T^{(1,0)}X)$, this gives:
 \begin{align*}
  2h (\nabla^{h_R}_X Y, Z) &= 2h_R (\nabla^{h_R}_X Y, Z) + i 2h_I (\nabla^{h_R}_X Y, Z) = 2h_R (\nabla^{h_R}_X Y, Z) + i 2h_I (\nabla^{h_I}_X Y, Z)\\
  &= (\text{Koszul formula for $h_R$}) + i (\text{Koszul formula for $h_I$})\\
  &= \text{Koszul formula for $h$} = 2h (\nabla^h_X Y, Z).
 \end{align*}
 Thus, we have:
 \begin{gather*}
  h(\nabla^{h}_X Y - \nabla^{h_R}_X Y, Z) = 0.
 \end{gather*}
 A priori, the last equation does not enforce $\nabla^h = \nabla^{h_R}$, since the holomorphic metric $h$ is only non-degenerate on $T^{(1,0)}X$, but vanishes on $T^{(0,1)}X$. However, since $\nabla^h$ is holomorphic and $\nabla^{h_R}$ is complex, the vector $\nabla^{h}_X Y - \nabla^{h_R}_X Y$ is of type $(1,0)$ concluding the proof.
\end{proof}

\begin{Rem}[Holomorphic connections as complexifications of real connections]
 Locally, every holomorphic connection $\nabla$ on $X$ can be interpreted as the complexification $\hat\nabla:\Gamma (T_\Cx U)\times \Gamma (T_\Cx U)\to \Gamma (T_\Cx U)$ of a real connection $\hat\nabla:\Gamma (T U)\times \Gamma (T U)\to \Gamma (T U)$ on a chart domain $U\subset X$. For instance, we can define $\hat\nabla$ by\footnote{These equations determine by $\Cx$-linearity the expressions $\hat\nabla_{\pa_{x_i}}\pa_{x_j}$, $\hat\nabla_{\pa_{x_i}}\pa_{y_j}$,\ldots, which in turn determine the Christoffel symbols of $\hat\nabla$. Since $\hat\nabla$ satisfies $\hat\nabla_{\pa_{z_i}}\pa_{z_j} = \overline{\hat\nabla_{\pa_{\bar{z}_i}}\pa_{\bar{z}_j}}$ and $\hat\nabla_{\pa_{\bar{z}_i}}\pa_{z_j} = \overline{\hat\nabla_{\pa_{z_i}}\pa_{\bar{z}_j}}$, the Christoffel symbols of $\hat\nabla$ w.r.t. the chart $(x_1,\ldots, y_n)$ are real implying that $\hat\nabla$ is indeed a real connection on $U$.}:
 \begin{gather*}
  \hat\nabla_{\pa_{z_i}}\pa_{z_j} = \nab{U}_{\pa_{z_i}}\pa_{z_j},\quad \hat\nabla_{\pa_{\bar{z}_i}}\pa_{\bar{z}_j} = \overline{\nab{U}_{\pa_{z_i}}\pa_{z_j}},\quad \hat\nabla_{\pa_{\bar{z}_i}}\pa_{z_j} = \hat\nabla_{\pa_{z_i}}\pa_{\bar{z}_j} = 0.
 \end{gather*}
 However, such a connection $\hat\nabla$ is not unique. In fact, replacing the last equation by
 \begin{gather*}
  \hat\nabla_{\pa_{\bar{z}_i}} \pa_{z_j} = \overline{\hat\nabla_{\pa_{z_i}} \pa_{\bar{z}_j}} = \sum^n_{k=1} f^k_{ij}\cdot \pa_{z_k} + g^k_{ij}\cdot \pa_{\bar{z}_k}
 \end{gather*}
 with $f^k_{ij}, g^k_{ij}\in C^\infty (U,\Cx)$ also yields a suitable connection $\hat\nabla$.
\end{Rem}

Lastly, we want to examine the relation between the holomorphic normal coordinates of $\nabla^h$ and the normal coordinates of $\nabla^{h_R}$. The following formula will be helpful in the upcoming discussion.

\begin{Prop}
 Let $(X,J)$ be a complex manifold with holomorphic metric $h = h_R + i h_I$ and Levi-Civita connection $\nabla^h\equiv \nabla^{h_R}\equiv \nabla^{h_I}$. Then:
 \begin{gather*}
  \nabla^{h_R}_{J(X)} Y = i\nabla^{h_R}_X Y\quad\forall X\in\Gamma (TX)\ \forall Y\in\Gamma (T^{(1,0)}X).
 \end{gather*}
\end{Prop}

\begin{proof}
 Take the notations from above and pick any point $p\in X$. We show:
 \begin{gather*}
  \nabla^{h_R}_{J(X)} Y (p) = i\nabla^{h_R}_X Y (p).
 \end{gather*}
 Let $(z_1 = x_1 + iy_1,\ldots, z_n = x_n + iy_n)$ be holomorphic normal coordinates of $(X,\nabla^h)$ near $p$ and write
 \begin{gather*}
  Y = \sum^n_{j = 1} c_j\pa_{z_j}
 \end{gather*}
 for some locally defined holomorphic functions $c_j$. Then:
 \begin{align*}
  \nabla^{h_R}_{J(X)} Y (p) &= \sum^n_{j = 1} c_j (p)\cdot \nabla^{h_R}_{J(X)}\pa_{z_j} (p) + dc_j (J(X)) (p)\cdot\pa_{z_j}\vert_p\\
  &= \sum^n_{j = 1} c_j (p)\cdot \nabla^{h_R}_{J(X)}\pa_{z_j} (p) + i\cdot dc_j (X) (p)\cdot\pa_{z_j}\vert_p\\
  &= \sum^n_{j = 1} i\cdot dc_j (X) (p)\cdot\pa_{z_j}\vert_p\\
  &= i\sum^n_{j = 1} c_j (p)\cdot \nabla^{h_R}_{X}\pa_{z_j} (p) + dc_j (X) (p)\cdot\pa_{z_j}\vert_p = i\nabla^{h_R}_X Y (p),
 \end{align*}
 where we used that $\nabla^{h_R}$ is $\Cx$-linear in both components, tensorial in the first component, and $\nabla^{h_R}_{\pa_{x_i}} \pa_{x_j} (p) = \nabla^{h_R}_{\pa_{x_i}} \pa_{y_j} (p) = \ldots = 0$ (cf. Proposition \autoref{prop:h_R_in_coordinates}).
\end{proof}

The formula we have just derived now allows us to link the (holomorphic) geodesics of $\nabla^h$ with the geodesics of $\nabla^{h_R} = \nabla^{h_I}$.

\begin{Prop}
 Let $(X,J)$ be a complex manifold with holomorphic metric $h = h_R + ih_I$  and Levi-Civita connection $\nabla^h\equiv\nabla^{h_R}\equiv\nabla^{h_I}$. Further, let $U\coloneqq [t_0, t_1] + i[s_0, s_1]$ be a domain in $\Cx$ and $\gamma:U\to X$ be a holomorphic curve. Define the curves $\gamma_s:[t_0,t_1]\to X$ and $\gamma_t: [s_0, s_1]\to X$ by $\gamma_s (t) \coloneqq \gamma (t+is) \eqqcolon \gamma_t (s)$. Then, $\gamma$ is a (holomorphic) geodesic of $\nabla^h$ iff $\gamma_s$ is a geodesic of $\nabla^{h_R} = \nabla^{h_I}$ for every $s\in[s_0, s_1]$ iff $\gamma_t$ is a geodesic of $\nabla^{h_R} = \nabla^{h_I}$ for every $t\in[t_0, t_1]$.
\end{Prop}

\begin{proof}
 This statement is mostly a consequence of the previous formula and the fact that $\nabla^{h_R}$ is complex: For $z = t+is$, we can write:
 \begin{gather*}
  \gamma^\prime (z) = \frac{1}{2}\left(\pa_t\gamma (z) - iJ(\pa_t\gamma (z))\right) = \frac{1}{2}\left(-J(\pa_s\gamma (z)) - i\pa_s\gamma (z)\right).
 \end{gather*}
 Hence:
 \begin{align*}
  \frac{\nabla^h}{dz}\gamma^\prime &= \nabla^h_{\gamma^\prime} \gamma^\prime = \nabla^{h_R}_{\gamma^\prime} \gamma^\prime = \nabla^{h_R}_{1/2 (\pa_t\gamma - iJ(\pa_t\gamma))} \gamma^\prime = \nabla^{h_R}_{\pa_t\gamma}\gamma^\prime = \frac{1}{2}\left(\nabla^{h_R}_{\pa_t\gamma} \pa_t\gamma - iJ\left(\nabla^{h_R}_{\pa_t\gamma} \pa_t\gamma\right)\right)\\
  &= \frac{1}{2}\left(\frac{\nabla^{h_R}}{dt}\frac{d\gamma_s}{dt} - iJ\left(\frac{\nabla^{h_R}}{dt}\frac{d\gamma_s}{dt}\right)\right).
 \end{align*}
 A similar expression can be found for $\gamma_t$ concluding the proof.
\end{proof}

We are now able to prove the following lemma:

\begin{Lem}
 Let $(X,J)$ be a complex manifold with holomorphic metric $h = h_R + ih_I$. Let $p\in X$ be any point. Then, holomorphic normal coordinates\linebreak $(z_1 = x_1 + iy_1,\ldots, z_n = x_n + iy_n)$ of $(X,\nabla^h)$ near $p$ give rise to normal coordinates $(x_1,\ldots, y_n)$ of $(X, \nabla^{h_R})$ near $p$.
\end{Lem}

\begin{proof}
 Combine all previous results from this section. Note that the same result is only true for $h_I$ after applying a linear transformation, since
 \begin{gather*}
  h_R\vert_p = \sum^n_{j = 1} dx^2_j\vert_p - dy^2_j\vert_p
 \end{gather*}
 is in standard form at $p$ in the coordinates $(x_1,\ldots, y_n)$, while the same is not true for
 \begin{gather*}
  h_I\vert_p = \sum^n_{j=1} dx_j\vert_p\otimes dy_j\vert_p + dy_j\vert_p\otimes dx_j\vert_p.
 \end{gather*}
\end{proof}

\end{appendix}

\newpage
\pagenumbering{Roman}
\addcontentsline{toc}{section}{References}
\markboth{}{References}
%\printbibliography[heading = bibintoc, title = {References}]
%\bibliographystyle{ieeetr}
%\bibliography{msbib}

\end{document}